\documentclass[oneside,english]{amsart}
\usepackage[T1]{fontenc}
\usepackage[latin9]{inputenc}
\usepackage{mathrsfs}
\usepackage{mathtools}
\usepackage{amsbsy}
\usepackage{amstext}
\usepackage{amsthm}
\usepackage{amssymb}

\makeatletter
\numberwithin{equation}{section}
\numberwithin{figure}{section}
\theoremstyle{plain}
\newtheorem{thm}{\protect\theoremname}[section]
  \theoremstyle{remark}
  \newtheorem{rem}[thm]{\protect\remarkname}
  \theoremstyle{plain}
  \newtheorem{lem}[thm]{\protect\lemmaname}
  \theoremstyle{definition}
  \newtheorem{defn}[thm]{\protect\definitionname}
  \theoremstyle{plain}
  \newtheorem{prop}[thm]{\protect\propositionname}
  \theoremstyle{plain}
  \newtheorem{cor}[thm]{\protect\corollaryname}

\makeatother

\usepackage{babel}
  \providecommand{\corollaryname}{Corollary}
  \providecommand{\definitionname}{Definition}
  \providecommand{\lemmaname}{Lemma}
  \providecommand{\propositionname}{Proposition}
  \providecommand{\remarkname}{Remark}
\providecommand{\theoremname}{Theorem}

\begin{document}

\title[Asymptotic Behavior and Stability of MCF with a Conical
End]{Asymptotic Behavior and Stability of Mean Curvature Flow with a Conical
End}

\author{Siao-Hao Guo}
\begin{abstract}
If the initial hypersurface of an immortal mean curvature flow is
asymptotic to a regular cone whose entropy is small, the flow will
become asymptotically self-expanding. Moreover, the expander that
gives rise to the limiting flow is asymptotically stable as an equilibrium
solution of the normalized mean curvature flow.
\end{abstract}

\maketitle

\section{Introduction}

By a \textbf{mean curvature flow} (MCF) we mean a one-parameter family
of smooth, properly embedded, and oriented hypersurfaces $\left\{ \Sigma_{t}\right\} _{0\leq t\leq T}$
in an open connected set $U\subset\mathbb{R}^{n+1}$ so that near
every point of the flow we can find a local parametrization for which
\begin{equation}
\partial_{t}X=\vec{H}_{\Sigma_{t}},\label{mean curvature flow}
\end{equation}
where $X$ is the position vector and $\vec{H}_{\Sigma_{t}}$ is the
mean curvature vector of $\Sigma_{t}$. 

In \cite{CM}, Colding and Minicozzi introduced a critical notion
for hypersurfaces called the \textbf{entropy}. Given a complete hypersurface
$\Sigma$ in $\mathbb{R}^{n+1}$ satisfying
\begin{equation}
\sup_{P\in\mathbb{R}^{n+1},\,r>0}\frac{\mathcal{H}^{n}\left(\Sigma\cap B_{r}\left(P\right)\right)}{\boldsymbol{\omega}_{n}r^{n}}<\infty,\label{Euclidean area growth}
\end{equation}
where $\mathcal{H}^{n}$ is the $n$ dimensional Hausdorff measure,
$B_{r}\left(P\right)$ denotes the ball in $\mathbb{R}^{n+1}$ of
radius $r$ and centered at $P$, and $\boldsymbol{\omega}_{n}=\mathcal{H}^{n}\left(B_{1}\left(O\right)\right)$,
its various Gaussian areas 
\[
F_{P,t}\left(\Sigma\right)=\int_{\Sigma}\frac{1}{\left(4\pi t\right)^{\frac{n}{2}}}e^{-\frac{\left|X-P\right|^{2}}{4t}}d\mathcal{H}^{n}\left(X\right)
\]
are finite and uniformly bounded for every $P\in\mathbb{R}^{n+1}$
and $t>0$. The entropy of $\Sigma$ is then defined as 
\[
E\left[\Sigma\right]=\sup_{P\in\mathbb{R}^{n+1},\,t>0}F_{P,t}\left(\Sigma\right).
\]
This quantity is invariant under rotation, translation, and dilation
of hypersurfaces, so it can be used to measure the ``complexity''
of a hypersurface about different points and on varying scales. It
is always bounded from below by one, which is the entropy of hyperplanes.
As an illustration, the entropy of an $L$-Lipschitz graph over a
hyperplane is bounded from above by $\sqrt{1+L^{2}}$ (see Remark
\ref{entropy of graph}). Returning to MCF, the entropy can serve
as a Liapunov function owing to Huisken's monotonicity formula (cf.
\cite{H}). To be more precise, if $\left\{ \Sigma_{t}\right\} $
is a MCF in $\mathbb{R}^{n+1}$ whose initial hypersurface satisfies
condition (\ref{Euclidean area growth}), the entropy $E\left[\Sigma_{t}\right]$
is non-increasing with time. More investigations on the entropy can
be found, for instance, in \cite{BW1}, \cite{BW2}, \cite{CIMW},
\cite{KZ}, and \cite{Z}.

In this paper we would like to explore the asymptotic behavior of
an immortal MCF. Recall that by \cite{EH1} and \cite{S}, if a MCF
starts with a Lipschitz graph (over a hyperplane) that is asymptotic
to a regular cone at infinity, it is smooth at all later times and
will become asymptotically self-expanding. More specifically, after
a proper rescaling (see (\ref{ABM: rescaling})), the flow will converge
to a smooth, properly embedded, and oriented hypersurface $\Gamma$
in $\mathbb{R}^{n+1}$ called an \textbf{expander}. Such a hypersurface
is characterized by the property that 
\begin{equation}
\left\{ \Gamma_{t}=\sqrt{t}\,\Gamma\right\} _{0<t<\infty}\label{self-expanding}
\end{equation}
is a MCF in $\mathbb{R}^{n+1}$, which is so-called \textbf{self-expanding}.
From Eqs. (\ref{mean curvature flow}) and (\ref{self-expanding})
we can infer that it satisfies
\begin{equation}
\vec{H}_{\Gamma}-\frac{1}{2}X^{\perp}=0,\label{ABM: expander}
\end{equation}
where the notation $\perp$ means the projection to the normal space
of $\Gamma$. Readers are referred to, for example, \cite{BW3}, \cite{BW4},
\cite{BW5}, and \cite{D} for more details of expanders. One of the
key ingredients of establishing the self-similarity (of the asymptotic
behavior) is based on the uniqueness of graphical MCF starting out
from a cone. However, such a uniqueness property does not hold in
general for non-graphical cases (cf. \cite{AIC}, \cite{BW4}). Nevertheless,
we have the following result.
\begin{thm}
\label{asymptotic behavior of MCF}Given $0<\kappa<\frac{1}{\sqrt{2}}$,
there exists a constant $\epsilon>0$ depending on $n$ and $\kappa$
with the following property. 

Suppose that $\left\{ \Sigma_{t}\right\} _{0\leq t<\infty}$ is a
MCF in $\mathbb{R}^{n+1}$ so that $\Sigma_{0}$ is asymptotic to
a regular cone $\mathcal{C}$ at infinity (see Definition \ref{asymptotically conical})
with $E\left[\mathcal{C}\right]<1+\epsilon$. Then there exists a
time $T>0$ so that 
\begin{equation}
\sup_{t\geq T}\,\sqrt{t}\left\Vert A_{\Sigma_{t}}\right\Vert _{L^{\infty}}\leq\kappa,\label{ABM: curvature estimate}
\end{equation}
where $A_{\Sigma_{t}}$ is the second fundamental form of $\Sigma_{t}$.
Furthermore, there exists an expander $\Gamma$, which is asymptotic
to $\mathcal{C}$ at infinity, so that 
\begin{equation}
\frac{1}{\sqrt{t}}\Sigma_{t}\overset{C^{\infty}}{\longrightarrow}\Gamma\quad\textrm{in}\,\,\,\mathbb{R}^{n+1}\;\,\,\textrm{as}\,\,t\rightarrow\infty\label{ABM: rescaling}
\end{equation}
in the sense that there exists a time $\mathfrak{T}>0$ after which
$\frac{1}{\sqrt{t}}\Sigma_{t}$ is a normal graph of $u_{t}$ over
$\Gamma$ with $u_{t}\overset{C^{\infty}}{\longrightarrow}0$ as $t\rightarrow\infty$. 
\end{thm}

The rescaled flow in (\ref{ABM: rescaling}) is known as a \textbf{normalized
mean curvature flow} (NMCF) (cf. \cite{EH1}). By making the following
change in the time variable:
\begin{equation}
\left\{ \hat{\Sigma}_{s}=\left.\frac{1}{\sqrt{t}}\Sigma_{t}\right|_{t=e^{s}}\right\} _{0\leq s<\infty},\label{ABM: rescaled MCF}
\end{equation}
the resulting flow satisfies the equation 
\begin{equation}
\left(\partial_{s}X\right)^{\perp}=\vec{H}_{\hat{\Sigma}_{s}}-\frac{1}{2}X^{\perp}.\label{ABM: NMCF}
\end{equation}
Comparing Eqs. (\ref{ABM: expander}) and (\ref{ABM: NMCF}), we can
see that an expander is indeed an equilibrium solution of NMCF. 

There is one sufficient condition that can fulfill the hypotheses
in Theorem \ref{asymptotic behavior of MCF}. Namely, let $\Sigma_{0}$
be a smooth, properly embedded, and oriented hypersurface in $\mathbb{R}^{n+1}$
that is asymptotic to a regular cone $\mathcal{C}$ at infinity and
satisfies 
\[
E\left[\Sigma_{0}\right]\leq1+\epsilon,
\]
where $\epsilon$ is the constant in Theorem \ref{asymptotic behavior of MCF}.
Note that the asymptotic condition implies $\left\Vert A_{\Sigma_{0}}\right\Vert _{L^{\infty}}<\infty$
and $E\left[\mathcal{C}\right]\leq E\left[\Sigma_{0}\right]$. Then
the corresponding MCF $\left\{ \Sigma_{t}\right\} _{t\geq0}$ is immortal
by virtue of White's regularity theorem (see Theorems \ref{White's regularity}
and \ref{refinement of White's}) and the short time existence theorem
(cf. Section 4 in \cite{EH2}). Consequently, Theorem \ref{asymptotic behavior of MCF}
is applicable to $\left\{ \Sigma_{t}\right\} $. Nonetheless, there
are, presumably, examples where the initial hypersurface has large
entropy but still persist for all time. To such cases Theorem \ref{asymptotic behavior of MCF}
still applies as long as the initial hypersurfaces are asymptotic
to a regular cone with small entropy. On the other hand, in view of
(\ref{ABM: rescaling}) and the fact that $E\left[\Gamma\right]=E\left[\mathcal{C}\right]$
(cf. Lemma 3.5 in \cite{BW3}), one might conjecture that the entropy
of the flow eventually has to go down to that of the cone. 

The curvature estimate (\ref{ABM: curvature estimate}) is resulting
from Theorem \ref{decay rate of curvature} and is equivalent to 
\begin{equation}
\sup_{s\geq\ln T}\,\left\Vert A_{\hat{\Sigma}_{s}}\right\Vert _{L^{\infty}}\leq\kappa\label{ABM: normalize curvature}
\end{equation}
(see (\ref{ABM: rescaled MCF})). The significance of the condition
$\kappa<\frac{1}{\sqrt{2}}$ can be seen from the following observation.
The linearized operator of the right side of Eq. (\ref{ABM: NMCF})
is given by 
\[
\mathcal{L}_{\hat{\Sigma}_{s}}=\triangle_{\hat{\Sigma}_{s}}+\frac{1}{2}\,X\cdot\nabla_{\hat{\Sigma}_{s}}+\left|A_{\hat{\Sigma}_{s}}\right|^{2}-\frac{1}{2},
\]
where $\nabla_{\hat{\Sigma}_{s}}$ is the Levi-Civita connection on
$\hat{\Sigma}_{s}$ and $\triangle_{\hat{\Sigma}_{s}}$ is the Laplace-Beltrami
operator. The operator $-\mathcal{L}_{\hat{\Sigma}_{s}}$ is self-adjoint
with respect to the inner product 
\[
\left\langle v_{1},v_{2}\right\rangle =\int_{\hat{\Sigma}_{s}}v_{1}v_{2}\,e^{\frac{\left|X\right|^{2}}{4}}d\mathcal{H}^{n}\left(X\right);
\]
the associated quadratic form is 
\begin{equation}
\left\langle -\mathcal{L}_{\hat{\Sigma}_{s}}v,v\right\rangle =\int_{\hat{\Sigma}_{s}}\left[\left|\nabla_{\hat{\Sigma}_{s}}v\right|^{2}+\left(\frac{1}{2}-\left|A_{\hat{\Sigma}_{s}}\right|^{2}\right)v^{2}\right]e^{\frac{\left|X\right|^{2}}{4}}d\mathcal{H}^{n}\left(X\right)\label{ABM: quadratic form}
\end{equation}
for $v\in C_{c}^{\infty}\left(\hat{\Sigma}_{s}\right)$ (cf. Section
2.7 in \cite{BW5}). Thus, by (\ref{ABM: normalize curvature}) and
(\ref{ABM: quadratic form}) we have $-\mathcal{L}_{\hat{\Sigma}_{s}}>0$
if $\kappa<\frac{1}{\sqrt{2}}$.

A crucial property that comes into play (and is closely related to
the aforementioned stability) is called the ``approaching property''
(see Theorem \ref{approaching of NMCF}). Roughly speaking, let $\tilde{\Sigma}_{T}$
be a small perturbation of $\Sigma_{T}$ given in Theorem \ref{asymptotic behavior of MCF},
then the corresponding MCF $\left\{ \tilde{\Sigma}_{t}\right\} _{t\geq T}$
is immortal; what's more, the two NMCFs $\left\{ \frac{1}{\sqrt{t}}\tilde{\Sigma}_{t}\right\} $
and $\left\{ \frac{1}{\sqrt{t}}\Sigma_{t}\right\} $ will approach
each other as $t\rightarrow\infty$. A special case is when $\left\{ \Sigma_{t}\right\} $
itself is self-expanding, that is, $\Sigma_{t}=\sqrt{t}\,\Gamma$
for some expander $\Gamma$. This leads to the following theorem.
\begin{thm}
\label{stability of expanders}Let $\epsilon$ and $\kappa$ be the
constants in Theorem \ref{asymptotic behavior of MCF}.

If $\Gamma$ is an expander that is asymptotic to a regular cone $\mathcal{C}$
at infinity with $E\left[\mathcal{C}\right]<1+\epsilon$. Then we
have $\left\Vert A_{\Gamma}\right\Vert _{L^{\infty}}\leq\kappa$,
and further, given $\Lambda>0$, there exists a constant $\delta>0$
depending on $n$, $\kappa$, and $\Lambda$ with the following property.

Let $\tilde{\Sigma}$ be a normal graph of $v\in C_{c}^{\infty}\left(\Gamma\right)$
over $\Gamma$ with
\[
\left\Vert \nabla_{\Gamma}^{2}v\right\Vert _{L^{\infty}}\leq\Lambda,\quad\left\Vert \nabla_{\Gamma}v\right\Vert _{L^{\infty}}+\left\Vert v\right\Vert _{L^{\infty}}\leq\delta.
\]
Then the MCF $\left\{ \tilde{\Sigma}_{t}\right\} _{t\geq1}$ starting
out from $\tilde{\Sigma}$ has long time existence; moreover, 
\[
\frac{1}{\sqrt{t}}\tilde{\Sigma}_{t}\overset{C^{\infty}}{\longrightarrow}\Gamma\quad\textrm{in}\,\,\,\mathbb{R}^{n+1}\;\,\,\textrm{as}\,\,t\rightarrow\infty.
\]
To be specific, for every $t>1$, $\frac{1}{\sqrt{t}}\tilde{\Sigma}_{t}$
is a normal graph of $w_{t}$ over $\Gamma$ with $w_{t}\overset{C^{\infty}}{\longrightarrow}0$
as $t\rightarrow\infty$. 

As a corollary, $\Gamma$ is isolated in the $C^{2}$ topology among
all expanders with the conical end $\mathcal{C}$. 
\end{thm}

Recently a result related to the corollary in Theorem \ref{stability of expanders}
has been found in \cite{BW5}, where Bernstein and Wang show that
any two expanders with the same conical end are isotopic, provided
that the entropy of the tangent cone is less than a constant determined
by the entropies of cylinders and non-flat minimal cones.

Lastly, the paper is organized as follows. The main theorems, Theorems
\ref{asymptotic behavior of MCF} and \ref{stability of expanders},
will be proved in Section \ref{Asymptotic self-similarity}. In Sections
\ref{White's regularity thm} and \ref{Asymptotically conical MCF}
we will review White's regularity theorem and the pseudolocality theorem
for MCF, respectively, so as to derive the curvature estimate (\ref{ABM: curvature estimate})
(see Theorem \ref{decay rate of curvature}). Also, in Section \ref{Asymptotically conical MCF}
the preservation of the asymptotically conical property along MCF
will be proved. The derivation of the approaching property (see Theorem
\ref{approaching of NMCF}) is in Section \ref{Approaching property}. 

\section*{Acknowledgement}

The author would like to thank Professor Peter Sternberg for his encouragement,
patience, and support. His useful suggestions on the paper were highly
appreciated. The author would also like to thank Professor Nam Quang
Le for his very helpful comments. 

\section{White's Regularity Theorem\label{White's regularity thm}}

We begin this section by reviewing White's regularity theorem for
MCF in the form of Theorem \ref{White's regularity}. Then a refined
version is given in Theorem \ref{refinement of White's}, whose corollary,
Lemma \ref{temporal curvature estimate}, is vital in the derivation
of the curvature estimate in Theorem \ref{decay rate of curvature}.

The following statement of White's regularity theorem, which is modified
from \cite{Wh} and Theorem 5.6 in \cite{E}, is written in such a
way that best fits what we need in this paper. For the sake of completeness,
a proof is also included.
\begin{thm}
\label{White's regularity}There exist constants $1<\lambda<2$ and
$M,K\geq1$ depending on $n$ with the following property.

Suppose that $\left\{ \Sigma_{t}\right\} _{0\leq t\leq1}$ is a MCF
in $B_{M+1}\left(O\right)$ satisfying
\[
\sup_{P\in B_{1}\left(O\right),\frac{1}{2}<t\leq1}\,F_{P,t}\left(\Sigma_{0}\cap B_{M\sqrt{t}}\left(P\right)\right)\leq\lambda.
\]
Then 
\[
\sup_{\frac{1}{2}<t\leq1}\,\,\sup_{P\in\Sigma_{t}\cap B_{1}\left(O\right)}\,r\left(P,t\right)\left|A_{\Sigma_{t}}\left(P\right)\right|\leq K,
\]
where $r\left(P,t\right)=\sup\left\{ r>0:B_{r}\left(P\right)\times\left(t-r^{2},t\right]\subset B_{1}\left(O\right)\times\left(\frac{1}{2},1\right]\right\} $.
In particular, we have 
\[
\sup_{\frac{3}{4}<t\leq1}\left\Vert A_{\Sigma_{t}}\right\Vert _{L^{\infty}\left(B_{\frac{1}{2}}\left(O\right)\right)}\leq2K.
\]
\end{thm}

\begin{proof}
Suppose to the contrary that there exist sequences of constants $\lambda_{i}\searrow1$,
$M_{i}\nearrow\infty$, $K_{i}\nearrow\infty$, and a sequence of
MCF $\left\{ \Sigma_{t}^{i}\right\} _{0\leq t\leq1}$ in $B_{M_{i}+1}\left(O\right)$
so that 
\[
\sup_{P\in B_{1}\left(O\right),\frac{1}{2}<t\leq1}\,F_{P,t}\left(\Sigma_{0}^{i}\cap B_{M_{i}\sqrt{t}}\left(P\right)\right)\leq\lambda_{i},
\]
\[
\sup_{\frac{1}{2}<t\leq1}\,\,\sup_{P\in\Sigma_{t}^{i}\cap B_{1}\left(O\right)}r\left(P,t\right)\left|A_{\Sigma_{t}^{i}}\left(P\right)\right|\geq K_{i}.
\]
For each $i$, choose $\frac{1}{2}<t_{i}\leq1$ and $P_{i}\in\Sigma_{t_{i}}^{i}\cap B_{1}\left(O\right)$
so that 
\[
\sup_{\frac{1}{2}<t\leq1}\,\,\sup_{P\in\Sigma_{t}^{i}\cap B_{1}\left(O\right)}r\left(P,t\right)\left|A_{\Sigma_{t}^{i}}\left(P\right)\right|=r_{i}A_{i},
\]
where $r_{i}=r\left(P_{i},t_{i}\right)$ and $A_{i}=\left|A_{\Sigma_{t_{i}}^{i}}\left(P_{i}\right)\right|$.
Note that 
\[
\sup_{t_{i}-\frac{r_{i}^{2}}{4}<t\leq t_{i}}\left\Vert A_{\Sigma_{t}^{i}}\right\Vert _{L^{\infty}\left(B_{\frac{r_{i}}{2}}\left(P_{i}\right)\right)}\leq2A_{i}
\]
and that $r_{i}A_{i}\geq K_{i}\rightarrow\infty$. In particular,
$A_{i}\geq\frac{K_{i}}{r_{i}}\rightarrow\infty$. Moreover, the local
monotonicity formula for MCF (cf. Chapter 4 in \cite{E}) gives
\begin{equation}
\left(1-\frac{2n}{M_{i}^{2}}\right)^{3}\leq\int_{\Sigma_{t}^{i}}\left(1-\frac{\left|X-P_{i}\right|^{2}+2nt}{M_{i}^{2}t_{i}}\right)_{+}^{3}\frac{1}{\left(4\pi\left(t_{i}-t\right)\right)^{\frac{n}{2}}}e^{-\frac{\left|X-P_{i}\right|^{2}}{4\left(t_{i}-t\right)}}d\mathcal{H}^{n}\left(X\right)\label{W: Gaussian}
\end{equation}
\[
\leq\int_{\Sigma_{0}^{i}}\left(1-\frac{\left|X-P_{i}\right|^{2}}{M_{i}^{2}t_{i}}\right)_{+}^{3}\frac{1}{\left(4\pi t_{i}\right)^{\frac{n}{2}}}e^{-\frac{\left|X-P_{i}\right|^{2}}{4t_{i}}}d\mathcal{H}^{n}\left(X\right)
\]
 
\[
\leq F_{P_{i},t_{i}}\left(\Sigma_{0}^{i}\cap B_{M_{i}\sqrt{t_{i}}}\left(P_{i}\right)\right)\leq\lambda_{i}.
\]
for $t_{i}-r_{i}^{2}\leq t<t_{i}$. Note that the first inequality
in (\ref{W: Gaussian}) comes from letting $t\nearrow t_{i}$ in the
integral on its right side. 

Let 
\[
\tilde{\Sigma}_{\tau}^{i}=A_{i}\left(\Sigma_{t_{i}+A_{i}^{-2}\tau}^{i}-P_{i}\right).
\]
Then $\left\{ \tilde{\Sigma}_{\tau}^{i}\right\} _{-\left(A_{i}r_{i}\right)^{2}\leq\tau\leq0}$
is a MCF in $B_{A_{i}r_{i}}\left(O\right)$ satisfying 
\[
\sup_{-\frac{1}{4}\left(A_{i}r_{i}\right)^{2}\leq\tau\leq0}\left\Vert A_{\tilde{\Sigma}_{\tau}^{i}}\right\Vert _{L^{\infty}\left(B_{\frac{1}{2}A_{i}r_{i}}\left(O\right)\right)}\leq2,\quad\left|A_{\tilde{\Sigma}_{0}^{i}}\left(O\right)\right|=1,
\]
and 
\[
\left(1-\frac{2n}{M_{i}^{2}}\right)^{3}\leq\int_{\tilde{\Sigma}_{\tau}^{i}}\left(1-\frac{\left|A_{i}^{-1}Y\right|^{2}+2n\left(t_{i}+A_{i}^{-2}\tau\right)}{M_{i}^{2}t_{i}}\right)_{+}^{3}\frac{1}{\left(4\pi\left(-\tau\right)\right)^{\frac{n}{2}}}e^{-\frac{\left|Y\right|^{2}}{4\left(-\tau\right)}}d\mathcal{H}^{n}\left(Y\right)\leq\lambda_{i}
\]
for $-\left(A_{i}r_{i}\right)^{2}\leq\tau<0$. It follows from the
smooth compactness theorem for MCF that, after passing to a subsequence,
\[
\left\{ \tilde{\Sigma}_{\tau}^{i}\right\} _{-\left(A_{i}r_{i}\right)^{2}\leq\tau\leq0}\,\stackrel{C_{loc}^{\infty}}{\longrightarrow}\,\left\{ \tilde{\Sigma}_{\tau}\right\} _{-\infty<\tau\leq0},
\]
where $\left\{ \tilde{\Sigma}_{\tau}\right\} _{-\infty<\tau\leq0}$
is a MCF in $\mathbb{R}^{n+1}$. Note that the limiting flow satisfies
\begin{equation}
\sup_{-\infty<\tau\leq0}\left\Vert A_{\tilde{\Sigma}_{\tau}}\right\Vert _{L^{\infty}}\leq2,\quad\left|A_{\tilde{\Sigma}_{0}}\left(O\right)\right|=1,\label{W: curvature bound}
\end{equation}
\begin{equation}
F_{O,-\tau}\left(\tilde{\Sigma}_{\tau}\right)=1\quad\,\forall\,\,\tau<0.\label{W: constancy}
\end{equation}
By Huisken's monotonicity formula (cf. \cite{H}) and condition (\ref{W: constancy}),
$\left\{ \tilde{\Sigma}_{\tau}\right\} $ must satisfy
\begin{equation}
\left(-\tau\right)\vec{H}_{\tilde{\Sigma}_{\tau}}+\frac{1}{2}Y^{\perp}=0\label{W: self-shrinking}
\end{equation}
for every $\tau<0$, where $Y$ is the position vector. Letting $\tau\nearrow0$
in Eq. (\ref{W: self-shrinking}) and using the uniform boundedness
of the curvature up to $\tau=0$, namely (\ref{W: curvature bound}),
we infer that $\tilde{\Sigma}_{0}$ must be a smooth cone (i.e. a
hyperplane), which contradicts the condition that $\left|A_{\tilde{\Sigma}_{0}}\left(O\right)\right|=1$
in (\ref{W: curvature bound}).
\end{proof}
The next theorem is an improvement of Theorem \ref{White's regularity}
in the way that the constant $K$ in Theorem \ref{White's regularity}
can be chosen arbitrarily small so long as the constant $\lambda$
is sufficiently close to one. The proof uses essentially the same
argument as in the preceding proof.
\begin{thm}
\label{refinement of White's}Given $\kappa>0$, there exist constants
$0<\epsilon<1$ and $M\geq1$ depending on $n$ and $\kappa$ with
the following property.

Suppose that $\left\{ \Sigma_{t}\right\} _{0\leq t\leq1}$ is a MCF
in $B_{2M}\left(O\right)$ such that 
\[
\sup_{P\in B_{M}\left(O\right),\frac{1}{2}<t\leq1}\,F_{P,t}\left(\Sigma_{0}\cap B_{M\sqrt{t}}\left(P\right)\right)\leq1+\epsilon.
\]
Then 
\[
\sup_{\frac{3}{4}\leq t\leq1}\,\left\Vert A_{\Sigma_{t}}\right\Vert _{L^{\infty}\left(B_{\frac{1}{2}}\left(O\right)\right)}\leq\kappa.
\]
\end{thm}

\begin{proof}
Suppose to the contrary that there exist sequences $\epsilon_{i}\searrow0$,
$M_{i}\nearrow\infty$, and a sequence of MCF $\left\{ \Sigma_{t}^{i}\right\} _{0\leq t\leq1}$
in $B_{2M_{i}}\left(O\right)$ so that 
\[
\sup_{P\in B_{M_{i}}\left(O\right),\frac{1}{2}<t\leq1}\,F_{P,t}\left(\Sigma_{0}^{i}\cap B_{M_{i}\sqrt{t}}\left(P\right)\right)\leq1+\epsilon_{i},
\]
\[
\sup_{\frac{3}{4}\leq t\leq1}\,\left\Vert A_{\Sigma_{t}^{i}}\right\Vert _{L^{\infty}\left(B_{\frac{1}{2}}\left(O\right)\right)}>\kappa.
\]
By Theorem \ref{White's regularity}, there exists a constant $K>1$
(depending on $n$) so that 
\[
\sup_{\frac{5}{8}\leq t\leq1}\,\left\Vert A_{\Sigma_{t}^{i}}\right\Vert _{L^{\infty}\left(B_{M_{i}-1}\left(O\right)\right)}\leq K.
\]
For each $i$, choose $\frac{3}{4}\leq t_{i}\leq1$ and $P_{i}\in\Sigma_{t_{i}}^{i}\cap B_{\frac{1}{2}}\left(O\right)$
so that $\left|A_{\Sigma_{t_{i}}^{i}}\left(P_{i}\right)\right|\geq\kappa$.
The local monotonicity formula for MCF implies that
\[
\left(1-\frac{2n}{M_{i}^{2}}\right)^{3}\leq\int_{\Sigma_{t}^{i}}\left(1-\frac{\left|X-P_{i}\right|^{2}+2nt}{M_{i}^{2}t_{i}}\right)_{+}^{3}\frac{1}{\left(4\pi\left(t_{i}-t\right)\right)^{\frac{n}{2}}}e^{-\frac{\left|X-P_{i}\right|^{2}}{4\left(t_{i}-t\right)}}d\mathcal{H}^{n}\left(X\right)
\]
\[
\leq\int_{\Sigma_{0}^{i}}\left(1-\frac{\left|X-P_{i}\right|^{2}}{M_{i}^{2}t_{i}}\right)_{+}^{3}\frac{1}{\left(4\pi t_{i}\right)^{\frac{n}{2}}}e^{-\frac{\left|X-P_{i}\right|^{2}}{4t_{i}}}d\mathcal{H}^{n}\left(X\right)
\]
 
\[
\leq F_{P_{i},t_{i}}\left(\Sigma_{0}^{i}\cap B_{M_{i}\sqrt{t_{i}}}\left(P_{i}\right)\right)\leq1+\epsilon_{i}.
\]
for $0\leq t<t_{i}$. 

Let 
\[
\tilde{\Sigma}_{\tau}^{i}=\Sigma_{t_{i}+\tau}^{i}-P_{i}.
\]
Then $\left\{ \tilde{\Sigma}_{\tau}^{i}\right\} _{-t_{i}\leq\tau\leq0}$
is a MCF in $B_{2M_{i}-\frac{1}{2}}\left(O\right)$ satisfying 
\[
\sup_{\frac{5}{8}-t_{i}\leq\tau\leq0}\left\Vert A_{\tilde{\Sigma}_{\tau}^{i}}\right\Vert _{L^{\infty}\left(B_{M_{i}-\frac{3}{2}}\left(O\right)\right)}\leq K,\quad\left|A_{\tilde{\Sigma}_{0}^{i}}\left(O\right)\right|\geq\kappa,
\]
and 
\[
\left(1-\frac{2n}{M_{i}^{2}}\right)^{3}\leq\int_{\tilde{\Sigma}_{\tau}^{i}}\left(1-\frac{\left|Y\right|^{2}+2n\left(t_{i}+\tau\right)}{M_{i}^{2}t_{i}}\right)_{+}^{3}\frac{1}{\left(4\pi\left(-\tau\right)\right)^{\frac{n}{2}}}e^{-\frac{\left|Y\right|^{2}}{4\left(-\tau\right)}}d\mathcal{H}^{n}\left(Y\right)\leq1+\epsilon_{i}
\]
for $-t_{i}\leq\tau<0$. Note that $\frac{5}{8}-t_{i}\leq-\frac{1}{8}$
since $t_{i}\geq\frac{3}{4}$. It follows from the smooth compactness
theorem for MCF that, after passing to a subsequence,
\[
\left\{ \tilde{\Sigma}_{\tau}^{i}\right\} _{-\frac{1}{8}<\tau\leq0}\,\stackrel{C_{loc}^{\infty}}{\longrightarrow}\,\left\{ \tilde{\Sigma}_{\tau}\right\} _{-\frac{1}{8}<\tau\leq0},
\]
where $\left\{ \tilde{\Sigma}_{\tau}\right\} _{-\frac{1}{8}<\tau\leq0}$
is a MCF in $\mathbb{R}^{n+1}$. Note that the limiting flow satisfies
\begin{equation}
\sup_{-\frac{1}{8}<\tau\leq0}\left\Vert A_{\tilde{\Sigma}_{\tau}}\right\Vert _{L^{\infty}}\leq K,\quad\left|A_{\tilde{\Sigma}_{0}}\left(O\right)\right|\geq\kappa,\label{RW: curvature bound}
\end{equation}
\[
F_{O,-\tau}\left(\tilde{\Sigma}_{\tau}\right)=1\quad\,\,\,\forall\,-\frac{1}{8}<\tau<0.
\]
By Huisken's monotonicity formula, the last condition implies that
$\left\{ \tilde{\Sigma}_{\tau}\right\} $ satisfies
\begin{equation}
\left(-\tau\right)\vec{H}_{\tilde{\Sigma}_{\tau}}+\frac{1}{2}Y^{\perp}=0\label{RW: self-shrinking}
\end{equation}
for $-\frac{1}{8}<\tau<0$, where $Y$ is the position vector. Letting
$\tau\nearrow0$ in Eq. (\ref{RW: self-shrinking}) and using the
uniform boundedness of the curvature up to $\tau=0$ in (\ref{RW: curvature bound}),
we deduce that $\tilde{\Sigma}_{0}$ must be a smooth cone (i.e. a
hyperplane), which contradicts the condition that $\left|A_{\tilde{\Sigma}_{0}}\left(O\right)\right|\geq\kappa$
in (\ref{RW: curvature bound}).
\end{proof}
Below we would like to provide a sufficient condition (see Lemma \ref{small Lipschitz graph })
under which the hypothesis in Theorem \ref{refinement of White's}
holds. To this end, let us first make the following observation concerning
the Gaussian areas of an entire Lipschitz graph over a hyperplane.
\begin{rem}
\label{entropy of graph}Let $\Sigma$ be an $L$-Lipschitz graph
over a hyperplane, say 
\[
\Sigma=\left\{ X=\left(x,u\left(x\right)\right):x\in\mathbb{R}^{n}\right\} 
\]
 with $\left\Vert \partial_{x}u\right\Vert _{L^{\infty}}\leq L$.
Given $P=\left(p,\mathfrak{p}\right)\in\mathbb{R}^{n}\times\mathbb{R}$
and $t>0$, we have
\[
F_{P,t}\left(\Sigma\right)=\int_{\mathbb{R}^{n}}\frac{1}{\left(4\pi t\right)^{\frac{n}{2}}}e^{-\frac{\left|x-p\right|^{2}+\left(u\left(x\right)-\mathfrak{p}\right)^{2}}{4t}}\sqrt{1+\left|\partial_{x}u\left(x\right)\right|^{2}}\,dx
\]
\[
\leq\sqrt{1+L^{2}}\int_{\mathbb{R}^{n}}\frac{1}{\left(4\pi t\right)^{\frac{n}{2}}}e^{-\frac{\left|x-p\right|^{2}}{4t}}dx\,=\sqrt{1+L^{2}}.
\]
Consequently, $E\left[\Sigma\right]\leq\sqrt{1+L^{2}}$.
\end{rem}

The following lemma ensures that a locally small Lipschitz graph has
small localized Gaussian areas.
\begin{lem}
\label{small Lipschitz graph }Given $\epsilon>0$, there exists a
constant $\delta>0$ depending on $n$ and $\epsilon$ with the following
property. 

If $\Sigma$ is a $\delta$-Lipschitz graph over a hyperplane in $B_{3M}\left(O\right)$,
where $M\geq1$ is a constant, then we have
\[
\sup_{P\in B_{M}\left(O\right),\,0<t\leq1}\,F_{P,t}\left(\Sigma\cap B_{M\sqrt{t}}\left(P\right)\right)\leq1+\epsilon.
\]
\end{lem}

\begin{proof}
Given $\epsilon>0$, let $\Sigma$ be as stated in the lemma, where
$0<\delta\ll1$ is a constant to be determined (which will depend
only on $n$, $\epsilon$). 

Given $P\in B_{M}\left(O\right)$ and $0<t\leq1$, let us assume that
$\Sigma\cap B_{M\sqrt{t}}\left(P\right)\neq\emptyset$; otherwise
$F_{P,t}\left(\Sigma\cap B_{M\sqrt{t}}\left(P\right)\right)=0$. Since
$\Sigma$ is a $\delta$-Lipschitz graph in $B_{3M}\left(O\right)$,
which strictly contains $B_{2M}\left(O\right)\supset B_{M\sqrt{t}}\left(P\right)$,
we can find a complete hypersurface $\tilde{\Sigma}$ in $\mathbb{R}^{n+1}$
that extends $\Sigma\cap B_{M\sqrt{t}}\left(P\right)$ and is a $2\delta$-Lipschitz
graph, provided that $\delta\ll1$ (depending on $n$). It then follows
from Remark \ref{entropy of graph} that 
\[
F_{P,t}\left(\Sigma\cap B_{M\sqrt{t}}\left(P\right)\right)\leq F_{P,t}\left(\tilde{\Sigma}\right)\leq\sqrt{1+4\delta^{2}}.
\]
Whence the lemma is proved if $0<\delta\leq\frac{1}{2}\sqrt{\epsilon\left(2+\epsilon\right)}$.
\end{proof}
Let us conclude this section with the following lemma, which is a
corollary of Theorem \ref{refinement of White's} and plays a key
role in deriving the curvature estimate in Theorem \ref{decay rate of curvature}.
\begin{lem}
\label{temporal curvature estimate}Given $\kappa>0$, let $\left\{ \Sigma_{t}\right\} _{0\leq t<\infty}$
be a MCF in $\mathbb{R}^{n+1}$ satisfying
\[
\sup_{t\geq\frac{1}{2}T}\,\,\sup_{P\in B_{2\Lambda\sqrt{t}}\left(O\right)}\,F_{P,t}\left(\Sigma_{0}\right)\leq1+\epsilon
\]
for some constants $T>0$ and $\Lambda\geq\frac{M}{\sqrt{2}-1}$,
where $\epsilon$ and $M$ are the constants in Theorem \ref{refinement of White's}.
Then we have
\[
\sup_{t\geq T}\,\sqrt{t}\left\Vert A_{\Sigma_{t}}\right\Vert _{L^{\infty}\left(B_{\Lambda\sqrt{t}}\left(O\right)\right)}\leq\kappa.
\]
\end{lem}

\begin{proof}
Fix $t_{0}\geq T$ and $P_{0}\in B_{\Lambda\sqrt{t_{0}}}\left(O\right)$,
and let 
\[
\tilde{\Sigma}_{\tau}=\frac{1}{\sqrt{t_{0}}}\left(\Sigma_{t_{0}\tau}-P_{0}\right).
\]
For every $Q\in B_{M}\left(O\right)$ and $\tau\in\left[\frac{1}{2},1\right]$,
we have
\[
F_{Q,\tau}\left(\tilde{\Sigma}_{0}\right)=F_{O,1}\left(\frac{1}{\sqrt{\tau}}\left(\tilde{\Sigma}_{0}-Q\right)\right)=F_{O,1}\left(\frac{1}{\sqrt{t_{0}\tau}}\left(\Sigma_{0}-P_{0}-\sqrt{t_{0}}Q\right)\right)
\]
 
\[
=F_{P_{0}+\sqrt{t_{0}}Q,\,t_{0}\tau}\left(\Sigma_{0}\right)\leq1+\epsilon.
\]
Note that the above inequality is obtained by using the hypothesis
and the fact that 
\[
\left|P_{0}+\sqrt{t_{0}}Q\right|\leq\left(\Lambda+M\right)\sqrt{t_{0}}\leq\sqrt{2}\Lambda\sqrt{t_{0}}\leq2\Lambda\sqrt{t_{0}\tau}.
\]
Thus, by Theorem \ref{refinement of White's} we obtain
\[
\sup_{\frac{3}{4}t_{0}\leq t\leq t_{0}}\,\sqrt{t_{0}}\left\Vert A_{\Sigma_{t}}\right\Vert _{L^{\infty}\left(B_{\frac{1}{2}\sqrt{t_{0}}}\left(P_{0}\right)\right)}=\sup_{\frac{3}{4}\leq\tau\leq1}\,\left\Vert A_{\tilde{\Sigma}_{\tau}}\right\Vert _{L^{\infty}\left(B_{\frac{1}{2}}\left(O\right)\right)}\leq\kappa.
\]
\end{proof}

\section{Asymptotically Conical MCF\label{Asymptotically conical MCF}}

In this section we first review the pseudolocality theorem for MCF
(see Theorem \ref{pseudo-locality}). Then we proceed to study the
asymptotically conical property, including the definition (Definition
\ref{asymptotically conical}), the smooth estimates (Proposition
\ref{spatial curvature estimate} and Corollary \ref{spatial smooth estimates}),
and the preservation along MCF (Proposition \ref{asymptotically conical along MCF}).
Finally, we deduce the curvature estimate under the small entropy
condition of the tangent cone (Theorem \ref{decay rate of curvature}).

The pseudolocality theorem that we are going to present in Theorem
\ref{pseudo-locality} is modified from Theorem 1.4 and Corollary
1.5 in \cite{CY}. A proof is provided for the sake of completeness.
To facilitate the proof, we need the following two lemmas. It is worth
noting that Lemma \ref{Chen-Yin lemma} improves Theorem 1.4 in \cite{CY}
(in the Euclidean setting), which helps simplify the proof of Theorem
\ref{pseudo-locality}. 
\begin{lem}
\label{Chen-Yin lemma}Given $\kappa>0$, there exist constants $\delta>0$
and $M\geq1$ depending on $n$ and $\kappa$ with the following property. 

Let $\left\{ \Sigma_{t}\right\} _{0\leq t\leq T}$ be a MCF in $B_{3M}\left(O\right)$,
where $0<T\leq1$ is a constant, so that $\Sigma_{0}$ is a $\delta$-Lipschitz
graph in $B_{3M}\left(O\right)$ . Then we have
\[
\sup_{0<t\leq T}\,\sqrt{t}\left\Vert A_{\Sigma_{t}}\right\Vert _{L^{\infty}\left(B_{\frac{1}{8}}\left(O\right)\right)}\leq\kappa.
\]
\end{lem}

\begin{proof}
Given $\kappa>0$, let $0<\epsilon<1$ and $M\geq1$ be the corresponding
constants in Theorem \ref{refinement of White's}. Without loss of
generality, we may assume that $M\geq\left[8\left(1-\frac{\sqrt{3}}{2}\right)\right]^{-1}$.
Let $\left\{ \Sigma_{t}\right\} _{0\leq t\leq T}$ be as stated in
the lemma, where $0<\delta\ll1$ is a constant to be determined (which
will depend only on $n$, $\kappa$). 

By Lemma \ref{small Lipschitz graph }, if $\delta\ll1$ (depending
on $n$, $\epsilon$), we have 
\begin{equation}
\sup_{P\in B_{M}\left(O\right),\,0<t\leq T}\,F_{P,t}\left(\Sigma_{0}\cap B_{M\sqrt{t}}\left(P\right)\right)\leq1+\epsilon.\label{CYL: Gaussian}
\end{equation}
Note that if $\frac{3}{4}\leq T\leq1$, in light of Theorem \ref{refinement of White's}
and condition (\ref{CYL: Gaussian}) we have
\[
\sup_{\frac{3}{4}\leq t\leq T}\,\sqrt{t}\left\Vert A_{\Sigma_{t}}\right\Vert _{L^{\infty}\left(B_{\frac{1}{8}}\left(O\right)\right)}\leq\sup_{\frac{3}{4}T\leq t\leq T}\,\sqrt{t}\left\Vert A_{\Sigma_{t}}\right\Vert _{L^{\infty}\left(B_{\frac{1}{2}\sqrt{T}}\left(O\right)\right)}\leq\kappa.
\]
To finish the proof, it suffices to show that
\[
\sup_{0<t\leq\min\left\{ \frac{3}{4},T\right\} }\,\sqrt{t}\left\Vert A_{\Sigma_{t}}\right\Vert _{L^{\infty}\left(B_{\frac{1}{8}}\left(O\right)\right)}\leq\kappa.
\]
For this purpose, let us fix $0<t_{0}\leq\min\left\{ \frac{3}{4},T\right\} $
and $P_{0}\in B_{\frac{1}{8}}\left(O\right)$, and define
\[
\tilde{\Sigma}_{\tau}=\frac{1}{\sqrt{t_{0}}}\left(\Sigma_{t_{0}\tau}-P_{0}\right).
\]
For every $Q\in B_{M}\left(O\right)$ and $\tau\in\left[\frac{1}{2},1\right]$
we have
\[
F_{Q,\tau}\left(\tilde{\Sigma}_{0}\cap B_{M\sqrt{\tau}}\left(Q\right)\right)=F_{O,1}\left(\frac{1}{\sqrt{\tau}}\left(\tilde{\Sigma}_{0}-Q\right)\cap B_{M}\left(O\right)\right)
\]
\[
=F_{O,1}\left(\frac{1}{\sqrt{t_{0}\tau}}\left(\Sigma_{0}-P_{0}-\sqrt{t_{0}}Q\right)\cap B_{M}\left(O\right)\right)
\]
\[
=F_{P_{0}+\sqrt{t_{0}}Q,\,t_{0}\tau}\left(\Sigma_{0}\cap B_{M\sqrt{t_{0}\tau}}\left(P_{0}+\sqrt{t_{0}}Q\right)\right)\leq1+\epsilon.
\]
Note that the last inequality is obtained by using condition (\ref{CYL: Gaussian})
and the fact that 
\[
\left|P_{0}+\sqrt{t_{0}}Q\right|\leq\frac{1}{8}+\sqrt{\frac{3}{4}}M\leq M.
\]
Hence, by Theorem \ref{refinement of White's} we obtain
\[
\sup_{\frac{3}{4}t_{0}\leq t\leq t_{0}}\,\sqrt{t_{0}}\left\Vert A_{\Sigma_{t}}\right\Vert _{L^{\infty}\left(B_{\frac{1}{2}\sqrt{t_{0}}}\left(P_{0}\right)\right)}=\sup_{\frac{3}{4}\leq\tau\leq1}\,\left\Vert A_{\tilde{\Sigma}_{\tau}}\right\Vert _{L^{\infty}\left(B_{\frac{1}{2}}\left(O\right)\right)}\leq\kappa.
\]
\end{proof}
\begin{lem}
\label{lower bound on curvature}There exists a constant $0<\varsigma<1$
depending on $n$ with the following property.

Let $\left\{ \Sigma_{t}\right\} _{-T\leq t\leq0}$ be a MCF in $B_{R}\left(O\right)$,
where $0<T\leq\varsigma$ and $R\geq1$ are constants, satisfying
\[
\sup_{-T\leq t\leq0}\,\left\Vert A_{\Sigma_{t}}\right\Vert _{L^{\infty}\left(B_{R}\left(O\right)\right)}\leq2.
\]
Assume also that $O\in\Sigma_{0}$ and $\left|A_{\Sigma_{0}}\left(O\right)\right|\geq1$.
Then we have 
\[
\inf_{-T\leq t\leq0}\,\left\Vert A_{\Sigma_{t}}\right\Vert _{L^{\infty}\left(B_{R}\left(O\right)\right)}\geq\frac{1}{2}.
\]
\end{lem}

\begin{proof}
Recall that the evolution of the the second fundamental form along
MCF is given by 
\begin{equation}
\left(\partial_{t}-\triangle_{\Sigma_{t}}\right)\left|A_{\Sigma_{t}}\right|^{2}=-2\left|\nabla_{\Sigma_{t}}A_{\Sigma_{t}}\right|^{2}+2\left|A_{\Sigma_{t}}\right|^{4}\label{LBC: evolution of curvature}
\end{equation}
(cf. Chapter 3 in \cite{E}). Consider the cut-off function $\eta_{R}\left(s\right)=\left(1-\frac{s}{R^{2}}\right)_{+}^{3}$,
which satisfies 
\[
\left(\partial_{t}-\triangle_{\Sigma_{t}}\right)\left[\eta_{R}\left(\left|X\right|^{2}+2n\left(t+T\right)\right)\right]\leq0.
\]
Using the product rule and Cauchy-Schwarz inequality, we get
\begin{equation}
\left(\partial_{t}-\triangle_{\Sigma_{t}}\right)\left[\eta_{R}\left(\left|X\right|^{2}+2n\left(t+T\right)\right)\left|A_{\Sigma_{t}}\right|^{2}\right]\label{LBC: localized curvature}
\end{equation}
\[
\leq\eta_{R}\left(\left|X\right|^{2}+2n\left(t+T\right)\right)\,\left(-2\left|\nabla_{\Sigma_{t}}A_{\Sigma_{t}}\right|^{2}+2\left|A_{\Sigma_{t}}\right|^{4}\right)
\]
\[
-8\,\eta'_{R}\left(\left|X\right|^{2}+2n\left(t+T\right)\right)\,\left|A_{\Sigma_{t}}\right|\,X^{\top}\cdot\,\nabla_{\Sigma_{t}}\left|A_{\Sigma_{t}}\right|
\]
\[
\leq2\,\eta_{R}\left(\left|X\right|^{2}+2n\left(t+T\right)\right)\,\left|A_{\Sigma_{t}}\right|^{4}\,+\,8\left\{ \frac{\left[\eta'_{R}\left(\left|X\right|^{2}+2n\left(t+T\right)\right)\right]^{2}}{\eta_{R}\left(\left|X\right|^{2}+2n\left(t+T\right)\right)}\right\} \left|X\right|^{2}\left|A_{\Sigma_{t}}\right|^{2}
\]
\[
\leq2\left|A_{\Sigma_{t}}\right|^{4}\,+\,72\left|A_{\Sigma_{t}}\right|^{2}\leq320\qquad\forall\,t\in\left[-T,0\right].
\]
Note that in the last line we use the property that 
\[
\frac{\left(\eta'_{R}\left(s\right)\right)^{2}}{\eta_{R}\left(s\right)}=\frac{9}{R^{4}}\left(1-\frac{s}{R^{2}}\right)_{+}
\]
and the assumption that $\left\Vert A_{\Sigma_{t}}\right\Vert _{L^{\infty}}\leq2$. 

Now consider 
\[
\phi\left(t\right)=\max_{X\in\Sigma_{t}\cap B_{R}\left(O\right)}\left[\eta_{R}\left(\left|X\right|^{2}+2n\left(t+T\right)\right)\left|A_{\Sigma_{t}}\right|^{2}\right],\quad-T\leq t\leq0.
\]
The assumption that $\left|A_{\Sigma_{0}}\left(O\right)\right|\geq1$
implies
\[
\phi\left(0\right)\geq\left(1-\frac{2nT}{R^{2}}\right)^{3}\geq\left(1-2nT\right)^{3}.
\]
Additionally, applying the maximum principle (cf. Chapter 2 in \cite{M})
to Eq. (\ref{LBC: localized curvature}) gives 
\[
D^{-}\phi\left(t\right)\coloneqq\limsup_{h\searrow0}\frac{\phi\left(t\right)-\phi\left(t-h\right)}{h}\leq320\quad\,\forall\,t\in\left[-T,0\right].
\]
It follows from the comparison principle for ODE (cf. Chapter 2 in
\cite{Wa}) that 
\[
\phi\left(t\right)\geq\left(1-2nT\right)^{3}+320t\geq\frac{1}{2}
\]
for $-T\leq t\leq0$, provided that $0<T\ll1$ (depending on $n$).
Consequently, we get
\[
\max_{\left|X\right|^{2}\leq R^{2}-2n\left(t+T\right)}\,\left|A_{\Sigma_{t}}\right|^{2}\geq\phi\left(t\right)\geq\frac{1}{2}\qquad\forall\,t\in\left[-T,0\right].
\]
\end{proof}
\begin{thm}
\label{pseudo-locality}There exist constants $\delta>0$ and $M,K\geq1$
depending on $n$ with the following property.

Let $\left\{ \Sigma_{t}\right\} _{0\leq t\leq T}$ be a MCF in $B_{M}\left(O\right)$,
where $0<T\leq1$ is a constant, so that $\Sigma_{0}$ is a $\delta$-Lipschitz
graph in $B_{M}\left(O\right)$ with $\left\Vert A_{\Sigma_{0}}\right\Vert _{L^{\infty}\left(B_{M}\left(O\right)\right)}\leq1.$
Then we have
\[
\sup_{0<t\leq T}\,\sup_{P\in\Sigma_{t}\cap B_{\frac{1}{8}}\left(O\right)}r\left(P\right)\left|A_{\Sigma_{t}}\left(P\right)\right|\leq K,
\]
where $r\left(P\right)=\sup\left\{ r>0:B_{r}\left(P\right)\subset B_{\frac{1}{8}}\left(O\right)\right\} $.
In particular, we get
\[
\sup_{0\leq t\leq T}\,\left\Vert A_{\Sigma_{t}}\right\Vert _{L^{\infty}\left(B_{\frac{1}{10}}\left(O\right)\right)}\leq40K.
\]
\end{thm}

\begin{proof}
Suppose to the contrary that there exist sequences of constants $\delta_{i}\searrow0$,
$M_{i}\nearrow\infty$, $K_{i}\nearrow\infty$, and a sequence of
MCF $\left\{ \Sigma_{t}^{i}\right\} _{0\leq t\leq T_{i}}$ in $B_{M_{i}}\left(O\right)$
for some constant $0<T_{i}\leq1$ so that $\Sigma_{0}^{i}$ is a $\delta_{i}$-Lipschitz
graph in $B_{M_{i}}\left(O\right)$ with $\left\Vert A_{\Sigma_{0}^{i}}\right\Vert _{L^{\infty}\left(B_{M_{i}}\left(O\right)\right)}\leq1$
and 
\[
\sup_{0\leq t\leq T_{i}}\,\sup_{P\in\Sigma_{t}^{i}\cap B_{\frac{1}{8}}\left(O\right)}r\left(P\right)\left|A_{\Sigma_{t}^{i}}\left(P\right)\right|>K_{i}.
\]
For each $i$, choose $0\leq t_{i}\leq T_{i}$ and $P_{i}\in\Sigma_{t_{i}}\cap B_{\frac{1}{8}}\left(O\right)$
so that 
\[
\sup_{0\leq t\leq T_{i}}\,\sup_{P\in\Sigma_{t}^{i}\cap B_{\frac{1}{8}}\left(O\right)}r\left(P\right)\left|A_{\Sigma_{t}^{i}}\left(P\right)\right|=r_{i}A_{i},
\]
where $r_{i}=r\left(P_{i}\right)$ and $A_{i}=\left|A_{\Sigma_{t_{i}}^{i}}\left(P_{i}\right)\right|$.
Note that 
\[
\sup_{0\leq t\leq t_{i}}\left\Vert A_{\Sigma_{t}^{i}}\right\Vert _{L^{\infty}\left(B_{\frac{1}{2}r_{i}}\left(P_{i}\right)\right)}\leq2A_{i}
\]
and that $A_{i}>\frac{K_{i}}{r_{i}}\rightarrow\infty$. Also, it must
be true that $t_{i}>0$ for $i\gg1$; otherwise the condition that
$A_{i}\rightarrow\infty$ would contradict with the assumption that
$\left\Vert A_{\Sigma_{0}^{i}}\right\Vert _{L^{\infty}\left(B_{M_{i}}\left(O\right)\right)}\leq1$
for all $i$. Furthermore, in view of Lemma \ref{Chen-Yin lemma}
and the condition that $\delta_{i}\searrow0$ and $M_{i}\nearrow\infty$,
we may assume that 
\[
\sqrt{t_{i}}A_{i}=\sqrt{t_{i}}\left|A_{\Sigma_{t_{i}}^{i}}\left(P_{i}\right)\right|\rightarrow0.
\]
Let
\[
\tilde{\Sigma}_{\tau}^{i}=A_{i}\left(\Sigma_{t_{i}+A_{i}^{-2}\tau}^{i}-P_{i}\right),\quad-t_{i}A_{i}^{2}\leq\tau\leq0,
\]
which is a MCF in $B_{\frac{1}{2}r_{i}A_{i}}\left(O\right)$ satisfying
\[
\sup_{-t_{i}A_{i}^{2}\leq\tau\leq0}\left\Vert A_{\tilde{\Sigma}_{\tau}^{i}}\right\Vert _{L^{\infty}\left(B_{\frac{1}{2}r_{i}A_{i}}\left(O\right)\right)}\leq2,\quad\left|A_{\tilde{\Sigma}_{0}^{i}}\left(O\right)\right|=1,
\]
\[
\left\Vert A_{\tilde{\Sigma}_{-t_{i}A_{i}^{2}}^{i}}\right\Vert _{L^{\infty}\left(B_{\frac{1}{2}r_{i}A_{i}}\left(O\right)\right)}\leq A_{i}^{-1}.
\]
On the other hand, since $r_{i}A_{i}\rightarrow\infty$ and $t_{i}A_{i}^{2}\rightarrow0$,
applying Lemma \ref{lower bound on curvature} to the MCF $\left\{ \tilde{\Sigma}_{\tau}^{i}\right\} _{-t_{i}A_{i}^{2}\leq\tau\leq0}$
gives
\[
\left\Vert A_{\tilde{\Sigma}_{-t_{i}A_{i}^{2}}^{i}}\right\Vert _{L^{\infty}\left(B_{\frac{1}{2}r_{i}A_{i}}\left(O\right)\right)}\geq\frac{1}{2}
\]
for $i\gg1$. This is a contradiction. 
\end{proof}
Next, we would like to study the asymptotically conical property.
Let us begin with the following two definitions concerning a cone
. 
\begin{defn}
\label{regular cone}We say $\mathcal{C}$ is a \textit{regular cone}
if 
\begin{enumerate}
\item $\mathcal{\lambda C}=\mathcal{C}$ for every constant $\lambda>0$
(scale invariance).
\item $\mathcal{C}\setminus\left\{ O\right\} $ is a smooth, properly embedded,
and oriented hypersurface in $\mathbb{R}^{n+1}$. 
\end{enumerate}
\end{defn}

\medskip{}
\begin{defn}
\label{asymptotically conical}A hypersurface $\Sigma$ in $\mathbb{R}^{n+1}$
is said to be \textit{asymptotic to a regular cone} $\mathcal{C}$
at infinity if 
\begin{enumerate}
\item The ``zooming out'' of $\Sigma$ converges locally smoothly to $\mathcal{C}$
in $\mathbb{R}^{n+1}$ away from the origin, i.e.
\[
\frac{1}{R}\,\Sigma\overset{C_{loc}^{\infty}}{\longrightarrow}\mathcal{C}\quad\textrm{\,in}\,\,\,\mathbb{R}^{n+1}\setminus\left\{ O\right\} \quad\textrm{as}\,\;R\rightarrow\infty.
\]
\item There exists $R_{0}>0$ so that $\Sigma\setminus B_{R_{0}}\left(O\right)$
is a normal graph of $u$ over $\mathcal{C}$ outside a compact subset
with
\[
\left\Vert \nabla_{\mathcal{C}}\,u\right\Vert _{L^{\infty}\left(\mathcal{C}\setminus B_{R}\left(O\right)\right)}+\left\Vert u\right\Vert _{L^{\infty}\left(\mathcal{C}\setminus B_{R}\left(O\right)\right)}\rightarrow0\quad\textrm{as}\;\,R\rightarrow\infty.
\]
\end{enumerate}
\end{defn}

\smallskip{}
Compared with the definitions of asymptotically conical given elsewhere,
see \cite{BW5} or Chapter 2 in \cite{E} for instance, Definition
\ref{asymptotically conical} seems more restrictive because of the
presence of the second condition. However, this condition turns out
to be natural in light of Corollary \ref{self-expanding MCF out of cone},
where we show that every time-slice of a self-expanding MCF coming
out of a regular cone does have this property. 

One of the crucial properties following from the asymptotically conical
condition is that outside a large ball, the curvature is inversely
proportional to the radial distance. As a result of the pseudolocality
theorem, this property is preserved along MCF for a period of time.
This is seen in the following proposition.
\begin{prop}
\label{spatial curvature estimate}Let $\left\{ \Sigma_{t}\right\} _{0\leq t\leq T}$
be a MCF in $\mathbb{R}^{n+1}$, where $T>0$ is a constant, so that
$\Sigma_{0}$ is asymptotic to a regular cone $\mathcal{C}$ at infinity.
Then there exist constants $\Lambda,\mathcal{R},\mathcal{K}\geq1$
depending on $n$, $\mathcal{C}$, and $\Sigma_{0}$ so that 
\[
\sup_{0\leq t\leq T}\left\Vert \,\left|X\right|A_{\Sigma_{t}}\right\Vert _{L^{\infty}\left(\mathbb{R}^{n+1}\setminus B_{\max\left\{ \mathcal{R},\Lambda\sqrt{t}\right\} }\left(O\right)\right)}\leq\mathcal{K}.
\]
Here $X$ denotes the position vector of $\Sigma_{t}$.
\end{prop}

\begin{proof}
Let $\delta>0$ and $M,K\geq1$ be the constants in Theorem \ref{pseudo-locality}.
By the second condition in Definition \ref{regular cone}, we can
find $0<\rho<1$ so that for any $Q\in\partial B_{1}\left(O\right)$,
$\mathcal{C}\cap B_{2\rho}\left(Q\right)$ is either empty or a $\frac{\delta}{2}$-Lipschitz
graph (over a hyperplane) with 
\[
2\rho\left\Vert A_{\mathcal{C}}\right\Vert _{L^{\infty}\left(\mathcal{C}\cap B_{2\rho}\left(Q\right)\right)}\leq1.
\]

By the first condition in Definition \ref{asymptotically conical},
there is $\mathcal{R}\geq1$ so that for any $P\in\mathbb{R}^{n+1}\setminus B_{\mathcal{R}}\left(O\right)$,
$\frac{1}{\left|P\right|}\Sigma_{0}\cap B_{\rho}\left(\frac{P}{\left|P\right|}\right)$
is either empty or a $\delta$-Lipschitz graph with 
\[
\rho\left\Vert A_{\frac{1}{\left|P\right|}\Sigma_{0}}\right\Vert _{L^{\infty}\left(B_{\rho}\left(\frac{P}{\left|P\right|}\right)\right)}\leq1.
\]
By rescaling, it means that for any $P\in\mathbb{R}^{n+1}\setminus B_{\mathcal{R}}\left(O\right)$,
we have $\Sigma_{0}\cap B_{\rho\left|P\right|}\left(P\right)$ is
either empty or a $\delta$-Lipschitz graph with 
\[
\rho\left|P\right|\,\left\Vert A_{\Sigma_{0}}\right\Vert _{L^{\infty}\left(B_{\rho\left|P\right|}\left(P\right)\right)}\leq1.
\]

Now fix $P\in\mathbb{R}^{n+1}\setminus B_{\mathcal{R}}\left(O\right)$.
Note that 
\[
\tilde{\Sigma}_{0}=\left(\frac{\rho\left|P\right|}{M}\right)^{-1}\left(\Sigma_{0}-P\right)
\]
is either empty or a $\delta$-Lipschitz graph in $B_{M}\left(O\right)$
with $\left\Vert A_{\tilde{\Sigma}_{0}}\right\Vert _{L^{\infty}\left(B_{M}\left(O\right)\right)}\leq\frac{1}{M}\leq1$.
Applying Theorem \ref{pseudo-locality} to the MCF 
\[
\tilde{\Sigma}_{\tau}=\left(\frac{\rho\left|P\right|}{M}\right)^{-1}\left(\Sigma_{\left(\frac{\rho\left|P\right|}{M}\right)^{2}\tau}-P\right),\quad0\leq\tau\leq\min\left\{ 1,\left(\frac{\rho\left|P\right|}{M}\right)^{-2}T\right\} 
\]
gives 
\[
\sup_{0\leq t\leq\min\left\{ \left(\frac{\rho\left|P\right|}{M}\right)^{2},T\right\} }\left(\frac{\rho\left|P\right|}{M}\right)\left\Vert A_{\Sigma_{t}}\right\Vert _{L^{\infty}\left(B_{\frac{\rho\left|P\right|}{10M}}\left(P\right)\right)}\leq40K.
\]
By setting $\Lambda=\frac{M}{\rho}$ and $\mathcal{K}=\frac{40MK}{\rho}$
we obtain
\[
\sup_{P\in\mathbb{R}^{n+1}\setminus B_{\mathcal{R}}\left(O\right)}\,\sup_{0\leq t\leq\min\left\{ \left(\frac{\left|P\right|}{\Lambda}\right)^{2},T\right\} }\left|P\right|\,\left\Vert A_{\Sigma_{t}}\right\Vert _{L^{\infty}\left(B_{\frac{\left|P\right|}{10\Lambda}}\left(P\right)\right)}\leq\mathcal{K}.
\]
In particular, we have
\[
\sup_{0\leq t\leq T}\,\sup_{P\in\Sigma_{t}\setminus B_{\max\left\{ \mathcal{R},\Lambda\sqrt{t}\right\} }\left(O\right)}\left|P\right|\,\left|A_{\Sigma_{t}}\left(P\right)\right|\leq\mathcal{K}.
\]
\end{proof}
By virtue of the smooth estimates for MCF (cf. Section 3 in \cite{EH2}),
the corresponding estimates for higher order derivatives of the second
fundamental form follow at once from the preceding proposition. To
make the treatment comprehensive, below (in Proposition \ref{smooth estimates})
we include the statement and a proof of the smooth estimates for MCF.
The proof requires the following maximum principle (cf. Proposition
3.17 in \cite{E}). Readers who are familiar with the smooth estimates
for MCF may skip the proof (including Lemma \ref{Ecker-Huisken maximum principle}).
\begin{lem}
\label{Ecker-Huisken maximum principle}Let $\left\{ \Sigma_{t}\right\} _{0\leq t\leq T}$
be a MCF in $B_{1}\left(O\right)$, where $T>0$ is a constant, and
$f$ be a non-negative function on the flow satisfying 
\[
\left(\partial_{t}-\triangle_{\Sigma_{t}}\right)f\leq-\alpha f^{2}+\beta
\]
for some constants $\alpha>0$ and $\beta\geq0$. Then we have 
\[
\sup_{0\leq t\leq T}\,\,\sup_{\Sigma_{t}\cap B_{\frac{1}{2}}\left(O\right)}f\,\leq\,64\,\max\left\{ \sup_{\Sigma_{0}\cap B_{1}\left(O\right)}f,\,\,\,\frac{C\left(n\right)}{\alpha}+\sqrt{\frac{\beta}{\alpha}}\right\} 
\]
and 
\[
\sup_{\frac{T}{2}\leq t\leq T}\,\,\sup_{\Sigma_{t}\cap B_{\frac{1}{2}}\left(O\right)}f\,\leq\,\frac{128}{27}\left[\frac{C\left(n\right)}{\alpha}\left(1+\frac{1}{T}\right)+\sqrt{\frac{\beta}{\alpha}}\,\right].
\]
\end{lem}

\begin{proof}
Consider the cut-off function 
\[
\eta_{1}\left(X,t\right)=\left(1-\left|X\right|^{2}-\frac{t}{2\max\left\{ T,1\right\} }\right)_{+}^{3},
\]
which satisfies 
\begin{equation}
\left(\partial_{t}-\triangle_{\Sigma_{t}}\right)\left[\eta_{1}\left(X,t\right)\right]\,+\,2\frac{\left|\nabla_{\Sigma_{t}}\left[\eta_{1}\left(X,t\right)\right]\right|^{2}}{\eta_{1}\left(X,t\right)}\leq C\left(n\right)\quad\forall\,\,\,t\in\left[0,T\right].\label{MP: evolution of cut-off}
\end{equation}
By the product rule we have
\begin{equation}
\left(\partial_{t}-\triangle_{\Sigma_{t}}\right)\left[\eta_{1}\left(X,t\right)\,f\right]\label{MP: product rule}
\end{equation}
\[
\leq\,\eta_{1}\left(X,t\right)\,\left(-\alpha f^{2}+\beta\right)\,+\,f\,\left(\partial_{t}-\triangle_{\Sigma_{t}}\right)\left[\eta_{1}\left(X,t\right)\right]\,-\,2\,\nabla_{\Sigma_{t}}\left[\eta_{1}\left(X,t\right)\right]\,\cdot\nabla_{\Sigma_{t}}f.
\]
Substituting 
\[
\nabla_{\Sigma_{t}}f=\frac{1}{\eta_{1}\left(X,t\right)}\left\{ \nabla_{\Sigma_{t}}\left[\eta_{1}\left(X,t\right)f\right]-f\,\nabla_{\Sigma_{t}}\left[\eta_{1}\left(X,t\right)\right]\right\} 
\]
and Eq. (\ref{MP: evolution of cut-off}) in Eq. (\ref{MP: product rule})
gives
\begin{equation}
\left(\partial_{t}-\triangle_{\Sigma_{t}}\right)\left[\eta_{1}\left(X,t\right)\,f\right]\label{MP: localized eq}
\end{equation}
\[
\leq\,-2\,\nabla_{\Sigma_{t}}\left[\ln\eta_{1}\left(X,t\right)\right]\,\cdot\nabla_{\Sigma_{t}}\left[\eta_{1}\left(X,t\right)\,f\right]\,+\,\eta_{1}\left(X,t\right)\,\left(-\alpha f^{2}+\beta\right)\,+\,C\left(n\right)f.
\]
Let 
\[
M_{1}=\sup_{0\leq t\leq T}\,\sup_{\Sigma_{t}}\left[\eta_{1}\left(X,t\right)f\right].
\]
Then either 
\[
M_{1}\leq\max_{X\in\Sigma_{0}}\left[\eta_{1}\left(X,0\right)\left.f\right|_{t=0}\right],
\]
or else 
\[
M_{1}>\max_{X\in\Sigma_{0}}\left[\eta_{1}\left(X,0\right)\left.f\right|_{t=0}\right].
\]
In the latter case the maximum must be attained at some interior points,
say $\left(P,t_{0}\right)$, i.e.
\[
\eta_{1}\left(P,t_{0}\right)f\left(P,t_{0}\right)=M_{1}\quad\textrm{with}\,\,\,0<t_{0}\leq T,\,\,P\in\Sigma_{t_{0}}\cap\,\textrm{spt }\eta_{1}\left(\cdot,t_{0}\right).
\]
On substituting $\left(P,t_{0}\right)$ for $\left(X,t\right)$ in
Eq. (\ref{MP: localized eq}) we obtain
\[
0\leq-\alpha M_{1}^{2}+C\left(n\right)M_{1}+\beta.
\]
It follows that 
\[
M_{1}\leq\frac{C\left(n\right)+\sqrt{C^{2}\left(n\right)+4\alpha\beta}}{2\alpha}\leq\frac{C\left(n\right)}{\alpha}+\sqrt{\frac{\beta}{\alpha}}.
\]
The first estimate in the lemma follows immediately by noting that
$\frac{1}{64}\leq\eta_{1}\left(X,t\right)\leq1$ for $X\leq\frac{1}{2}$,
$0\leq t\leq T$.

Likewise, if we consider the cut-off function $\eta_{2}\left(X,t\right)=\frac{t}{T}\left(1-\left|X\right|^{2}\right)_{+}^{3}$,
which satisfies 
\[
\left(\partial_{t}-\triangle_{\Sigma_{t}}\right)\left[\eta_{2}\left(X,t\right)\right]+2\frac{\left|\nabla_{\Sigma_{t}}\left[\eta_{2}\left(X,t\right)\right]\right|^{2}}{\eta_{2}\left(X,t\right)}\leq C\left(n\right)\left(1+\frac{1}{T}\right)\quad\,\forall\,\,t\in\left(0,T\right],
\]
then by the same reasoning as that for $\eta_{1}\left(X,t\right)f$,
we infer that 
\[
\left(\partial_{t}-\triangle_{\Sigma_{t}}\right)\left[\eta_{2}\left(X,t\right)\,f\right]\leq\,-2\,\nabla_{\Sigma_{t}}\left[\ln\eta_{2}\left(X,t\right)\right]\,\cdot\nabla_{\Sigma_{t}}\left[\eta_{2}\left(X,t\right)\,f\right]
\]
\[
+\,\eta_{2}\left(X,t\right)\,\left(-\alpha f^{2}+\beta\right)\,+\,C\left(n\right)\left(1+\frac{1}{T}\right)f.
\]
It follows that 
\[
M_{2}=\sup_{0\leq t\leq T}\,\sup_{\Sigma_{t}}\left[\eta_{2}\left(X,t\right)f\right]
\]
either is zero or else satisfies $0\leq-\alpha M_{1}^{2}+C\left(n,T\right)M_{1}+\beta.$
In the latter case we get 
\[
M_{2}\leq\frac{C\left(n\right)\left(1+\frac{1}{T}\right)+\sqrt{C^{2}\left(n\right)\left(1+\frac{1}{T}\right)^{2}+4\alpha\beta}}{2\alpha}\leq\frac{C\left(n\right)\left(1+\frac{1}{T}\right)}{\alpha}+\sqrt{\frac{\beta}{\alpha}}.
\]
Therefore, the second estimate in the lemma follows in view of the
fact that $\frac{27}{128}\leq\eta_{2}\left(X,t\right)\leq1$ for $X\leq\frac{1}{2}$,
$\frac{T}{2}\leq t\leq T$.
\end{proof}
Below are the smooth estimates for MCF (cf. Section 3 in \cite{EH2}
or Chapter 3 in \cite{E}).
\begin{prop}
\label{smooth estimates}Let $\left\{ \Sigma_{t}\right\} _{0\leq t\leq T}$
be a MCF in $B_{1}\left(O\right)$, where $T>0$ is a constant, satisfying
\[
\sup_{0\leq t\leq T}\left\Vert A_{\Sigma_{t}}\right\Vert _{L^{\infty}\left(B_{1}\left(O\right)\right)}\leq K
\]
for some constant $K>0$. Then for every $k\in\mathbb{N}$ we have
\[
\sup_{0\leq t\leq T}\left\Vert \nabla_{\Sigma_{t}}^{k}A_{\Sigma_{t}}\right\Vert _{L^{\infty}\left(B_{\frac{1}{2}}\left(O\right)\right)}\leq C\left(n,k,K,\left\Vert \nabla_{\Sigma_{0}}A_{\Sigma_{0}}\right\Vert _{L^{\infty}\left(B_{1}\left(O\right)\right)},\cdots,\left\Vert \nabla_{\Sigma_{0}}^{k}A_{\Sigma_{0}}\right\Vert _{L^{\infty}\left(B_{1}\left(O\right)\right)}\right)
\]
and 
\[
\sup_{\frac{T}{2}\leq t\leq T}\left\Vert \nabla_{\Sigma_{t}}^{k}A_{\Sigma_{t}}\right\Vert _{L^{\infty}\left(B_{\frac{1}{2}}\left(O\right)\right)}\leq C\left(n,k,K,T\right).
\]
\end{prop}

\begin{proof}
We will illustrate the idea by presenting the estimate for the first
derivative; estimates for all the other higher order derivatives follow
from a similar argument so are omitted. 

Recall that 
\begin{equation}
\left(\partial_{t}-\triangle_{\Sigma_{t}}\right)\left|\nabla_{\Sigma_{t}}A_{\Sigma_{t}}\right|^{2}=-2\left|\nabla_{\Sigma_{t}}^{2}A_{\Sigma_{t}}\right|^{2}+A_{\Sigma_{t}}\ast A_{\Sigma_{t}}\ast\nabla_{\Sigma_{t}}A_{\Sigma_{t}}\ast\nabla_{\Sigma_{t}}A_{\Sigma_{t}}\label{SE: evolution of derivative of curvatures}
\end{equation}
\[
\leq-2\left|\nabla_{\Sigma_{t}}^{2}A_{\Sigma_{t}}\right|^{2}+C\left(n\right)\left|A_{\Sigma_{t}}\right|^{2}\left|\nabla_{\Sigma_{t}}A_{\Sigma_{t}}\right|^{2},
\]
where the notation $\ast$ means some form of contraction of tensors
(cf. Section 2.3 in \cite{M}). As in Proof of Proposition 3.22 in
\cite{E}, let us consider the function 
\[
f=\left(7\bar{K}^{2}+\left|A_{\Sigma_{t}}\right|^{2}\right)\left|\nabla_{\Sigma_{t}}A_{\Sigma_{t}}\right|^{2},
\]
where $\bar{K}=\max\left\{ K,1\right\} $. Using the product rule
and Eqs. (\ref{LBC: evolution of curvature}) and (\ref{SE: evolution of derivative of curvatures})
we get
\begin{equation}
\left(\partial_{t}-\triangle_{\Sigma_{t}}\right)f\leq\left(7\bar{K}^{2}+\left|A_{\Sigma_{t}}\right|^{2}\right)\left(-2\left|\nabla_{\Sigma_{t}}^{2}A_{\Sigma_{t}}\right|^{2}+C\left(n\right)\left|A_{\Sigma_{t}}\right|^{2}\left|\nabla_{\Sigma_{t}}A_{\Sigma_{t}}\right|^{2}\right)\label{SE: eq of f}
\end{equation}
\[
+\left(-2\left|\nabla_{\Sigma_{t}}A_{\Sigma_{t}}\right|^{2}+2\left|A_{\Sigma_{t}}\right|^{4}\right)\left|\nabla_{\Sigma_{t}}A_{\Sigma_{t}}\right|^{2}-2\,\nabla_{\Sigma_{t}}\left|A_{\Sigma_{t}}\right|^{2}\cdot\nabla_{\Sigma_{t}}\left|\nabla_{\Sigma_{t}}A_{\Sigma_{t}}\right|^{2}
\]
By Cauchy-Schwarz inequality, the last term in the above can be estimated
as follows:
\begin{equation}
-2\,\nabla_{\Sigma_{t}}\left|A_{\Sigma_{t}}\right|^{2}\cdot\nabla_{\Sigma_{t}}\left|\nabla_{\Sigma_{t}}A_{\Sigma_{t}}\right|^{2}\,\leq\,8\left|\nabla_{\Sigma_{t}}A_{\Sigma_{t}}\right|^{2}\left|A_{\Sigma_{t}}\right|\,\left|\nabla_{\Sigma_{t}}^{2}A_{\Sigma_{t}}\right|\label{SE: CS ineq}
\end{equation}
\[
\leq\left|\nabla_{\Sigma_{t}}A_{\Sigma_{t}}\right|^{4}\,+\,16\left|A_{\Sigma_{t}}\right|^{2}\left|\nabla_{\Sigma_{t}}^{2}A_{\Sigma_{t}}\right|^{2}.
\]
On substituting Eq. (\ref{SE: CS ineq}) in Eq. (\ref{SE: eq of f})
and using the condition that $\left|A_{\Sigma_{t}}\right|\leq\bar{K}$
and 
\[
7\bar{K}^{2}\left|\nabla_{\Sigma_{t}}A_{\Sigma_{t}}\right|^{2}\leq f\leq8\bar{K}^{2}\left|\nabla_{\Sigma_{t}}A_{\Sigma_{t}}\right|^{2},
\]
we obtain
\[
\left(\partial_{t}-\triangle_{\Sigma_{t}}\right)f\leq-\left|\nabla_{\Sigma_{t}}A_{\Sigma_{t}}\right|^{4}\,+\,2\left[4C\left(n\right)+1\right]\bar{K}^{4}\left|\nabla_{\Sigma_{t}}A_{\Sigma_{t}}\right|^{2}
\]
\[
\leq\frac{-1}{64\bar{K}^{2}}f^{2}\,+\,\frac{2}{7}\left[4C\left(n\right)+1\right]\bar{K}^{2}f\,\,\,\leq\frac{-1}{128\bar{K}^{2}}f^{2}\,+\,32\left(\frac{2}{7}\left[4C\left(n\right)+1\right]\right)^{2}\bar{K}^{4}.
\]
It follows from Lemma \ref{Ecker-Huisken maximum principle} that
\[
\sup_{0\leq t\leq T}\left\Vert \nabla_{\Sigma_{t}}A_{\Sigma_{t}}\right\Vert _{L^{\infty}\left(B_{\frac{1}{2}}\left(O\right)\right)}\leq C\left(n,K,\left\Vert \nabla_{\Sigma_{t}}A_{\Sigma_{0}}\right\Vert _{L^{\infty}\left(B_{1}\left(O\right)\right)}\right)
\]
and
\[
\sup_{\frac{T}{2}\leq t\leq T}\left\Vert \nabla_{\Sigma_{t}}A_{\Sigma_{t}}\right\Vert _{L^{\infty}\left(B_{\frac{1}{2}}\left(O\right)\right)}\leq C\left(n,K,T\right).
\]
\end{proof}
As a corollary of Propositions \ref{spatial curvature estimate} and
\ref{smooth estimates}, we have the following estimates for the derivatives
of curvature for any order.
\begin{cor}
\label{spatial smooth estimates}Under the hypothesis in Proposition
\ref{spatial curvature estimate}, for every $k\in\mathbb{N}$, there
exist constants $\Lambda,\mathcal{R}\geq1$ depending on $n$, $k$,
$\mathcal{C}$, and $\Sigma_{0}$ so that 
\[
\sup_{0\leq t\leq T}\left\Vert \,\left|X\right|^{k+1}\nabla_{\Sigma_{t}}^{k}A_{\Sigma_{t}}\right\Vert _{L^{\infty}\left(\mathbb{R}^{n+1}\setminus B_{\max\left\{ \mathcal{R},\Lambda\sqrt{t}\right\} }\left(O\right)\right)}\leq C\left(n,k,\mathcal{C},\Sigma_{0}\right).
\]
\end{cor}

\begin{proof}
We will only present the estimate for $k=1$, as all other cases follow
from a similar argument.

In the same setting as in the proof of Proposition \ref{spatial curvature estimate}
(with a possibly larger $\mathcal{R}$), we may assume that for any
$P\in\mathbb{R}^{n+1}\setminus B_{\mathcal{R}}\left(O\right)$, 
\[
\frac{1}{\left|P\right|}\Sigma_{0}\cap B_{\rho}\left(\frac{P}{\left|P\right|}\right)
\]
is either empty or a $\delta$-Lipschitz graph with 
\[
\left|P\right|\,\left\Vert A_{\Sigma_{0}}\right\Vert _{L^{\infty}\left(B_{\rho\left|P\right|}\left(P\right)\right)}=\left\Vert A_{\frac{1}{\left|P\right|}\Sigma_{0}}\right\Vert _{L^{\infty}\left(B_{\rho}\left(\frac{P}{\left|P\right|}\right)\right)}\leq\frac{1}{\rho},
\]
\[
\left|P\right|^{2}\left\Vert \nabla_{\Sigma_{0}}A_{\Sigma_{0}}\right\Vert _{L^{\infty}\left(B_{\rho\left|P\right|}\left(P\right)\right)}=\left\Vert \nabla_{\frac{1}{\left|P\right|}\Sigma_{0}}A_{\frac{1}{\left|P\right|}\Sigma_{0}}\right\Vert _{L^{\infty}\left(B_{\rho}\left(\frac{P}{\left|P\right|}\right)\right)}\leq2\left\Vert \nabla_{\mathcal{C}}A_{\mathcal{C}}\right\Vert _{L^{\infty}\left(B_{\frac{3}{2}}\left(O\right)\setminus B_{\frac{1}{2}}\left(O\right)\right)}.
\]
Note that we have shown that
\[
\sup_{P\in\mathbb{R}^{n+1}\setminus B_{\mathcal{R}}\left(O\right)}\,\sup_{0\leq t\leq\min\left\{ \left(\frac{\left|P\right|}{\Lambda}\right)^{2},T\right\} }\left|P\right|\left\Vert A_{\Sigma_{t}}\right\Vert _{L^{\infty}\left(B_{\frac{\left|P\right|}{10\Lambda}}\left(P\right)\right)}\leq\mathcal{K}
\]
in the proof of Proposition \ref{spatial curvature estimate}. 

For every $P\in\mathbb{R}^{n+1}\setminus B_{\mathcal{R}}\left(O\right)$,
applying Proposition \ref{smooth estimates} to the MCF 
\[
\tilde{\Sigma}_{\tau}=\left(\frac{\left|P\right|}{10\Lambda}\right)^{-1}\left(\Sigma_{\left(\frac{\left|P\right|}{10\Lambda}\right)^{2}\tau}-P\right),\quad0\leq\tau\leq\min\left\{ 100,\left(\frac{\left|P\right|}{10\Lambda}\right)^{-2}T\right\} 
\]
gives 
\[
\left|P\right|^{2}\left\Vert \nabla_{\Sigma_{t}}A_{\Sigma_{t}}\right\Vert _{L^{\infty}\left(B_{\frac{\left|P\right|}{20\Lambda}}\left(P\right)\right)}\leq C\left(n,\Lambda,\mathcal{K},\left\Vert \nabla_{\mathcal{C}}A_{\mathcal{C}}\right\Vert _{L^{\infty}\left(B_{\frac{3}{2}}\left(O\right)\setminus B_{\frac{1}{2}}\left(O\right)\right)}\right)
\]
for $0\leq t\leq\min\left\{ \left(\frac{\left|P\right|}{\Lambda}\right)^{2},T\right\} $.
In other words, we get 
\[
\sup_{P\in\mathbb{R}^{n+1}\setminus B_{\mathcal{R}}\left(O\right)}\,\sup_{0\leq t\leq\min\left\{ \left(\frac{\left|P\right|}{\Lambda}\right)^{2},T\right\} }\left|P\right|^{2}\left\Vert \nabla_{\Sigma_{t}}A_{\Sigma_{t}}\right\Vert _{L^{\infty}\left(B_{\frac{\left|P\right|}{20\Lambda}}\left(P\right)\right)}
\]
\[
\leq C\left(n,\Lambda,\mathcal{K},\left\Vert \nabla_{\mathcal{C}}A_{\mathcal{C}}\right\Vert _{L^{\infty}\left(B_{\frac{3}{2}}\left(O\right)\setminus B_{\frac{1}{2}}\left(O\right)\right)}\right),
\]
which implies
\[
\sup_{0\leq t\leq T}\,\sup_{P\in\Sigma_{t}\setminus B_{\max\left\{ \mathcal{R},\Lambda\sqrt{t}\right\} }\left(O\right)}\left|P\right|^{2}\left|\nabla_{\Sigma_{t}}A_{\Sigma_{t}}\left(P\right)\right|\leq C\left(n,\Lambda,\mathcal{K},\left\Vert \nabla_{\mathcal{C}}A_{\mathcal{C}}\right\Vert _{L^{\infty}\left(B_{\frac{3}{2}}\left(O\right)\setminus B_{\frac{1}{2}}\left(O\right)\right)}\right).
\]
\end{proof}
Our next goal is to prove the preservation of the asymptotically conical
property along MCF (see Proposition \ref{asymptotically conical along MCF}).
Since the evolution of various quantities on the flow, including the
position, direction, convexity, etc, are controlled by the curvature
and its derivatives, which are inversely proportional to some power
of the radial distance (by the preceding corollary), it is plausible
that the changes in these quantities are tiny along the flow outside
a large ball. Accordingly, we can expect that the later time-slices
should stay close to the cone in the distance. To carry out the idea,
we will use the local graph parametrization of the flow (see Proposition
\ref{analyticity: vertical parametrization}). 

In preparation for Proposition \ref{analyticity: vertical parametrization},
the following lemma shows that we can locally parametrize each time-slice
of the flow as a graph over a time-dependent hyperplane, provided
that the mean curvature stays uniformly bounded.
\begin{lem}
\label{Allard's regularity}Given $\delta>0$, there exists a constant
$0<\theta<1$ depending on $n$ and $\delta$ with the following property.

Let $\left\{ \Sigma_{t}\right\} _{0\leq t\leq T}$ be a MCF in $B_{r}\left(O\right)$,
where $0<r\leq1$ and $0<T\leq\theta^{2}r^{2}$ are constants, so
that $\Sigma_{0}\cap B_{r}\left(O\right)$ is a $\theta$-Lipschitz
graph passing through $O$ and that 
\[
\sup_{0\leq t\leq T}\left\Vert H_{\Sigma_{t}}\right\Vert _{L^{\infty}\left(B_{r}\left(O\right)\right)}\leq n.
\]
Then for every $0\leq t\leq T$, $\Sigma_{t}\cap B_{\theta r}\left(\gamma\left(t\right)\right)$
is a $\delta$-Lipschitz graph over a subset of $T_{\gamma\left(t\right)}\Sigma_{t}$,
where $\gamma:\left[0,T\right]\rightarrow\mathbb{R}^{n+1}$ is the
``trajectory'' of the origin along the MCF, i.e. $\gamma\left(0\right)=O$,
$\gamma\left(t\right)\in\Sigma_{t}$, and $\gamma'\left(t\right)=\vec{H}_{\Sigma_{t}}\left(\gamma\left(t\right)\right)$.
\end{lem}

\begin{proof}
Let $0<\vartheta\ll1$ be a constant to be determined. Assume that
$\left\{ \Sigma_{t}\right\} _{0\leq t\leq T}$ is a MCF in $B_{r}\left(O\right)$
with $0<r\leq1$, $0<T\ll1$ (to be determined and will be depending
only on $r,\vartheta$) so that $\Sigma_{0}\cap B_{r}\left(O\right)$
is a $\vartheta$-Lipschitz graph containing $O$ and that 
\[
\sup_{0\leq t\leq T}\left\Vert H_{\Sigma_{t}}\right\Vert _{L^{\infty}\left(B_{r}\left(O\right)\right)}\leq n.
\]
Firstly, set $r_{1}=\frac{r}{\sqrt{1+\vartheta^{2}}}$. Upon integrating
the equation 
\[
\left(\partial_{t}-\triangle_{\Sigma_{t}}\right)\left(1-\frac{\left|X\right|^{2}+2nt}{r_{1}^{2}}\right)_{+}^{3}\leq0
\]
and using 
\[
\partial_{t}\,d\mathcal{H}^{n}\lfloor\Sigma_{t}=-H_{\Sigma_{t}}^{2}\,d\mathcal{H}^{n}\lfloor\Sigma_{t}
\]
(cf. Chapter 4 in \cite{E}), we get
\[
\int_{\Sigma_{t}}\left(1-\frac{\left|X\right|^{2}+2nt}{r_{1}^{2}}\right)_{+}^{3}d\mathcal{H}^{n}\left(X\right)-\int_{\Sigma_{0}}\left(1-\frac{\left|X\right|^{2}}{r_{1}^{2}}\right)_{+}^{3}d\mathcal{H}^{n}\left(X\right)
\]
\[
=-\int_{0}^{t}\int_{\Sigma_{\tau}}H_{\Sigma_{\tau}}^{2}\left(1-\frac{\left|X\right|^{2}+2n\tau}{r_{1}^{2}}\right)_{+}^{3}d\mathcal{H}^{n}\left(X\right)d\tau\leq0,
\]
which implies
\begin{equation}
\mathcal{H}^{n}\left(\Sigma_{t}\cap B_{\vartheta r_{1}}\left(O\right)\right)\leq\frac{\mathcal{H}^{n}\left(\Sigma_{0}\cap B_{r_{1}}\left(O\right)\right)}{\left[1-\vartheta^{2}-2n\left(1+\vartheta^{2}\right)\frac{T}{r^{2}}\right]^{3}}\leq\frac{\sqrt{1+\vartheta^{2}}\,\boldsymbol{\omega}_{n}r_{1}^{n}}{\left[1-\vartheta^{2}-2n\left(1+\vartheta^{2}\right)\frac{T}{r^{2}}\right]^{3}}\label{A: area ratio}
\end{equation}
for $0\leq t\leq T$. Note that the last term in the above inequality
is obtained by using the condition that $\Sigma_{0}\cap B_{r}\left(O\right)$
is a $\vartheta$-Lipschitz graph containing $O$.

Next, set 
\[
r_{2}=\frac{\vartheta r_{1}}{1+\vartheta}=\frac{\vartheta r}{\left(1+\vartheta\right)\sqrt{1+\vartheta^{2}}}.
\]
Let $\phi\left(X\right)$ be a function satisfying $\chi_{B_{r_{2}}\left(O\right)}\leq\phi\leq\chi_{B_{\left(1+\vartheta\right)r_{2}}\left(O\right)}$
with $\left\Vert D\phi\right\Vert _{L^{\infty}\left(\mathbb{R}^{n+1}\right)}\leq\frac{2}{\vartheta r_{2}}$.
By the weak formulation of MCF (cf. Chapter 4 in \cite{E}), we have
\begin{equation}
\int_{\Sigma_{t}}\phi\left(X\right)d\mathcal{H}^{n}\left(X\right)-\int_{\Sigma_{0}}\phi\left(X\right)d\mathcal{H}^{n}\left(X\right)\label{A: weak MCF}
\end{equation}
\[
\leq\int_{0}^{t}\int_{\Sigma_{\tau}}\left\{ -H_{\Sigma_{\tau}}^{2}\phi\left(X\right)+\vec{H}_{\Sigma_{\tau}}\cdot D\phi\left(X\right)\right\} d\mathcal{H}^{n}\left(X\right)d\tau
\]
 
\[
\leq nT\left\Vert D\phi\right\Vert _{L^{\infty}\left(\mathbb{R}^{n+1}\right)}\sup_{0\leq\tau\leq T}\mathcal{H}^{n}\left(\Sigma_{\tau}\cap B_{\theta r_{1}}\left(O\right)\right).
\]
It follows from (\ref{A: area ratio}) and (\ref{A: weak MCF}) that
\[
\frac{\mathcal{H}^{n}\left(\Sigma_{t}\cap B_{r_{2}}\left(O\right)\right)}{\boldsymbol{\omega}_{n}r_{2}^{n}}
\]
\[
\leq\left(1+\vartheta\right)^{n}\frac{\mathcal{H}^{n}\left(\Sigma_{0}\cap B_{\left(1+\vartheta\right)r_{2}}\left(O\right)\right)}{\boldsymbol{\omega}_{n}\left(\left(1+\vartheta\right)r_{2}\right)^{n}}\,+nT\left\Vert D\phi\right\Vert _{L^{\infty}\left(\mathbb{R}^{n+1}\right)}\sup_{0\leq\tau\leq T}\frac{\mathcal{H}^{n}\left(\Sigma_{\tau}\cap B_{\vartheta r_{1}}\left(O\right)\right)}{\boldsymbol{\omega}_{n}r_{2}^{n}}
\]
 
\[
\leq\left(1+\vartheta\right)^{n}\sqrt{1+\vartheta^{2}}+\frac{2\left(1+\vartheta^{2}\right)\left(1+\vartheta\right)^{n+1}}{\left[1-\vartheta^{2}-2n\left(1+\vartheta^{2}\right)\frac{T}{r^{2}}\right]^{3}}\frac{nT}{\vartheta^{n+2}r}.
\]
Thus we have
\[
\frac{\mathcal{H}^{n}\left(\Sigma_{t}\cap B_{r_{2}}\left(O\right)\right)}{\boldsymbol{\omega}_{n}r_{2}^{n}}\leq1+C\left(n\right)\vartheta
\]
provided that $0<T\leq\vartheta^{n+3}r^{2}$. 

Now set 
\[
r_{3}=r_{2}-n\vartheta^{n+3}r=\left(\frac{\vartheta}{\left(1+\vartheta\right)\sqrt{1+\vartheta^{2}}}-n\vartheta^{n+3}\right)r.
\]
Let $\gamma:\left[0,T\right]\rightarrow\mathbb{R}^{n+1}$ be the trajectory
of the origin along the MCF $\left\{ \Sigma_{t}\right\} $. Since
\[
\left|\gamma'\left(t\right)\right|=\left|\vec{H}_{\Sigma_{t}}\left(\gamma\left(t\right)\right)\right|\leq n,
\]
the mean value theorem implies 
\[
\left|\gamma\left(t\right)\right|\leq nt\leq nT\leq n\vartheta^{n+3}r\quad\,\,\forall\,\,0\leq t\leq T.
\]
Whence, $B_{r_{3}}\left(\gamma\left(t\right)\right)\subset B_{r_{2}}\left(O\right)$
and we have 
\begin{equation}
\frac{\mathcal{H}^{n}\left(\Sigma_{t}\cap B_{r_{3}}\left(\gamma\left(t\right)\right)\right)}{\boldsymbol{\omega}_{n}r_{3}^{n}}\leq\frac{\mathcal{H}^{n}\left(\Sigma_{t}\cap B_{r_{2}}\left(O\right)\right)}{\boldsymbol{\omega}_{n}r_{2}^{n}}\left(\frac{r_{2}}{r_{3}}\right)^{n}\label{A: unit area ratio}
\end{equation}
\[
\leq\frac{1+C\left(n\right)\vartheta}{1-n\vartheta^{n+2}\left(1+\vartheta\right)\sqrt{1+\vartheta^{2}}}\,=\,1+\mathfrak{C}\left(n\right)\vartheta
\]
for all $0\leq t\leq T$. Note also that 
\begin{equation}
r_{3}\left\Vert H_{\Sigma_{t}}\right\Vert _{L^{\infty}\left(B_{r_{3}}\left(\gamma\left(t\right)\right)\right)}\leq nr_{3}\leq n\vartheta r\leq n\vartheta.\label{A: small mean curvature}
\end{equation}

Lastly, given $\delta>0,$ it follows from Allard's regularity theorem
(cf. \cite{A}) and conditions (\ref{A: unit area ratio}) and (\ref{A: small mean curvature})
that there exist $0<\vartheta\ll1$ (depending on $n,\delta$) so
that each $\Sigma_{t}\cap B_{\vartheta r_{3}}\left(\gamma\left(t\right)\right)$
is a $\delta$-Lipschitz graph over (a subset of) $T_{\gamma\left(t\right)}\Sigma_{t}$.
Note that the lemma would hold by choosing 
\[
\theta=\min\left\{ \vartheta^{\frac{n+3}{2}},\frac{\vartheta}{\left(1+\vartheta\right)\sqrt{1+\vartheta^{2}}}-n\vartheta^{n+3}\right\} .
\]
\end{proof}
By virtue of Ecker-Huisken's gradient estimate for MCF (cf. Section
2 in \cite{EH2}), in the following proposition we show that the graphs
in the preceding lemma can be chosen to be over the same hyperplane,
which then gives the desired local graph parametrization.
\begin{prop}
\label{analyticity: vertical parametrization}Given $\delta>0$, there
exists a constant $0<\theta<1$ depending on $n$ and $\delta$ with
the following property. 

Let $\left\{ \Sigma_{t}\right\} _{0\leq t\leq T}$ be a MCF in $B_{r}\left(O\right)$,
where $0<r\leq1$ and $0<T\leq\theta^{2}r^{2}$ are constants, so
that $\Sigma_{0}\cap B_{r}\left(O\right)$ is a $\theta$-Lipschitz
graph containing $O$ and that 
\[
\sup_{0\leq t\leq T}\left\Vert H_{\Sigma_{t}}\right\Vert _{L^{\infty}\left(B_{r}\left(O\right)\right)}\leq n.
\]
Then there exist a time-dependent domain $\Omega_{t}\subset T_{O}\Sigma_{0}\simeq\mathbb{R}^{n}$
and a time-dependent function $u\left(\cdot,t\right):\Omega_{t}\rightarrow\mathbb{R}$
so that $\left\Vert \partial_{x}u\left(\cdot,t\right)\right\Vert _{L^{\infty}\left(\Omega_{t}\right)}\leq\delta$
and
\[
\Sigma_{t}\cap B_{\theta r}\left(O\right)=\left\{ X=\left(x,u\left(x,t\right)\right):x\in\Omega_{t}\right\} 
\]
for $0\leq t\leq T$. In addition, we have $\Omega_{t}\supset B_{\frac{\theta r}{2\sqrt{1+\delta^{2}}}}^{n}\left(O\right)$
(the $n$-dimensional ball).
\end{prop}

\begin{proof}
Let $\left\{ \Sigma_{t}\right\} _{0\leq t\leq T}$ be a MCF in $B_{r}\left(O\right)$,
where $0<r\leq1$ and $0<T\ll1$ (to be determined). Let $\mathbf{e}=N_{\Sigma_{0}}\left(O\right)$,
where $N_{\Sigma_{0}}$ is the unit normal vector of $\Sigma_{0}$.
Following \cite{EH2}, let us consider the quantity $N_{\Sigma_{t}}\cdot\mathbf{e}$,
which satisfies 
\[
\left(\partial_{t}-\triangle_{\Sigma_{t}}\right)N_{\Sigma_{t}}\cdot\mathbf{e}=\left|A_{\Sigma_{t}}\right|^{2}N_{\Sigma_{t}}\cdot\mathbf{e}.
\]
Wherever $N_{\Sigma_{t}}\cdot\mathbf{e}>0$ holds, the chain rule
implies
\begin{equation}
\left(\partial_{t}-\triangle_{\Sigma_{t}}\right)\left(N_{\Sigma_{t}}\cdot\mathbf{e}\right)^{-1}=-\left|A_{\Sigma_{t}}\right|^{2}\left(N_{\Sigma_{t}}\cdot\mathbf{e}\right)^{-1}-2\left(N_{\Sigma_{t}}\cdot\mathbf{e}\right)^{-3}\left|\nabla_{\Sigma_{t}}\left(N_{\Sigma_{t}}\cdot\mathbf{e}\right)\right|^{2}.\label{AVP: gradient}
\end{equation}
For each $0<s\leq r$, let $\eta_{s}\left(X,t\right)=\left(1-\frac{\left|X\right|^{2}+2nt}{s^{2}}\right)_{+}^{3}$
, which satisfies
\begin{equation}
\left(\partial_{t}-\triangle_{\Sigma_{t}}\right)\left[\eta_{s}\left(X,t\right)\right]\leq0.\label{AVP: localization}
\end{equation}
Applying the product rule to Eqs. (\ref{AVP: gradient}) and (\ref{AVP: localization})
gives 
\begin{equation}
\left(\partial_{t}-\triangle_{\Sigma_{t}}\right)\left[\eta_{s}\left(X,t\right)\left(N_{\Sigma_{t}}\cdot\mathbf{e}\right)^{-1}\right]\label{AVP: equation}
\end{equation}
\[
\leq\eta_{s}\left(X,t\right)\left\{ -\left|A_{\Sigma_{t}}\right|^{2}\left(N_{\Sigma_{t}}\cdot\mathbf{e}\right)^{-1}-2\left(N_{\Sigma_{t}}\cdot\mathbf{e}\right)^{-3}\left|\nabla_{\Sigma_{t}}\left(N_{\Sigma_{t}}\cdot\mathbf{e}\right)\right|^{2}\right\} 
\]
\[
-2\,\nabla_{\Sigma_{t}}\left[\eta_{s}\left(X,t\right)\right]\cdot\,\nabla_{\Sigma_{t}}\left[\left(N_{\Sigma_{t}}\cdot\mathbf{e}\right)^{-1}\right].
\]
On substituting 
\[
\nabla_{\Sigma_{t}}\left[\eta_{s}\left(X,t\right)\right]=\left(N_{\Sigma_{t}}\cdot\mathbf{e}\right)\nabla_{\Sigma_{t}}\left[\eta_{s}\left(X,t\right)\left(N_{\Sigma_{t}}\cdot\mathbf{e}\right)^{-1}\right]\,+\,\left[\eta_{s}\left(X,t\right)\left(N_{\Sigma_{t}}\cdot\mathbf{e}\right)^{-1}\right]\nabla_{\Sigma_{t}}\left(N_{\Sigma_{t}}\cdot\mathbf{e}\right)
\]
in the last term of Eq. (\ref{AVP: equation}), we obtain
\begin{equation}
\left(\partial_{t}-\triangle_{\Sigma_{t}}\right)\left[\eta_{s}\left(X,t\right)\left(N_{\Sigma_{t}}\cdot\mathbf{e}\right)^{-1}\right]\label{AVP: evolution of localized gradient}
\end{equation}
\[
\leq2\,\nabla_{\Sigma_{t}}\ln\left(N_{\Sigma_{t}}\cdot\mathbf{e}\right)\cdot\nabla_{\Sigma_{t}}\left[\eta_{s}\left(X,t\right)\left(N_{\Sigma_{t}}\cdot\mathbf{e}\right)^{-1}\right]-\left|A_{\Sigma_{t}}\right|^{2}\eta_{s}\left(X,t\right)\left(N_{\Sigma_{t}}\cdot\mathbf{e}\right)^{-1}.
\]

Now given $\delta>0$, by Lemma \ref{Allard's regularity} there exists
$0<\vartheta<1$ so that if $\Sigma_{0}\cap B_{r}\left(O\right)$
is a $\vartheta$-Lipschitz graph containing $O$ and that 
\[
\sup_{0\leq t\leq T}\left\Vert H_{\Sigma_{t}}\right\Vert _{L^{\infty}\left(B_{r}\left(O\right)\right)}\leq n
\]
for each $0\leq t\leq T$ with $T\leq\vartheta^{2}r^{2}$, then $\Sigma_{t}\cap B_{\vartheta r}\left(\gamma\left(t\right)\right)$
is a $\varepsilon\delta$-Lipschitz graph over a subset of $T_{\gamma\left(t\right)}\Sigma_{t}$.
Here $0<\varepsilon\ll1$ is a constant to be determined (and depending
only on $n,\delta$) and $\gamma:\left[0,T\right]\rightarrow\mathbb{R}^{n+1}$
is the trajectory of the origin along the flow (as defined in Lemma
\ref{Allard's regularity}). It follows that 
\[
\inf_{B_{\vartheta r}\left(\gamma\left(t\right)\right)}N_{\Sigma_{t}}\cdot N_{\Sigma_{t}}\left(\gamma\left(t\right)\right)\geq\frac{1}{\sqrt{1+\varepsilon^{2}\delta^{2}}}\quad\forall\,\,0\leq t\leq T,
\]
which implies 
\[
\left|N_{\Sigma_{t}}-N_{\Sigma_{t}}\left(\gamma\left(t\right)\right)\right|^{2}=2\left[1-N_{\Sigma_{t}}\cdot N_{\Sigma_{t}}\left(\gamma\left(t\right)\right)\right]\leq2\left(1-\frac{1}{\sqrt{1+\varepsilon^{2}\delta^{2}}}\right)\quad\textrm{in}\,\,B_{\vartheta r}\left(\gamma\left(t\right)\right)
\]
for $0\leq t\leq T$. In addition, as noted in the proof of Lemma
\ref{Allard's regularity}, the mean value theorem implies 
\[
\left|\gamma\left(t\right)\right|\leq nt\leq nT\quad\forall\,\,0\leq t\leq T.
\]
So we have $B_{\vartheta r}\left(\gamma\left(t\right)\right)\supset B_{\left(1-\varepsilon\vartheta\right)\vartheta r}\left(O\right)$
for all $0\leq t\leq T$, provided that $T\leq\varepsilon^{2}\vartheta^{2}r^{2}$
and $\varepsilon\leq\frac{1}{n}$. 

On the other hand, let
\[
T_{*}=\sup\left\{ \tau\in\left[0,T\right]:\inf_{B_{\left(1-\varepsilon\vartheta\right)\vartheta r}\left(O\right)}N_{\Sigma_{t}}\cdot\mathbf{e}\geq\frac{1}{\sqrt{1+\delta^{2}}}\quad\forall\,\,0\leq t\leq\tau\right\} .
\]
Note that $T_{*}>0$ by continuity. We claim that $T_{*}=T$ under
the assumption that $0<\varepsilon\ll1$. Let us prove that by a contradiction
argument as follows. Assume that $T_{*}<T$. Note that 
\begin{equation}
\inf_{B_{\left(1-\varepsilon\vartheta\right)\vartheta r}\left(O\right)}N_{\Sigma_{T_{*}}}\cdot\mathbf{e}=\frac{1}{\sqrt{1+\delta^{2}}}.\label{maximal time of gradient estimate}
\end{equation}
Moreover, applying the maximum principle (cf. Chapter 2 in \cite{M})
to Eq. (\ref{AVP: evolution of localized gradient}) with the choice
$s=\left(1-\varepsilon\vartheta\right)\vartheta r,$ we get
\[
\sup_{B_{s}\left(O\right)}\eta_{s}\left(X,t\right)\left(N_{\Sigma_{t}}\cdot\mathbf{e}\right)^{-1}\leq\sup_{B_{s}\left(O\right)}\eta_{s}\left(X,0\right)\left(N_{\Sigma_{0}}\cdot\mathbf{e}\right)^{-1}\quad\,\forall\,\,0\leq t\leq T_{*},
\]
which implies
\[
N_{\Sigma_{t}}\left(\gamma\left(t\right)\right)\cdot\mathbf{e}\geq\frac{\left(1-\frac{\left|\gamma\left(t\right)\right|^{2}+2nt}{s^{2}}\right)^{3}}{\sqrt{1+\varepsilon^{2}\delta^{2}}}\geq\frac{\left(1-\frac{\left(2n+1\right)\varepsilon^{2}}{\left(1-\varepsilon\vartheta\right)^{2}}\right)^{3}}{\sqrt{1+\varepsilon^{2}\delta^{2}}}
\]
for all $0\leq t\leq T_{*}$. Consequently, we have 
\[
\inf_{B_{s}\left(O\right)}N_{\Sigma_{t}}\cdot\mathbf{e}\geq N_{\Sigma_{t}}\left(\gamma\left(t\right)\right)\cdot\mathbf{e}-\sup_{B_{s}\left(O\right)}\left|N_{\Sigma_{t}}-N_{\Sigma_{t}}\left(\gamma\left(t\right)\right)\right|
\]
\[
\geq\frac{\left(1-\frac{\left(2n+1\right)\varepsilon^{2}}{\left(1-\varepsilon\vartheta\right)^{2}}\right)^{3}}{\sqrt{1+\varepsilon^{2}\delta^{2}}}-\sqrt{2}\left(1-\frac{1}{\sqrt{1+\varepsilon^{2}\delta^{2}}}\right)^{\frac{1}{2}}\geq\frac{1}{\sqrt{1+\frac{\delta^{2}}{2}}}
\]
for all $0\leq t\leq T_{*}$, provided that $0<\varepsilon\ll1$ (depends
on $n,\delta$). This contradicts (\ref{maximal time of gradient estimate}). 

From the previous argument, most of the proposition follows by setting
\[
\theta=\min\left\{ \varepsilon\vartheta,\left(1-\varepsilon\vartheta\right)\vartheta,\frac{1}{2n\left(\sqrt{1+\delta^{2}}+1\right)}\right\} .
\]
However, there is one more thing that needs to be addressed, namely
\[
\Omega_{t}\supset B_{\frac{\theta r}{2\sqrt{1+\delta^{2}}}}\left(O\right)\cap\mathbb{R}^{n}.
\]
To see that, firstly note that 
\[
\left|\gamma\left(t\right)\right|\leq nt\leq nT\leq n\theta^{2}r^{2},
\]
so we have $B_{\theta r-n\theta^{2}r^{2}}\left(\gamma\left(t\right)\right)\subset B_{\theta r}\left(O\right)$.
Since $\Sigma_{t}\cap B_{\theta r-n\theta^{2}r^{2}}\left(\gamma\left(t\right)\right)$
is a $\delta$-Lipschitz graph passing through $\gamma\left(t\right)$,
it must be contained in the solid cone 
\[
\left|x_{n+1}-\gamma_{n+1}\left(t\right)\right|\leq\delta\sqrt{\left|x_{1}-\gamma_{1}\left(t\right)\right|^{2}+\cdots+\left|x_{n}-\gamma_{n}\left(t\right)\right|^{2}},
\]
where $x_{i}$ and $\gamma_{i}\left(t\right)$ are the $i^{\textrm{th}}$
coordinates of $X$ and $\gamma\left(t\right)$, respectively. Note
that the projection of the solid cone on $\mathbb{R}^{n}$ contains
the $n$-dimensional ball $B_{\frac{\theta r-n\theta^{2}r^{2}}{\sqrt{1+\delta^{2}}}}^{n}\left(\gamma_{1}\left(t\right),\cdots,\gamma_{n}\left(t\right)\right)$.
So we conclude that 
\[
\Omega_{t}\supset B_{\frac{\theta r-n\theta^{2}r^{2}}{\sqrt{1+\delta^{2}}}}^{n}\left(\gamma_{1}\left(t\right),\cdots,\gamma_{n}\left(t\right)\right)\supset B_{\frac{\theta r-n\theta^{2}r^{2}}{\sqrt{1+\delta^{2}}}-n\theta^{2}r^{2}}^{n}\left(0,\cdots,0\right).
\]
Note that 
\[
\frac{\theta r-n\theta^{2}r^{2}}{\sqrt{1+\delta^{2}}}-n\theta^{2}r^{2}\geq\left(\frac{1-n\theta}{\sqrt{1+\delta^{2}}}-n\theta\right)\theta r\geq\frac{\theta r}{2\sqrt{1+\delta^{2}}}
\]
since $0<r\leq1$ and $0<\theta\leq\frac{1}{2n\left(\sqrt{1+\delta^{2}}+1\right)}$.
\end{proof}
The following lemma shows that for a small Lipschitz graph (over a
hyperplane), the modulus of the derivatives of the graph are equivalent
to that of the covariant derivatives of its curvatures.
\begin{lem}
\label{equivalence of smooth estimates}There exists a constant $\delta>0$
depending on $n$ with the following property.

Let $\Sigma=\left\{ X=\left(x,u\left(x\right)\right):x\in\Omega\subset\mathbb{R}^{n}\right\} $
be a graph of $u$ with $\left\Vert \partial_{x}u\right\Vert _{L^{\infty}}\leq2\delta$.
Then for every $k\in\mathbb{N}$ we have
\[
\left\Vert \partial_{x}^{k+1}u\right\Vert _{L^{\infty}}\leq C\left(n,k,\left\Vert A_{\Sigma}\right\Vert _{L^{\infty}},\cdots,\left\Vert \nabla_{\Sigma}^{k-1}A_{\Sigma}\right\Vert _{L^{\infty}}\right)
\]
and
\[
\left\Vert \nabla_{\Sigma}^{k-1}A_{\Sigma}\right\Vert _{L^{\infty}}\leq C\left(n,k,\left\Vert \partial_{x}^{2}u\right\Vert _{L^{\infty}},\cdots,\left\Vert \partial_{x}^{k+1}u\right\Vert _{L^{\infty}}\right).
\]
\end{lem}

\begin{proof}
In the first place, recall that the pull-back metric of $\Sigma$
is given by 
\begin{equation}
g_{ij}=\delta_{ij}+\partial_{i}u\,\partial_{j}u,\label{ESE: metric}
\end{equation}
which is equivalent to the Euclidean metric $\delta_{ij}$ provided
that $0<\delta\ll1$ (depending on $n$). It follows that the norms
defined by these two metrics, $g_{ij}$ and $\delta_{ij}$, are equivalent.
For instance, the norm of the second fundamental form
\[
\left|A_{\Sigma}\right|^{2}=\sum_{i,j,k,l=1}^{n}g^{ik}g^{jl}A_{ij}A_{kl}
\]
is equivalent to the $l^{2}$ norm of its coordinate matrix $\left(A_{ij}\right)$,
i.e.

\[
\frac{1}{C\left(n\right)}\sum_{i,j=1}^{n}A_{ij}^{2}\leq\,\left|A_{\Sigma}\right|^{2}\,\leq C\left(n\right)\sum_{i,j=1}^{n}A_{ij}^{2}.
\]
As a consequence, we do not have to distinguish between the (Riemannian)
norm of a tensor with the $l^{2}$ norm of its coordinates.

For $k=1$, note that the coordinates of the second fundamental form
$A_{\Sigma}$ is given by 
\begin{equation}
A_{ij}=\frac{\partial_{ij}u}{\sqrt{1+\left|\partial_{x}u\right|^{2}}},\label{ESE: second fundamental form}
\end{equation}
from which one can easily see that the estimates hold for $k=1$.

For $k=2$, note that the Christoffel symbols are given by 
\begin{equation}
\Gamma_{ij}^{k}=\frac{\partial_{k}u\,\partial_{ij}u}{1+\left|\partial_{x}u\right|^{2}}.\label{ESE: Christoffel symbols}
\end{equation}
The coordinates of the first covariant derivative of the second fundamental
form $\nabla_{\Sigma}A_{\Sigma}$ are given by

\[
\nabla_{k}A_{ij}=\partial_{k}A_{ij}-\Gamma_{ki}^{l}A_{lj}-\Gamma_{kj}^{l}A_{il}
\]
with 
\[
\partial_{k}A_{ij}=\frac{\partial_{ijk}u}{\sqrt{1+\left|\partial_{x}u\right|^{2}}}-\frac{\partial_{l}u\,\partial_{ij}u\,\partial_{kl}u}{\left(1+\left|\partial_{x}u\right|^{2}\right)^{\frac{3}{2}}}.
\]
The estimates for $k=2$ follow from the above expressions and the
estimates for $k=1$.

Other cases are omitted since they can be derived in a similar fashion.
\end{proof}
We are now in a position to prove the preservation of the asymptotically
conical property along MCF. Since the initial hypersurface is assumed
to be asymptotic to a cone, all we need to do is to show that the
later time-slices stay close to the initial hypersurface with respect
to the corresponding scale (depending on the location). This will
be done through the local graph parametrization as we will be able
to use Eq. (\ref{ACM: equation}). 
\begin{prop}
\label{asymptotically conical along MCF}Let $\left\{ \Sigma_{t}\right\} _{0\leq t\leq T}$
be a MCF in $\mathbb{R}^{n+1}$, where $T>0$ is a constant, so that
$\Sigma_{0}$ is asymptotic to a regular cone $\mathcal{C}$ at infinity.
Then for every $0<t\leq T$, $\Sigma_{t}$ is asymptotic to $\mathcal{C}$
at infinity as well. 
\end{prop}

\begin{proof}
\textbf{$\left\langle \right.$Step 1$\left.\right\rangle $} In this
step, let us make the following observation, which will be the heart
of the matter of the proof. 

Fix $k\in\mathbb{N}$. By Proposition \ref{spatial curvature estimate}
and Corollary \ref{spatial smooth estimates}, there exist constants
$\Lambda,\mathcal{R},K\geq1$ (depending on $n,k,\mathcal{C},\Sigma_{0}$)
so that 
\begin{equation}
\sup_{0\leq t\leq T}\left\Vert \,\left|X\right|^{i}\nabla_{\Sigma_{t}}^{i-1}A_{\Sigma_{t}}\right\Vert _{L^{\infty}\left(\mathbb{R}^{n+1}\setminus B_{\max\left\{ \mathcal{R},\Lambda\sqrt{t}\right\} }\left(O\right)\right)}\leq K\quad\,\forall\,\,i=1,\cdots,k.\label{ACM: pre curvature}
\end{equation}
For $R>0$, consider the MCF 
\begin{equation}
\Sigma_{\tau}^{R}=\frac{1}{R}\Sigma_{R^{2}\tau},\quad0\leq\tau\leq\frac{T}{R^{2}}.\label{ACM: rescaling}
\end{equation}
Let $0<\delta<1$ be the constant in Lemma \ref{equivalence of smooth estimates}.
With this choice of $\delta$, let $0<\theta<1$ be the corresponding
constant in Proposition \ref{analyticity: vertical parametrization}.
By the second condition in Definition \ref{regular cone}, we can
find $0<\rho\leq\frac{1}{4K}$ so that for every $Q\in\partial B_{1}\left(O\right)$,
$\mathcal{C}\cap B_{2\rho}\left(Q\right)$ is either empty or a $\frac{\theta}{2}$-Lipschitz
graph (over a hyperplane). Now given $P\in\mathbb{R}^{n+1}\setminus\left\{ O\right\} $,
there are three cases to consider.\textbf{ }

\textbf{\textit{Case 1:}} $P\in\mathcal{C}$. 

If $\left|P\right|R\geq2\max\left\{ \mathcal{R},\Lambda\sqrt{T}\right\} $,
then (\ref{ACM: pre curvature}) implies
\begin{equation}
\sup_{0\leq\tau\leq\frac{T}{R^{2}}}\left|P\right|^{i}\left\Vert \nabla_{\Sigma_{\tau}^{R}}^{i-1}A_{\Sigma_{\tau}^{R}}\right\Vert _{L^{\infty}\left(B_{2\rho\left|P\right|}\left(P\right)\right)}\leq\frac{K}{\left(1-2\rho\right)^{i}}\leq2^{i}K\quad\,\forall\,\,i=1,\cdots,k.\label{ACM: curvature}
\end{equation}
In particular, we have 
\begin{equation}
\sup_{0\leq\tau\leq\frac{T}{R^{2}}}\rho\left|P\right|\left\Vert H_{\Sigma_{\tau}^{R}}\right\Vert _{L^{\infty}\left(B_{2\rho\left|P\right|}\left(P\right)\right)}\leq2\sqrt{n}\rho K\leq\sqrt{n}.\label{ACM: mean curvature}
\end{equation}
Also, by rescaling, $\mathcal{C}\cap B_{2\rho\left|P\right|}\left(P\right)$
is a $\frac{\theta}{2}$-Lipschitz graph over $T_{P}\mathcal{C}\simeq\mathbb{R}^{n}$.
Thus, using the second condition in Definition \ref{asymptotically conical}
that
\begin{equation}
\Sigma_{0}^{R}=\frac{1}{R}\Sigma_{0}\stackrel{C^{1}}{\longrightarrow}\mathcal{C}\quad\textrm{in}\,\,B_{\left(1+2\rho\right)\left|P\right|}\left(O\right)\setminus B_{\left(1-2\rho\right)\left|P\right|}\left(O\right)\quad\textrm{as}\,\,R\rightarrow\infty,\label{ACM: convergence}
\end{equation}
the estimates (\ref{ACM: mean curvature}), and by applying Proposition
\ref{analyticity: vertical parametrization} to a translation of the
flow
\begin{equation}
\left\{ \frac{1}{\rho\left|P\right|}\left(\Sigma_{\rho^{2}\left|P\right|^{2}s}^{R}-P\right)\right\} _{0\leq s\leq\min\left\{ \theta^{2},\frac{T}{\rho^{2}\left|P\right|^{2}R^{2}}\right\} },\label{ACM: normalization}
\end{equation}
we can find a time-dependent domain 
\[
B_{\frac{\theta\rho\left|P\right|}{4\sqrt{1+\delta^{2}}}}^{n}\left(O\right)\subset\Omega_{\tau}^{P,R}\subset\mathbb{R}^{n}\simeq T_{P}\mathcal{C}
\]
and a time-dependent function $u^{P,R}\left(\cdot,\tau\right):\Omega_{\tau}^{P,R}\rightarrow\mathbb{R}$
so that 
\begin{equation}
\left\Vert \partial_{x}u^{P,R}\left(\cdot,\tau\right)\right\Vert _{L^{\infty}\left(\Omega_{\tau}\right)}\leq2\delta\label{ACM: gradient}
\end{equation}
and 
\begin{equation}
\Sigma_{\tau}^{R}\cap B_{\frac{1}{2}\theta\rho\left|P\right|}\left(P\right)=\left\{ X=P+\left(x,u^{P,R}\left(x,\tau\right)\right):x\in\Omega_{\tau}^{P,R}\right\} \label{ACM: pre graph}
\end{equation}
for $0\leq\tau\leq\min\left\{ \frac{1}{4}\theta^{2}\rho^{2}\left|P\right|^{2},\frac{T}{R^{2}}\right\} =\frac{T}{R^{2}}$,
provided that $\left|P\right|R\geq\frac{2\sqrt{T}}{\theta\rho}$.
Furthermore, applying Lemma \ref{equivalence of smooth estimates}
to each time-slice of the flow (\ref{ACM: normalization}) and using
the estimates (\ref{ACM: curvature}) and (\ref{ACM: gradient}),
we obtain
\begin{equation}
\sup_{0\leq\tau\leq\frac{T}{R^{2}}}\left(\rho\left|P\right|\right)^{i}\left\Vert \partial_{x}^{i+1}u^{P,R}\left(\cdot,\tau\right)\right\Vert _{L^{\infty}\left(B_{\frac{\theta\rho\left|P\right|}{4\sqrt{1+\delta^{2}}}}^{n}\left(O\right)\right)}\leq C\left(n,k\right)\quad\,\forall\,\,i=1,\cdots,k.\label{ACM: bounds}
\end{equation}
It follows from (\ref{ACM: gradient}), (\ref{ACM: bounds}), and
the equation of graphical MCF 
\begin{equation}
\partial_{\tau}u^{P,R}=\left(\delta_{ij}-\frac{\partial_{i}u^{P,R}\,\partial_{j}u^{P,R}}{1+\left|\nabla u^{P,R}\right|^{2}}\right)\partial_{ij}u^{P,R}\label{ACM: equation}
\end{equation}
(cf. Chapter 2 in \cite{E}) that 
\[
\sup_{0\leq\tau\leq\frac{T}{R^{2}}}\left\Vert u^{R}\left(\cdot,\tau\right)-u^{R}\left(\cdot,0\right)\right\Vert _{L^{\infty}\left(B_{\frac{\theta\rho\left|P\right|}{4\sqrt{1+\delta^{2}}}}^{n}\left(O\right)\right)}\leq C\left(n,k,\mathcal{C},\Sigma_{0}\right)\frac{T}{\left|P\right|R^{2}}.
\]
More generally, differentiating Eq. (\ref{ACM: equation}) with respect
to $x$ a number of times and using the estimates (\ref{ACM: gradient})
and (\ref{ACM: bounds}) yields
\begin{equation}
\sup_{0\leq\tau\leq\frac{T}{R^{2}}}\left\Vert \partial_{x}^{i}u^{P,R}\left(\cdot,\tau\right)-\partial_{x}^{i}u^{P,R}\left(\cdot,0\right)\right\Vert _{L^{\infty}\left(B_{\frac{\theta\rho\left|P\right|}{4\sqrt{1+\delta^{2}}}}^{n}\left(O\right)\right)}\leq C\left(n,k,\mathcal{C},\Sigma_{0}\right)\frac{T}{\left|P\right|^{1+i}R^{2}}\label{ACM: smooth estimates}
\end{equation}
for all $i\in\left\{ 0,\cdots,k-1\right\} $ so long as $\left|P\right|R\geq C\left(n,k,\mathcal{C},\Sigma_{0}\right)\left(\sqrt{T}+1\right)$.
Moreover, we have 
\begin{equation}
\Sigma_{\tau}^{R}\cap B_{\frac{\theta\rho\left|P\right|}{4\sqrt{1+\delta^{2}}}}\left(P\right)=\left\{ X=P+\left(x,u^{P,R}\left(x,\tau\right)\right):x\in\Omega_{\tau}^{P,R}\right\} \cap B_{\frac{\theta\rho\left|P\right|}{4\sqrt{1+\delta^{2}}}}\left(P\right)\label{ACM: graph}
\end{equation}
by (\ref{ACM: pre graph}). 

\textbf{\textit{Case 2:}} $\textrm{dist}\left(P,\mathcal{C}\right)\leq\varrho\left|P\right|$,
where
\begin{equation}
\varrho=\frac{\frac{\theta\rho}{4\sqrt{1+\delta^{2}}}}{1+\frac{\theta\rho\left|P\right|}{4\sqrt{1+\delta^{2}}}}.\label{ACM: radius}
\end{equation}

In fact, we can somehow reduce this case to Case 1. To see this, choose
$\tilde{P}\in\mathcal{C}$ so that $\textrm{dist}\left(P,\mathcal{C}\right)=\left|P-\tilde{P}\right|$.
Then we have 
\[
\left|\tilde{P}\right|\leq\left|P\right|+\left|P-\tilde{P}\right|\leq\left(1+\varrho\right)\left|P\right|
\]
and 
\[
\left|\tilde{P}\right|\geq\left|P\right|-\left|P-\tilde{P}\right|\geq\left(1-\varrho\right)\left|P\right|.
\]
It then follows from (\ref{ACM: radius}) that 
\[
\left|P-\tilde{P}\right|=\textrm{dist}\left(P,\mathcal{C}\right)\leq\varrho\left|P\right|=\frac{\theta\rho}{4\sqrt{1+\delta^{2}}}\left(1-\varrho\right)\left|P\right|\leq\frac{\theta\rho}{4\sqrt{1+\delta^{2}}}\left|\tilde{P}\right|.
\]
Thereby we get $P\in B_{\frac{\theta\rho\left|\tilde{P}\right|}{4\sqrt{1+\delta^{2}}}}\left(\tilde{P}\right)$,
in which (\ref{ACM: smooth estimates}) and (\ref{ACM: graph}) hold
(with the point $P$ therein replaced by $\tilde{P}$). 

\textbf{\textit{Case 3:}} $\textrm{dist}\left(P,\mathcal{C}\right)>\varrho\left|P\right|$. 

Note that $\textrm{dist}\left(RP,\mathcal{C}\right)>\varrho\left|P\right|R$
for any $R>0$. By the second condition in Definition \ref{asymptotically conical},
if $\left|P\right|R\gg1$ (depending on $n,k,\mathcal{C},\Sigma_{0}$),
we have 
\[
\Sigma_{0}\cap B_{\frac{5}{6}\varrho\left|P\right|R}\left(RP\right)=\emptyset\,\,\Rightarrow\,\,\frac{1}{R}\Sigma_{0}\cap B_{\frac{5}{6}\varrho\left|P\right|}\left(P\right)=\emptyset.
\]
Applying the avoidance principle for MCF (cf. Chapter 3 in \cite{E})
gives
\[
\Sigma_{\tau}^{R}\cap\partial B_{\sqrt{\left(\frac{2}{3}\varrho\left|P\right|\right)^{2}-2n\tau}}\left(P\right)=\emptyset\quad\,\,\forall\,\,0\leq\tau\leq\min\left\{ \frac{\left(\frac{2}{3}\varrho\left|P\right|\right)^{2}}{2n},\frac{T}{R^{2}}\right\} ,
\]
which implies
\begin{equation}
\Sigma_{\tau}^{R}\cap B_{\frac{1}{2}\varrho\left|P\right|}\left(P\right)=\emptyset\label{ACM: disjoint}
\end{equation}
for $0\leq\tau\leq\min\left\{ \frac{7\left(\varrho\left|P\right|\right)^{2}}{72n},\frac{T}{R^{2}}\right\} =\frac{T}{R^{2}}$,
provided that $\left|P\right|R\geq\sqrt{\frac{72nT}{7\varrho^{2}}}$. 

\textbf{$\left\langle \right.$Step 2$\left.\right\rangle $} Now
let us see how do we use the observation from Step 1 to complete the
proof.

To prove the proposition, it suffices to show that for any given compact
set $\mathscr{K}\subset\mathbb{R}^{n+1}\setminus\left\{ O\right\} $,
$k\in\mathbb{N}$, and $\varepsilon>0$, there exist $\mathscr{R}_{1}>0$
(depending on $n$, $k$, $\mathcal{C}$, $\Sigma_{0}$, $\mathscr{K}$,
$T$, $\varepsilon$) and $\mathscr{R}_{2}>0$ (depending on $n$,
$\mathcal{C}$, $\Sigma_{0}$, $T$, $\varepsilon$) so that for every
$0\leq t\leq T$, the following two properties hold:
\begin{enumerate}
\item $\frac{1}{R}\Sigma_{t}$ is $\varepsilon$-close, in the $C^{k-1}$
topology, to $\frac{1}{R}\Sigma_{0}$ in $\mathscr{K}$ for every
$R\geq\mathscr{R}_{1}$;
\item $\Sigma_{t}$ is $\varepsilon$-close, in the $C^{1}$ topology, to
$\Sigma_{0}$ in $\mathbb{R}^{n+1}\setminus B_{\mathscr{R}_{2}}\left(O\right)$. 
\end{enumerate}
To verify property (1), let us first find a countable covering for
$\mathscr{K}$ in the form
\begin{equation}
\bigcup_{i}B_{r_{i}}\left(P_{i}\right)\quad\textrm{with}\,\,\,r_{i}=\left\{ \begin{array}{c}
\frac{\theta\rho\left|P_{i}\right|}{4\sqrt{1+\delta^{2}}},\quad\textrm{if}\,\,P_{i}\in\mathcal{C}\\
\frac{1}{2}\varrho\left|P_{i}\right|,\quad\textrm{if}\,\,\textrm{dist}\left(P_{i},\mathcal{C}\right)>\varrho\left|P_{i}\right|
\end{array}\right.,\label{ACM: covering}
\end{equation}
where all the constants are from Step 1 The existence of such a covering
is ensured by the argument in Step 1 (since Case 2 and Case 3 include
all possibilities and Case 2 can be ``reduced'' to Case 1). It is
not hard to see that property (1) follows from (\ref{ACM: rescaling}),
(\ref{ACM: smooth estimates}), (\ref{ACM: graph}), and (\ref{ACM: disjoint}).

The verification of property (2) basically follows from the same line
of argument as in verifying property (1), with only a slight modification.
We begin by finding a countable covering in the form (\ref{ACM: covering})
for $\mathbb{R}^{n+1}\setminus B_{R_{0}}\left(O\right)$, where $R_{0}>0$
is the (or possibly larger) constant from Definition \ref{asymptotically conical}.
Then setting $k=2$ and $R=1$ in Step 1 and replacing (\ref{ACM: convergence})
by the second condition in Definition \ref{asymptotically conical},
the same argument carries over, under the assumption that $\left|P\right|\geq\mathscr{R}_{2}$,
and leads to the following results for the three cases, from which
the conclusion follows easily. 

\textbf{\textit{Case 1:}} If $P\in\mathcal{C}$, then for every $0\leq t\leq T$,
$\Sigma_{t}\cap B_{\frac{\theta\rho\left|P\right|}{4\sqrt{1+\delta^{2}}}}\left(P\right)$
can be parametrized as a graph of $u^{P}\left(\cdot,t\right)$, i.e.
\begin{equation}
\Sigma_{t}\cap B_{\frac{\theta\rho\left|P\right|}{4\sqrt{1+\delta^{2}}}}\left(P\right)=\left\{ X=P+\left(x,u^{P}\left(x,t\right)\right):x\in\Omega_{t}^{P}\right\} \cap B_{\frac{\theta\rho\left|P\right|}{4\sqrt{1+\delta^{2}}}}\left(P\right),\label{ACM: graph'}
\end{equation}
with $u^{P}\left(\cdot,t\right)$ satisfying 
\begin{equation}
\sup_{0\leq t\leq T}\left\Vert \partial_{x}^{i}u^{P}\left(\cdot,t\right)-\partial_{x}^{i}u^{P}\left(\cdot,0\right)\right\Vert _{L^{\infty}\left(B_{\frac{\theta\rho\left|P\right|}{4\sqrt{1+\delta^{2}}}}^{n}\left(O\right)\right)}\leq C\left(n,\mathcal{C},\Sigma_{0}\right)\frac{T}{\left|P\right|^{1+i}}\label{ACM: smooth estimates'}
\end{equation}
for $i\in\left\{ 0,1\right\} $, as long as $\left|P\right|\geq C\left(n,\mathcal{C},\Sigma_{0}\right)\left(\sqrt{T}+1\right)$. 

\textbf{\textit{Case 2: }}If $\textrm{dist}\left(P,\mathcal{C}\right)\leq\varrho\left|P\right|$,
then we can find $\tilde{P}\in\mathcal{C}$ so that $P\in B_{\frac{\theta\rho\left|\tilde{P}\right|}{4\sqrt{1+\delta^{2}}}}\left(\tilde{P}\right)$.

\textbf{\textit{Case 3:}} If $\textrm{dist}\left(P,\mathcal{C}\right)>\varrho\left|P\right|$,
then 
\begin{equation}
\Sigma_{t}\cap B_{\frac{1}{2}\varrho\left|P\right|}\left(P\right)=\emptyset\label{ACM: disjoint'}
\end{equation}
for all $0\leq t\leq T$, provided that $\left|P\right|\geq C\left(n,\mathcal{C},\Sigma_{0}\right)\left(\sqrt{T}+1\right)$. 
\end{proof}
Below we have a further remark regarding Proposition \ref{asymptotically conical along MCF},
which will be used in proving Theorem \ref{blow-down of MCF}.
\begin{rem}
\label{ACM: graphic along MCF}By examining the second condition in
Definition \ref{asymptotically conical} and the verification of property
(2) in the preceding proof (with a focus on conditions (\ref{ACM: graph'}),
(\ref{ACM: smooth estimates'}), and (\ref{ACM: disjoint'})), it
can be observed that the following property also holds: 

\medskip{}

\textit{Under the hypothesis of Proposition \ref{asymptotically conical along MCF}
with an extra condition that $T\geq1$, given $\varepsilon>0$, there
exists a constant $M\geq1$ depending on $n$, $\mathcal{C}$, $\Sigma_{0}$
and $\varepsilon$ so that for every $0\leq t\leq T$, $\Sigma_{t}\setminus B_{M\sqrt{T}}\left(O\right)$
is normal graph of  $u_{t}$ }over\textit{ $\mathcal{C}$ outside
a compact subset with 
\[
\left\Vert \nabla_{\mathcal{C}}\,u_{t}\right\Vert _{L^{\infty}}+\frac{1}{\sqrt{T}}\left\Vert u_{t}\right\Vert _{L^{\infty}}\leq\varepsilon.
\]
Particularly, it follows that $\frac{1}{\sqrt{T}}\Sigma_{T}\setminus B_{M}\left(O\right)$
is normal graph of 
\[
\omega_{T}\left(Y\right)=\frac{1}{\sqrt{T}}\,u_{T}\left(\sqrt{T}\,Y\right)
\]
}over\textit{ $\mathcal{C}$ outside a compact subset with 
\[
\left\Vert \nabla_{\mathcal{C}}\,\omega_{T}\right\Vert _{L^{\infty}}+\left\Vert \omega_{T}\right\Vert _{L^{\infty}}\leq\varepsilon.
\]
\smallskip{}
}
\end{rem}

Though a regular cone $\mathcal{C}$ has a singularity at the tip
$O$, the argument in proving Proposition \ref{asymptotically conical along MCF}
 still carries over to a self-expanding MCF coming out of $\mathcal{C}$,
in which case the initial hypersurface is trivially asymptotic to
a cone at infinity.
\begin{cor}
\label{self-expanding MCF out of cone}Suppose that the self-expanding
MCF $\left\{ \Gamma_{\tau}\right\} $ is smooth in the space-time
$\left(\mathbb{R}^{n+1}\times\left[0,\infty\right)\right)\setminus\left\{ \left(O,0\right)\right\} $
with $\Gamma_{0}=\mathcal{C}$, where $\mathcal{C}$ is a regular
cone. Then $\Gamma_{\tau}$ is asymptotic to $\mathcal{C}$ at infinity
for every $\tau>0$. 
\end{cor}

The rest of the section is devoted to proving the curvature estimate
in Theorem \ref{decay rate of curvature}. The estimate (\ref{NG: Radon=002013Nikodym derivative})
in the following lemma is necessary for the proof; other parts of
the lemma will be useful in Sections \ref{Approaching property} and
\ref{Asymptotic self-similarity}. 
\begin{lem}
\label{normal graph}There exists a constant $0<\varsigma<1$ depending
on $n$ with the following property.

Let $\Sigma$ and $\tilde{\Sigma}$ be two smooth, properly embedded,
and oriented hypersurfaces in $\mathbb{R}^{n+1}$ so that $\tilde{\Sigma}$
is a normal graph of $v$ over $\Sigma$, i.e.
\[
\tilde{\Sigma}=\left\{ \tilde{X}=X+v\,N_{\Sigma}:X\in\Sigma\right\} ,
\]
with
\[
\left\Vert \nabla_{\Sigma}v\right\Vert _{L^{\infty}}+\left\Vert A_{\Sigma}v\right\Vert _{L^{\infty}}\leq\varsigma.
\]
Then we have the following estimates for the Radon\textendash Nikodym
derivative
\begin{equation}
\frac{d\mathcal{H}^{n}\lfloor\tilde{\Sigma}}{d\mathcal{H}^{n}\lfloor\Sigma}\leq1+C\left(n\right)\left(\left|\nabla_{\Sigma}v\right|+\left|A_{\Sigma}v\right|\right)\label{NG: Radon=002013Nikodym derivative}
\end{equation}
and the second fundamental form 
\begin{equation}
\left|A_{\tilde{\Sigma}}\right|\,\leq\,\left(\left|A_{\Sigma}\right|+\left|\nabla_{\Sigma}^{2}v\right|\right)\,\left[\,1\,+\,C\left(n\right)\left(\left|\nabla_{\Sigma}v\right|+\left|A_{\Sigma}v\right|\right)\,\right]\label{NG: curvature estimate}
\end{equation}
\[
+\,C\left(n\right)\,\left|v\nabla_{\Sigma}A_{\Sigma}\right|\,\left(\left|\nabla_{\Sigma}v\right|+\left|A_{\Sigma}v\right|\right).
\]
Moreover, we have the following formula relating the mean curvatures
of $\tilde{\Sigma}$ and $\Sigma$:
\begin{equation}
\left(N_{\Sigma}\cdot N_{\tilde{\Sigma}}\right)^{-1}\left[H_{\tilde{\Sigma}}+v\nabla_{\Sigma}H_{\Sigma}\cdot N_{\tilde{\Sigma}}\right]\,-\,H_{\Sigma}\label{NG: biased mean curvature}
\end{equation}
\[
=a\cdot\nabla_{\Sigma}^{2}v+\left|A_{\Sigma}\right|^{2}v
\]
\[
+\,v\nabla_{\Sigma}A_{\Sigma}\ast\nabla_{\Sigma}v+A_{\Sigma}\ast Q\left(\nabla_{\Sigma}v,A_{\Sigma}v\right)+v\nabla_{\Sigma}A_{\Sigma}\ast Q\left(\nabla_{\Sigma}v,A_{\Sigma}v\right)
\]
where $a$ is a 2-tensor defined by
\[
a^{ij}=g^{ij}+2A^{ij}v+Q^{ij}\left(\nabla_{\Sigma}v,A_{\Sigma}v\right)
\]
and satisfying
\begin{equation}
\frac{2}{3}g^{ij}\leq a^{ij}\leq\frac{4}{3}g^{ij}.\label{NG: elliptic}
\end{equation}
Note that 
\begin{itemize}
\item The notation $Q$ means an analytic function/tensor that is at least
``quadratic'' (in the form of contraction via the metric $g_{ij}$
of $\Sigma$ and its inverse $g^{ij}$) in its arguments.
\item The notation $\ast$ means some form of contraction of tensors.
\item $A^{ij}$ denotes raising the indices of $A_{ij}$ , where $A_{ij}$
are the coordinates of $A_{\Sigma}$.
\item $a\cdot\nabla_{\Sigma}^{2}v=a^{ij}\nabla_{ij}v$.
\end{itemize}
\end{lem}

\begin{proof}
First of all, by a simple calculation, the pull-back metrics of $\tilde{\Sigma}$
and $\Sigma$ are related by
\[
\tilde{g}_{ij}=g_{ij}-2A_{ij}v+A_{ij}^{2}v^{2}+\nabla_{i}v\nabla_{j}v,
\]
where $\tilde{g}_{ij}$ and $g_{ij}$ are the metrics of $\tilde{\Sigma}$
and $\Sigma$, respectively. In particular, we get
\[
\textrm{det}\left(\tilde{g}_{ij}\right)=\textrm{det}\left(g_{ij}\right)\cdot\textrm{det}\left(\delta_{j}^{i}-2A_{j}^{i}v+A_{k}^{i}A_{j}^{k}v^{2}+\nabla^{i}v\nabla_{j}v\right).
\]
Thus, if $\left|\nabla v\right|+\left|Av\right|\ll1$ (depending on
$n$), using the Taylor expansion we obtain 
\[
\sqrt{\textrm{det}\,\tilde{g}}=\sqrt{\textrm{det}\,g}\,\left[1-Hv+Q\left(\nabla v,Av\right)\right],
\]
where $H=H_{\Sigma}$ is the mean curvature of $\Sigma$. Thereby
we get
\[
\frac{d\mathcal{H}^{n}\lfloor\tilde{\Sigma}}{d\mathcal{H}^{n}\lfloor\Sigma}=\frac{\sqrt{\textrm{det}\,\tilde{g}}}{\sqrt{\textrm{det}\,g}}\leq1+C\left(n\right)\left(\left|\nabla v\right|+\left|Av\right|\right).
\]
Next, note that the second fundamental forms of $\tilde{\Sigma}$
and $\Sigma$ are related by 
\[
\tilde{A}_{ij}=\left(A_{ij}+\nabla_{ij}v-A_{ij}^{2}v\right)N\cdot\tilde{N}-\left(A_{i}^{k}\nabla_{j}v+A_{j}^{k}\nabla_{i}v+v\nabla_{i}A_{j}^{k}\right)e_{k}\cdot\tilde{N},
\]
where $\tilde{A}_{ij}$ are the coordinates of $A_{\tilde{\Sigma}}$,
$N$ and $\tilde{N}$ are the unit normal vectors of $\Sigma$ and
$\tilde{\Sigma}$, respectively, and $\left\{ e_{1},\cdots,e_{n}\right\} $
is a coordinate frame in $\Sigma$. In addition, the Taylor expansion
of $\tilde{g}^{ij}$, under the assumption that $\left|\nabla v\right|+\left|Av\right|\ll1$
(depending on $n$), is given by 
\[
\tilde{g}^{ij}=g^{ij}+2A^{ij}v+Q^{ij}\left(\nabla v,Av\right).
\]
It follows that
\begin{equation}
\tilde{A}_{j}^{i}=\tilde{g}^{ik}\tilde{A}_{kj}=\left\{ A_{j}^{i}+\nabla_{j}^{i}v+A_{k}^{i}A_{j}^{k}v+\left[2A^{ik}v+Q^{ik}\left(\nabla v,Av\right)\right]\nabla_{kj}v\right\} N\cdot\tilde{N}\label{NG: shape operator}
\end{equation}
\[
-\left\{ 2A^{ik}A_{kj}^{2}v^{2}+Q^{ik}\left(\nabla v,Av\right)\left[A_{kj}-A_{kj}^{2}v\right]\right\} N\cdot\tilde{N}
\]
\[
-\left[g^{ik}+2A^{ik}v+Q^{ik}\left(\nabla v,Av\right)\right]\left(A_{k}^{l}\nabla_{j}v+A_{j}^{l}\nabla_{k}v+v\nabla_{k}A_{j}^{l}\right)e_{l}\cdot\tilde{N}.
\]
Note that
\begin{equation}
N\cdot\tilde{N}=1+Q\left(\nabla v,Av\right),\quad e_{k}\cdot\tilde{N}=-\nabla_{k}v\,\left[1+Q\left(\nabla v,Av\right)\right].\label{NG: unit normal}
\end{equation}
Using (\ref{NG: shape operator}) and (\ref{NG: unit normal}), we
then get (\ref{NG: curvature estimate}) and that
\[
\left(N\cdot\tilde{N}\right)^{-1}\tilde{H}
\]
\[
=H+\triangle v+\left|A\right|^{2}v+\left[2A^{ij}v+Q^{ij}\left(\nabla v,Av\right)\right]\nabla_{ij}v
\]
\[
+v\nabla A\ast\nabla v+A\ast Q\left(\nabla v,Av\right)+v\nabla A\ast Q\left(\nabla v,Av\right),
\]
from which (\ref{NG: biased mean curvature}) follows. Note also that
(\ref{NG: elliptic}) holds if $\left|\nabla v\right|+\left|Av\right|\ll1$
(depending on $n$). 
\end{proof}
Now we are ready to prove the curvature estimate in the following
theorem. The idea primarily rests on White's regularity theorem along
with the pseudolocality theorem in the forms of Lemma \ref{temporal curvature estimate}
and Proposition \ref{spatial curvature estimate}, respectively. 
\begin{thm}
\label{decay rate of curvature}Given $\kappa>0$, if $\left\{ \Sigma_{t}\right\} _{0\leq t<\infty}$
is a MCF in $\mathbb{R}^{n+1}$ and if $\Sigma_{0}$ is asymptotic
to a regular cone $\mathcal{C}$ at infinity with $E\left[\mathcal{C}\right]<1+\epsilon$,
where $\epsilon$ is the constant in Theorem \ref{refinement of White's},
then we have 
\[
\sup_{t\geq T}\,\sqrt{t}\left\Vert A_{\Sigma_{t}}\right\Vert _{L^{\infty}}\leq\kappa
\]
for some constant $T>0$ that depends on $n$, $\kappa$, $\mathcal{C}$,
and $\Sigma_{0}$. 
\end{thm}

\begin{proof}
By Proposition \ref{spatial curvature estimate}, there exist constants
$\Lambda,\mathcal{R},\mathcal{K}\geq1$ (depending on $n$, $\mathcal{C}$,
$\Sigma_{0}$) so that 
\begin{equation}
\sup_{t\geq0}\left\Vert \,\left|X\right|A_{\Sigma_{t}}\right\Vert _{L^{\infty}\left(\mathbb{R}^{n+1}\setminus B_{\max\left\{ \mathcal{R},\Lambda\sqrt{t}\right\} }\left(O\right)\right)}\leq\mathcal{K}.\label{DRC: pseudolocality}
\end{equation}
Given $\kappa>0$, let $\tilde{\Lambda}=\max\left\{ \Lambda,\frac{\mathcal{K}}{\kappa},\frac{M}{\sqrt{2}-1}\right\} $,
where $M$ is the constant in Theorem \ref{refinement of White's}.
Then (\ref{DRC: pseudolocality}) implies 
\[
\tilde{\Lambda}\sqrt{t}\left\Vert A_{\Sigma_{t}}\right\Vert _{L^{\infty}\left(\mathbb{R}^{n+1}\setminus B_{\max\left\{ \mathcal{R},\tilde{\Lambda}\sqrt{t}\right\} }\left(O\right)\right)}\leq\mathcal{K}
\]
for all $t\geq0$; especially, we obtain
\[
\sup_{t\geq\left(\frac{\mathcal{R}}{\tilde{\Lambda}}\right)^{2}}\,\sqrt{t}\left\Vert A_{\Sigma_{t}}\right\Vert _{L^{\infty}\left(\mathbb{R}^{n+1}\setminus B_{\tilde{\Lambda}\sqrt{t}}\left(O\right)\right)}\leq\frac{\mathcal{K}}{\tilde{\Lambda}}\leq\kappa.
\]
To finish the proof, it suffices to show that 
\[
\sup_{t\geq T}\,\sqrt{t}\left\Vert A_{\Sigma_{t}}\right\Vert _{L^{\infty}\left(B_{\tilde{\Lambda}\sqrt{t}}\left(O\right)\right)}\leq\kappa
\]
for some constant $T\geq\left(\frac{\mathcal{R}}{\tilde{\Lambda}}\right)^{2}$.
In fact, the above condition will follow from Lemma \ref{temporal curvature estimate}
once we show that the hypothesis in the lemma is satisfied (with the
constant $\Lambda$ therein replaced by $\tilde{\Lambda}$). To this
end, below we will prove that given $P\in B_{2\tilde{\Lambda}\sqrt{t}}\left(O\right)$
and $t\geq\frac{T}{2}$, there holds 
\[
F_{P,t}\left(\Sigma_{0}\right)\leq1+\epsilon,
\]
provided that $T\gg1$ (depending on $n$, $\kappa$, $\mathcal{C}$,
$\Sigma_{0}$) 

By the the second condition in Definition \ref{asymptotically conical},
outside a sufficiently large ball, $\Sigma_{0}$ is a normal graph
of $u$ over $\mathcal{C}$ outside a compact subset with 
\[
\left\Vert u\right\Vert _{C^{1}\left(\mathcal{C}\setminus B_{R}\left(O\right)\right)}\rightarrow0\quad\textrm{as}\,\,R\rightarrow\infty.
\]
Let $0<\delta\ll1$ be a constant to be determined (which will depend
only on $n,\kappa,\mathcal{C}$). Choose $R_{0}\gg1$ (depending on
$n,\mathcal{C},\Sigma_{0},\delta$) and a subset $\Omega\subset\mathcal{C}$
so that 
\[
\Sigma_{0}\setminus B_{R_{0}}\left(O\right)=\left\{ X=Z+uN_{\mathcal{C}}:Z\in\Omega\right\} 
\]
with 
\[
\left\Vert \nabla_{\mathcal{C}}u\right\Vert _{L^{\infty}\left(\Omega\right)}+\left\Vert u\right\Vert _{L^{\infty}\left(\Omega\right)}\leq\delta.
\]
Note that 
\begin{equation}
\frac{-1}{4t}\left|X-P\right|^{2}=\frac{-1}{4t}\left|Z+uN_{\mathcal{C}}-P\right|^{2}\label{DRC: conical condition}
\end{equation}
\[
=\frac{-1}{4t}\left[\left|Z-P\right|^{2}+2u\left(Z-P\right)\cdot N_{\mathcal{C}}+u^{2}\right]
\]
\[
\leq\frac{-1}{4t}\left|Z-P\right|^{2}+\frac{\left|P\right|\left\Vert u\right\Vert _{L^{\infty}\left(\Omega\right)}}{2t}\leq\frac{-1}{4t}\left|Z-P\right|^{2}+\frac{\tilde{\Lambda}\left\Vert u\right\Vert _{L^{\infty}\left(\Omega\right)}}{\sqrt{t}}
\]
for all $X\in\Sigma_{0}\setminus B_{R_{0}}\left(O\right)$. Note that
in the last line of (\ref{DRC: conical condition}) we have used the
fact that 
\[
Z\cdot N_{\mathcal{C}}=0
\]
because $\mathcal{C}$ is a cone. It then follows from (\ref{NG: Radon=002013Nikodym derivative})
in Lemma \ref{normal graph} (substituting $\Omega,\Sigma_{0}\setminus B_{R_{0}}\left(O\right),u$
for $\Sigma,\tilde{\Sigma},v$, respectively) and (\ref{DRC: conical condition})
that 
\begin{equation}
\int_{\Sigma_{0}\setminus B_{R_{0}}\left(O\right)}\frac{1}{\left(4\pi t\right)^{\frac{n}{2}}}e^{-\frac{\left|X-P\right|^{2}}{4t}}d\mathcal{H}^{n}\left(X\right)\label{DRC: Gaussian outside a ball}
\end{equation}
\[
=\int_{\Omega}\frac{1}{\left(4\pi t\right)^{\frac{n}{2}}}e^{-\frac{\left|Z+uN_{\mathcal{C}}-P\right|^{2}}{4t}}\frac{d\mathcal{H}^{n}\lfloor\Sigma_{0}}{d\mathcal{H}^{n}\lfloor\mathcal{C}}\,d\mathcal{H}^{n}\left(Z\right)
\]
\[
\leq e^{\frac{\tilde{\Lambda}\left\Vert u\right\Vert _{L^{\infty}\left(\Omega\right)}}{\sqrt{t}}}\left[1+C\left(n\right)\left(\left\Vert \nabla_{\mathcal{C}}u\right\Vert _{L^{\infty}\left(\Omega\right)}+\left\Vert A_{\mathcal{C}}u\right\Vert _{L^{\infty}\left(\Omega\right)}\right)\right]\int_{\mathcal{C}}\frac{1}{\left(4\pi t\right)^{\frac{n}{2}}}e^{-\frac{\left|Z-P\right|^{2}}{4t}}d\mathcal{H}^{n}\left(Z\right)
\]
 
\[
\leq e^{\frac{\sqrt{2}\delta\tilde{\Lambda}}{\sqrt{T}}}\left[1+C\left(n,\mathcal{C}\right)\delta\right]E\left[\mathcal{C}\right]\,\,\leq\,\,e^{\frac{\sqrt{2}\tilde{\Lambda}}{\sqrt{T}}}\,\frac{1}{2}\left(E\left[\mathcal{C}\right]+1+\epsilon\right),
\]
provided that $0<\delta\ll1$ (depending on $n,\kappa,\mathcal{C}$).
Note that $E\left[\mathcal{C}\right]<1+\epsilon$ and that the constant
$\epsilon$ depends on $n$ and $\kappa$. Also, in the last line
of (\ref{DRC: Gaussian outside a ball}) we have used the property
that $\left\Vert \,\left|Z\right|A_{\mathcal{C}}\right\Vert _{L^{\infty}\left(\mathbb{R}^{n+1}\setminus\left(O\right)\right)}<\infty.$
Moreover, we have 

\begin{equation}
\int_{\Sigma_{0}\cap B_{R_{0}}\left(O\right)}\frac{1}{\left(4\pi t\right)^{\frac{n}{2}}}e^{-\frac{\left|X-P\right|^{2}}{4t}}d\mathcal{H}^{n}\left(X\right)\leq\frac{\mathcal{H}^{n}\left(\Sigma_{0}\cap B_{R_{0}}\left(O\right)\right)}{\left(2\pi T\right)^{\frac{n}{2}}}.\label{DRC: Gaussian inside a ball}
\end{equation}
From (\ref{DRC: Gaussian outside a ball}) and (\ref{DRC: Gaussian inside a ball}),
one can see that $F_{P,t}\left(\Sigma_{0}\right)\leq1+\epsilon$ so
long as $T\gg1$ (depending on $n,\kappa,\mathcal{C},\Sigma_{0}$). 
\end{proof}
We conclude this section by the following lemma, which will be used
in proving Proposition \ref{sequential blow-down}. Its proof follows
exactly the same argument as in the preceding proof. 
\begin{lem}
\label{Gaussian on the large scale}Let $\Sigma$ be a smooth, properly
embedded, and oriented hypersurface in $\mathbb{R}^{n+1}$ that is
asymptotic to a regular cone $\mathcal{C}$ at infinity with $E\left[\mathcal{C}\right]<\infty$.
Then we have
\[
\sup_{t\geq T}F_{O,t}\left(\Sigma\right)\leq2E\left[\mathcal{C}\right]
\]
for some constant $T>0$ depending on $n$, $\mathcal{C}$, and $\Sigma$. 
\end{lem}

\begin{proof}
By the the second condition in Definition \ref{asymptotically conical},
outside a sufficiently large ball, $\Sigma$ is a normal graph of
$u$ over $\mathcal{C}$ outside a compact subset with 
\[
\left\Vert u\right\Vert _{C^{1}\left(\mathcal{C}\setminus B_{R}\left(O\right)\right)}\rightarrow0\quad\textrm{as}\,\,R\rightarrow\infty.
\]
Let $0<\delta\ll1$ be a constant to be determined (which will depend
only on $n,\mathcal{C}$). Choose $R_{0}\gg1$ (depending on $n,\mathcal{C},\Sigma,\delta$)
and a subset $\Omega\subset\mathcal{C}$ so that 
\[
\Sigma_{0}\setminus B_{R_{0}}\left(O\right)=\left\{ X=Z+uN_{\mathcal{C}}:Z\in\Omega\right\} 
\]
with 
\[
\left\Vert \nabla_{\mathcal{C}}u\right\Vert _{L^{\infty}\left(\Omega\right)}+\left\Vert u\right\Vert _{L^{\infty}\left(\Omega\right)}\leq\delta.
\]
Note that 
\begin{equation}
\frac{-1}{4t}\left|X\right|^{2}=\frac{-1}{4t}\left|Z+uN_{\mathcal{C}}\right|^{2}\label{GLS: conical}
\end{equation}
\[
=\frac{-1}{4t}\left[\left|Z\right|^{2}+2uZ\cdot N_{\mathcal{C}}+u^{2}\right]\leq\frac{-1}{4t}\left|Z\right|^{2}
\]
for all $X\in\Sigma_{0}\setminus B_{R_{0}}\left(O\right)$. Note also
that in the above we have used the fact that 
\[
Z\cdot N_{\mathcal{C}}=0
\]
because $\mathcal{C}$ is a cone. It follows from (\ref{NG: Radon=002013Nikodym derivative})
in Lemma \ref{normal graph} (in which substituting $\Omega,\Sigma\setminus B_{R_{0}}\left(O\right),u$
for $\Sigma,\tilde{\Sigma},v$, respectively) and (\ref{GLS: conical})
that 
\begin{equation}
\int_{\Sigma\setminus B_{R_{0}}\left(O\right)}\frac{1}{\left(4\pi t\right)^{\frac{n}{2}}}e^{-\frac{\left|X\right|^{2}}{4t}}d\mathcal{H}^{n}\left(X\right)=\int_{\Omega}\frac{1}{\left(4\pi t\right)^{\frac{n}{2}}}e^{-\frac{\left|Z+uN_{\mathcal{C}}\right|^{2}}{4t}}\frac{d\mathcal{H}^{n}\lfloor\Sigma}{d\mathcal{H}^{n}\lfloor\mathcal{C}}\,d\mathcal{H}^{n}\left(Z\right)\label{GLS: Gaussian outside a ball}
\end{equation}
\[
\leq\left[1+C\left(n\right)\left(\left\Vert \nabla_{\mathcal{C}}u\right\Vert _{L^{\infty}\left(\Omega\right)}+\left\Vert A_{\mathcal{C}}u\right\Vert _{L^{\infty}\left(\Omega\right)}\right)\right]\int_{\mathcal{C}}\frac{1}{\left(4\pi t\right)^{\frac{n}{2}}}e^{-\frac{\left|Z\right|^{2}}{4t}}d\mathcal{H}^{n}\left(Z\right)
\]
 
\[
\leq\left[1+C\left(n,\mathcal{C}\right)\delta\right]E\left[\mathcal{C}\right]\,\leq\,\frac{3}{2}E\left[\mathcal{C}\right]
\]
for all $t>0$, provided that $0<\delta\ll1$ (depending on $n,\mathcal{C}$).
Note that in the above estimate we have used the property that $\left\Vert \,\left|Z\right|A_{\mathcal{C}}\right\Vert _{L^{\infty}\left(\mathbb{R}^{n+1}\setminus\left(O\right)\right)}<\infty$.
Moreover, we have 
\begin{equation}
\int_{\Sigma\cap B_{R_{0}}\left(O\right)}\frac{1}{\left(4\pi t\right)^{\frac{n}{2}}}e^{-\frac{\left|X\right|^{2}}{4t}}d\mathcal{H}^{n}\left(X\right)\leq\frac{\mathcal{H}^{n}\left(\Sigma\cap B_{R_{0}}\left(O\right)\right)}{\left(4\pi t\right)^{\frac{n}{2}}}\,\leq\,\frac{1}{2}E\left[\mathcal{C}\right]\label{GLS: Gaussian inside a ball}
\end{equation}
provided that $t\gg1$ (depending on $n,\mathcal{C},\Sigma$). The
lemma follows immediately from (\ref{GLS: Gaussian outside a ball})
and (\ref{GLS: Gaussian inside a ball}).
\end{proof}

\section{Approaching Property\label{Approaching property}}

In this section we demonstrate the stability (in finite time) of MCF
with a conical end (see Theorem \ref{stability}). Then we prove that
under an extra condition on the curvature (see (\ref{AS: growth rate of curvature})),
there is an approaching property of NMCF (see Theorems \ref{asymptotic stability}
and \ref{approaching of NMCF}). 

To begin with, in the following proposition we show how the deviation
of one hypersurface from another evolves along MCF, which plays a
pivotal role in proving Theorems \ref{stability} and \ref{asymptotic stability}.
\begin{prop}
\label{evolution of deviation}Let $\left\{ \Sigma_{t}\right\} _{0\leq t\leq T}$
and $\left\{ \tilde{\Sigma}_{t}\right\} _{0\leq t\leq T}$ be MCFs
in $\mathbb{R}^{n+1}$ so that both $\Sigma_{0}$ and $\tilde{\Sigma}_{0}$
are asymptotic to a regular cone $\mathcal{C}$ at infinity. Suppose
that for every $0\leq t\leq T$, $\tilde{\Sigma}_{t}$ is a normal
graph of $v\left(\cdot,t\right)$ over $\Sigma_{t}$, i.e. 
\begin{equation}
\tilde{\Sigma}_{t}=\left\{ \tilde{X}\left(\cdotp,t\right)=X\left(\cdotp,t\right)+v\left(\cdot,t\right)N_{\Sigma_{t}}:X\left(\cdotp,t\right)\in\Sigma_{t}\right\} ,\label{ED: position}
\end{equation}
with
\[
\sup_{0\leq t\leq T}\left(\left\Vert \nabla_{\Sigma_{t}}v\right\Vert _{L^{\infty}}+\left\Vert A_{\Sigma_{t}}v\right\Vert _{L^{\infty}}\right)\leq\varsigma,
\]
where $\varsigma$ is the constant in Lemma \ref{normal graph}. Then
$v_{\textrm{max}}^{2}\left(t\right)\coloneqq\left\Vert v\left(\cdot,t\right)\right\Vert _{L^{\infty}}^{2}$
satisfies
\begin{equation}
D^{-}v_{\textrm{max}}^{2}\left(t\right)\coloneqq\limsup_{h\searrow0}\frac{v_{\textrm{max}}^{2}\left(t\right)-v_{\textrm{max}}^{2}\left(t-h\right)}{h}\label{ED: ODE inequality}
\end{equation}
\[
\leq\left\{ 2\left\Vert A_{\Sigma_{t}}\right\Vert _{L^{\infty}}^{2}+C\left(n\right)\left\Vert A_{\Sigma_{t}}v\right\Vert _{L^{\infty}}\left[\left\Vert A_{\Sigma_{t}}\right\Vert _{L^{\infty}}^{2}+\left\Vert \nabla_{\Sigma_{t}}A_{\Sigma_{t}}\right\Vert _{L^{\infty}}\right]\right\} v_{\textrm{max}}^{2}\left(t\right),
\]
for $0<t\leq T$. 
\end{prop}

\begin{proof}
In the first place, from the argument used in proving Proposition
\ref{asymptotically conical along MCF}, for every $0\leq t\leq T$
we get 
\begin{equation}
\left\Vert \nabla_{\Sigma_{t}}v\right\Vert _{L^{\infty}\left(\Sigma_{t}\setminus B_{R}\left(O\right)\right)}+\left\Vert v\right\Vert _{L^{\infty}\left(\Sigma_{t}\setminus B_{R}\left(O\right)\right)}\rightarrow0\quad\textrm{as}\,\,R\rightarrow\infty,\label{ED: asymptotic behavior}
\end{equation}
which implies that $v_{\textrm{max}}^{2}\left(t\right)$ is finite
and continuous. Also, note that by Proposition \ref{spatial curvature estimate}
and Corollary \ref{spatial smooth estimates} we have
\[
\sup_{0\leq t\leq T}\left(\left\Vert A_{\Sigma_{t}}\right\Vert _{L^{\infty}}^{2}+\left\Vert \nabla_{\Sigma_{t}}A_{\Sigma_{t}}\right\Vert _{L^{\infty}}\right)<\infty.
\]
To see how the deviation $v\left(\cdot,t\right)$ evolves over time,
let us differentiate the equation in (\ref{ED: position}) with respect
to $t$ and use the equation 
\[
\partial_{t}N_{\Sigma_{t}}=-\nabla_{\Sigma_{t}}H_{\Sigma_{t}}
\]
to get
\[
\partial_{t}\tilde{X}=\left(H+\partial_{t}v\right)N-v\nabla H.
\]
Here, and hereafter, we follow the notations and conventions in Lemma
\ref{normal graph}. For instance, we use $N$ to denote $N_{\Sigma_{t}}$
, $\tilde{N}$ for $N_{\tilde{\Sigma}_{t}}$, $\nabla$ for $\nabla_{\Sigma_{t}}$
and so forth. Since $\left\{ \tilde{\Sigma}_{t}\right\} $ is also
a MCF, it must hold that
\begin{equation}
\tilde{H}=\partial_{t}\tilde{X}\cdot\tilde{N}=\left(H+\partial_{t}v\right)N\cdot\tilde{N}-v\nabla H\cdot\tilde{N},\label{ED: MCF}
\end{equation}
(cf. Chapter 1 in \cite{M}). Eq. (\ref{ED: MCF}) combined with Eq.
(\ref{NG: biased mean curvature}) in Lemma \ref{normal graph} imply
\begin{equation}
\partial_{t}v=\left(N\cdot\tilde{N}\right)^{-1}\left[\tilde{H}+v\nabla H\cdot\tilde{N}\right]-H\label{ED: equation of deviation}
\end{equation}
 
\[
=a^{ij}\nabla_{ij}v\,+\,\left|A\right|^{2}v
\]
 
\[
+\,v\nabla A\ast\nabla v\,+\,A\ast Q\left(\nabla v,Av\right)\,+\,v\nabla A\ast Q\left(\nabla v,Av\right).
\]
Applying the power rule to Eq. (\ref{ED: equation of deviation})
gives
\begin{equation}
\partial_{t}\,v^{2}=a^{ij}\nabla_{ij}\,v^{2}-2a^{ij}\nabla_{i}v\nabla_{j}v+2\left|A\right|^{2}v^{2}\label{ED: equation of square deviation}
\end{equation}
 
\[
+\,v^{2}\nabla A\ast\nabla v+Av*Q\left(\nabla v,Av\right)+v^{2}\nabla A\ast Q\left(\nabla v,Av\right).
\]

Let us fix $0<t_{0}\leq T$. If $v_{\textrm{max}}^{2}\left(t_{0}\right)=0$,
then 
\[
D^{-}v_{\textrm{max}}^{2}\left(t_{0}\right)=\limsup_{h\searrow0}\frac{v_{\textrm{max}}^{2}\left(t_{0}\right)-v_{\textrm{max}}^{2}\left(t_{0}-h\right)}{h}\leq0
\]
and (\ref{ED: ODE inequality}) holds trivially. So let us consider
the nontrivial case where $v_{\textrm{max}}^{2}\left(t_{0}\right)>0$.
By condition (\ref{ED: asymptotic behavior}), there exists $P\in\Sigma_{t_{0}}$
so that $v_{\textrm{max}}^{2}\left(t_{0}\right)=v^{2}\left(P,t_{0}\right)$.
Note that 
\[
\frac{v_{\textrm{max}}^{2}\left(t_{0}\right)-v_{\textrm{max}}^{2}\left(t_{0}-h\right)}{h}\leq\frac{v^{2}\left(P,t_{0}\right)-v^{2}\left(P,t_{0}-h\right)}{h}\quad\,\forall\,\,0<h\ll1.
\]
As $h\searrow0$ we obtain 
\begin{equation}
D^{-}v_{\textrm{max}}^{2}\left(t_{0}\right)\leq\partial_{t}\,v^{2}\left(P,t_{0}\right).\label{ED: time derivative}
\end{equation}
Furthermore, since $P$ is an interior maximum point for $v^{2}\left(\cdot,t_{0}\right)$,
we have 
\begin{equation}
0=\left.\nabla\,v^{2}\right|_{\left(P,t_{0}\right)}=2v\left(P,t_{0}\right)\nabla v\left(P,t_{0}\right)\,\,\Rightarrow\,\,\nabla v\left(P,t_{0}\right)=0\label{ED: first derivative}
\end{equation}
and 
\begin{equation}
\left.a^{ij}\nabla_{ij}\,v^{2}\right|_{\left(P,t_{0}\right)}\leq0\label{ED: second derivative}
\end{equation}
by condition (\ref{NG: elliptic}). Evaluating Eq. (\ref{ED: equation of square deviation})
at $(P,t_{0})$, using (\ref{ED: time derivative}), (\ref{ED: first derivative}),
and (\ref{ED: second derivative}), and noting that 
\[
\left|Q\left(\nabla v,Av\right)\right|_{\left(P,t_{0}\right)}=\left|Q\left(Av\right)\right|_{\left(P,t_{0}\right)}\leq C\left(n\right)\left\Vert Av\right\Vert _{L^{\infty}\left(\Sigma_{t_{0}}\right)}^{2},
\]
we obtain 
\[
D^{-}v_{\textrm{max}}^{2}\left(t_{0}\right)\leq2\left\Vert A\right\Vert _{L^{\infty}\left(\Sigma_{t_{0}}\right)}^{2}v_{\textrm{max}}^{2}\left(t_{0}\right)
\]
\[
+\left\Vert Av\right\Vert _{L^{\infty}\left(\Sigma_{t_{0}}\right)}\cdot\,C\left(n\right)\left\Vert A\right\Vert _{L^{\infty}\left(\Sigma_{t_{0}}\right)}^{2}v_{\textrm{max}}^{2}\left(t_{0}\right)\,\,+\,\,v_{\textrm{max}}^{2}\left(t_{0}\right)\left\Vert \nabla A\right\Vert _{L^{\infty}\left(\Sigma_{t_{0}}\right)}\cdotp C\left(n\right)\left\Vert Av\right\Vert _{L^{\infty}\left(\Sigma_{t_{0}}\right)}^{2},
\]
from which the proposition follows.
\end{proof}
In order to prove Theorem \ref{stability}, we have to consider the
evolution of the gradient of the deviation as well. Below we derive
its equation using some facts from Riemannian geometry.
\begin{lem}
\label{evolution of modulus gradient}Let $v$ be a function on MCF
$\left\{ \Sigma_{t}\right\} $ satisfying 
\[
\partial_{t}v-a^{ij}\nabla_{ij}v=f,
\]
where $a^{ij}$ is a 2-tensor, $\nabla_{ij}$ is the Hessian in $\Sigma_{t}$,
and $f$ is a function. Then we have 
\[
\partial_{t}\left|\nabla v\right|^{2}-a^{ij}\nabla_{ij}\left|\nabla v\right|^{2}=-2a^{ij}\nabla_{ik}v\,\nabla_{j}^{k}v+2\nabla^{k}a^{ij}\,\nabla_{k}v\,\nabla_{ij}v
\]
\[
+2\left[a^{kl}\left(A_{k}^{j}A_{l}^{i}-A_{kl}A^{ij}\right)+HA^{ij}\right]\nabla_{i}v\,\nabla_{j}v+2\nabla^{i}v\,\nabla_{i}f.
\]
\end{lem}

\begin{proof}
Recall that the inverse of the metric $g_{ij}$ satisfies
\[
\partial_{t}g^{ij}=2HA^{ij}
\]
(cf. Chapter 2 in \cite{M}), which, combined with the equation of
$v$, implies
\begin{equation}
\partial_{t}\left|\nabla v\right|^{2}=2\left(\nabla_{k}\partial_{t}v\right)\nabla^{k}v+2HA^{ij}\nabla_{i}v\nabla_{j}v\label{EMG: time derivative}
\end{equation}
\[
=2\left(a^{ij}\nabla_{kij}v+\nabla_{k}a^{ij}\,\nabla_{ij}v+\nabla_{k}f\right)\nabla^{k}v+2HA^{ij}\nabla_{i}v\nabla_{j}v.
\]
On the other hand, we have 
\begin{equation}
a^{ij}\nabla_{ij}\left|\nabla v\right|^{2}=2a^{ij}\left(\nabla_{ijk}v\nabla^{k}v+\nabla_{ik}v\nabla_{j}^{k}v\right).\label{EMG: Hessian}
\end{equation}
Using the Riemann curvature tensor and Gauss equation, we obtain
\begin{equation}
\nabla_{ijk}v=\nabla_{ikj}v=\nabla_{kij}v-R_{ikjl}\nabla^{l}v=\nabla_{kij}v+(A_{ij}A_{kl}-A_{il}A_{jk})\nabla^{l}v\label{EMG: Riemann curvature}
\end{equation}
Substituting Eq. (\ref{EMG: Riemann curvature}) for $\nabla_{ijk}v$
in Eq. (\ref{EMG: Hessian}) gives 
\begin{equation}
a^{ij}\nabla_{ij}\left|\nabla v\right|^{2}=2a^{ij}\left[\nabla_{kij}v+(A_{ij}A_{kl}-A_{il}A_{jk})\nabla^{l}v\right]\nabla^{k}v+2a^{ij}\nabla_{ik}v\nabla_{j}^{k}v.\label{EMG: revised Hessian}
\end{equation}
The lemma follows from subtracting Eq. (\ref{EMG: revised Hessian})
from Eq. (\ref{EMG: time derivative}).
\end{proof}
The following is a (finite time) stability theorem for MCF with a
conical end. 
\begin{thm}
\label{stability}Let $\left\{ \Sigma_{t}\right\} _{0\leq t\leq T}$
and $\left\{ \tilde{\Sigma}_{t}\right\} _{0\leq t\leq T}$ be MCFs
in $\mathbb{R}^{n+1}$ so that both $\Sigma_{0}$ and $\tilde{\Sigma}_{0}$
are asymptotic to a regular cone $\mathcal{C}$ at infinity. Let 
\begin{equation}
K=\sup_{0\leq t\leq T}\left(\left\Vert A_{\Sigma_{t}}\right\Vert _{L^{\infty}}+\left\Vert \nabla_{\Sigma_{t}}A_{\Sigma_{t}}\right\Vert _{L^{\infty}}+\left\Vert \nabla_{\Sigma_{t}}^{2}A_{\Sigma_{t}}\right\Vert _{L^{\infty}}\right),\label{S: curvature bound}
\end{equation}
which is finite by Proposition \ref{spatial curvature estimate} and
Corollary \ref{spatial smooth estimates}.

Given $\varepsilon>0$, there exists $\delta>0$ depending on $n$,
$K,$ $T$, and $\varepsilon$ so that if $\tilde{\Sigma}_{0}$ is
a normal graph of $v_{0}$ over $\Sigma_{0}$ with
\[
\left\Vert \nabla_{\Sigma_{0}}v_{0}\right\Vert _{L^{\infty}}+\left\Vert v_{0}\right\Vert _{L^{\infty}}\leq\delta,
\]
then for every $0\leq t\leq T$, $\tilde{\Sigma}_{t}$ is a normal
graph of $v_{t}$ over $\Sigma_{t}$ with
\[
\sup_{0\leq t\leq T}\left(\left\Vert \nabla_{\Sigma_{t}}v_{t}\right\Vert _{L^{\infty}}+\left\Vert v_{t}\right\Vert _{L^{\infty}}\right)\leq\varepsilon.
\]
\end{thm}

\begin{proof}
Given $\varepsilon>0$, without loss of generality we may assume that
$\varepsilon\ll1$ (depending on $n$, $K$).

Suppose that $\tilde{\Sigma}_{0}$ is a normal graph of $v_{0}$ over
$\Sigma_{0}$ with $\left\Vert \nabla_{\Sigma_{0}}v_{0}\right\Vert _{L^{\infty}}+\left\Vert v_{0}\right\Vert _{L^{\infty}}\leq\delta$,
where $0<\delta\ll1$ is a constant to be specified (which will depend
on $n$, $K,$ $T$, $\varepsilon$). By continuity and the argument
in proving Proposition \ref{asymptotically conical along MCF}, it
is not hard to see that $\tilde{\Sigma}_{t}$ is a normal graph of
$v_{t}$ over $\Sigma_{t}$ with $\left\Vert \nabla_{\Sigma_{t}}v_{t}\right\Vert _{L^{\infty}}+\left\Vert v_{t}\right\Vert _{L^{\infty}}\leq\varepsilon$
for $0\leq t\ll1$. Let $T_{*}\in\left(0,T\right]$ be the supremum
of times before which $\tilde{\Sigma}_{t}$ is a normal graph of $v_{t}$
over $\Sigma_{t}$ with
\[
\sup_{0\leq t\leq T_{*}}\left(\left\Vert \nabla_{\Sigma_{t}}v_{t}\right\Vert _{L^{\infty}}+\left\Vert v_{t}\right\Vert _{L^{\infty}}\right)\leq\varepsilon,
\]
which implies
\[
\sup_{0\leq t\leq T_{*}}\left(\left\Vert \nabla_{\Sigma_{t}}v_{t}\right\Vert _{L^{\infty}\left(\Sigma_{t}\right)}+K\left\Vert v_{t}\right\Vert _{L^{\infty}\left(\Sigma_{t}\right)}\right)\leq\left(1+K\right)\varepsilon\leq\varsigma,
\]
provided that $\varepsilon\leq\frac{\varsigma}{K+1}$, where $\varsigma$
is the constant in Proposition \ref{evolution of deviation}. Let
$v$ be the function defined in the space-time whose time-slice is
given by $v_{t}$, i.e. $v\left(\cdot,t\right)=v_{t}$. By the argument
in Proposition \ref{evolution of deviation}, $v$ satisfies Eq. (\ref{ED: equation of deviation}).
It follows from Lemma \ref{evolution of modulus gradient} that 
\begin{equation}
\partial_{t}\left|\nabla v\right|^{2}-a^{ij}\nabla_{ij}\left|\nabla v\right|^{2}\,=\,-2a^{ij}\nabla_{ik}v\,\nabla_{j}^{k}v+2\nabla^{k}a^{ij}\,\nabla_{k}v\,\nabla_{ij}v\label{S: equation}
\end{equation}
\[
+2\left[a^{kl}\left(A_{k}^{j}A_{l}^{i}-A_{kl}A^{ij}\right)+HA^{ij}\right]\nabla_{i}v\,\nabla_{j}v+2\nabla^{i}v\,\nabla_{i}f,
\]
where 
\[
a^{ij}=g^{ij}+2A^{ij}v+Q^{ij}\left(\nabla v,Av\right),
\]
 
\[
f=\left|A\right|^{2}v+v\nabla A\ast\nabla v+A\ast Q\left(\nabla v,Av\right)+v\nabla A\ast Q\left(\nabla v,Av\right).
\]
Note that
\[
a^{ij}\nabla_{ik}v\,\nabla_{j}^{k}v\,\geq\,\left[\,1\,-C\left(n\right)\left(\left|\nabla v\right|+\left|Av\right|\right)\,\right]\,\left|\nabla^{2}v\right|^{2},
\]
 
\[
\left|\nabla a\right|\leq C\left(n\right)\left(\left|\nabla^{2}v\right|+\left|A\nabla v\right|+\left|v\nabla A\right|\right),
\]
 
\[
\left|\nabla f\right|\leq C\left(n\right)\left\{ \left(\left|v\nabla A\right|+\left|A\right|\right)\left|\nabla^{2}v\right|\,+\,\left(\left|\nabla v\right|+\left|Av\right|\right)\left[\left|v\nabla^{2}A\right|+\left|v\nabla A\right|^{2}+\left|\nabla A\right|+\left|A\right|^{2}\right]\right\} .
\]
Using the Cauchy-Schwarz inequality, Eq. (\ref{S: equation}) implies
\begin{equation}
\partial_{t}\left|\nabla v\right|^{2}-a^{ij}\nabla_{ij}\left|\nabla v\right|^{2}\label{S: revised equation}
\end{equation}
\[
\leq C\left(n\right)\left[\left|v\nabla^{2}A\right|+\left|v\nabla A\right|^{2}+\left|\nabla A\right|+\left|A\right|^{2}\right]\,\left|\nabla v\right|^{2}
\]
\[
+\,C\left(n\right)\left[\left|A\right|^{2}\left|v\nabla^{2}A\right|+\left|A\right|^{2}\left|v\nabla A\right|^{2}+\left|\nabla A\right|^{2}+\left|A\right|^{2}\left|\nabla A\right|+\left|A\right|^{4}\right]\,\left|v\right|^{2}
\]
\[
\leq C\left(n,K\right)\left(\left|\nabla v\right|^{2}+\left|v\right|^{2}\right),
\]
provided that $\varepsilon\ll1$ (depending on $n$, $K$). By condition
(\ref{ED: asymptotic behavior}), it is permitted to apply the maximum
principle to Eq. (\ref{S: revised equation}) (in exactly the same
way as we did in the proof in Proposition \ref{evolution of deviation})
and infer that $\left|\nabla v\right|_{\textrm{max}}^{2}\left(t\right)\coloneqq\left\Vert \nabla v\right\Vert _{L^{\infty}\left(\Sigma_{t}\right)}^{2}$
satisfies 
\begin{equation}
D^{-}\left|\nabla v\right|_{\textrm{max}}^{2}\left(t\right)\leq C\left(n,K\right)\left(\left|\nabla v\right|_{\textrm{max}}^{2}\left(t\right)+v_{\textrm{max}}^{2}\left(t\right)\right),\label{S: ODE inequality}
\end{equation}
where $D^{-}$ is the Dini derivative and $v_{\textrm{max}}^{2}\left(t\right)=\left\Vert v\right\Vert _{L^{\infty}\left(\Sigma_{t}\right)}^{2}$
as defined in Proposition \ref{evolution of deviation}. Adding up
Eqs. (\ref{ED: ODE inequality}) and (\ref{S: ODE inequality}) gives
\[
D^{-}\left(\left|\nabla v\right|_{\textrm{max}}^{2}+v_{\textrm{max}}^{2}\right)\leq C\left(n,K\right)\left(\left|\nabla v\right|_{\textrm{max}}^{2}+v_{\textrm{max}}^{2}\right).
\]
It follows from the comparison principle for ODE (cf. \cite{Wa})
that 
\begin{equation}
\left|\nabla v\right|_{\textrm{max}}^{2}\left(t\right)+v_{\textrm{max}}^{2}\left(t\right)\leq\left(\left|\nabla v\right|_{\textrm{max}}^{2}\left(0\right)+v_{\textrm{max}}^{2}\left(0\right)\right)e^{C\left(n,K\right)t}\label{S: Gronwall's inequality}
\end{equation}
for all $0\leq t\leq T_{*}$. 

Therefore, if $0<\delta\ll1$ so that $\delta e^{C\left(n,K\right)T}\leq\frac{1}{2}\varepsilon.$
Then we have $T_{*}=T$ and the theorem follows from (\ref{S: Gronwall's inequality}).
\end{proof}
As a corollary of Theorem \ref{stability}, we have the following
uniqueness theorem for MCF with prescribed initial hypersurface. 
\begin{cor}
\label{uniqueness}Let $\Sigma$ be a smooth, properly embedded, and
oriented hypersurface in $\mathbb{R}^{n+1}$ that is asymptotic to
a regular cone $\mathcal{C}$ at infinity. Then for any $T>0$, there
is at most one MCF $\left\{ \Sigma_{t}\right\} _{0\leq t\leq T}$
in $\mathbb{R}^{n+1}$ with $\Sigma_{0}=\Sigma$. 
\end{cor}

In Theorem \ref{stability}, we have shown that if two MCFs are initially
close in the $C^{1}$ topology and asymptotic to the same cone at
infinity, then they stay close for some time. What can we say about
any kind of stability in the long run? To answer that question, we
need to have the smooth estimates of the deviation first (see Proposition
\ref{interpolation inequality}). The following lemma is the smooth
estimates of the curvature.
\begin{lem}
\label{smooth decay estimate }Given $0<\sigma<1$, let $\left\{ \Sigma_{t}\right\} _{T\leq t\leq S}$
be a MCF in $\mathbb{R}^{n+1}$, where $0<T<S$ are constants satisfying
$S\geq\left(1+\sigma\right)T$, so that
\[
\sup_{T\leq t\leq S}\sqrt{t}\left\Vert A_{\Sigma_{t}}\right\Vert _{L^{\infty}}\leq\kappa
\]
for some constant $\kappa>0$. Then for every $k\in\mathbb{N}$ we
have 
\[
\sup_{\left(1+\sigma\right)T\leq t\leq S}\sqrt{t}^{k+1}\left\Vert \nabla_{\Sigma_{t}}^{k}A_{\Sigma_{t}}\right\Vert _{L^{\infty}}\leq C\left(n,\sigma,\kappa,k\right).
\]
\end{lem}

\begin{proof}
Given $P\in\mathbb{R}^{n+1}$ and $t_{0}\in\left[\left(1+\sigma\right)T,S\right]$,
let
\[
\tilde{\Sigma}_{\tau}=\sqrt{\frac{1+\sigma}{\sigma t_{0}}}\left(\Sigma_{\frac{1+\sigma\tau}{1+\sigma}t_{0}}-P\right).
\]
Then $\left\{ \tilde{\Sigma}_{\tau}\right\} _{0\leq\tau\leq1}$ is
a MCF in $B_{1}\left(O\right)$ satisfying 
\[
\sup_{0\leq\tau\leq1}\left\Vert A_{\tilde{\Sigma}_{\tau}}\right\Vert _{L^{\infty}\left(B_{1}\left(O\right)\right)}\leq\sqrt{\sigma}\kappa.
\]
It follows from Proposition \ref{smooth estimates} that for every
$k\in\mathbb{N}$ we have
\[
\sup_{\frac{3}{4}t_{0}\leq t\leq t_{0}}\sqrt{t_{0}}^{k+1}\left\Vert \nabla_{\Sigma_{t}}^{k}A_{\Sigma_{t}}\right\Vert _{L^{\infty}\left(B_{\frac{1}{2}\sqrt{\frac{\sigma t_{0}}{1+\sigma}}}\left(P\right)\right)}
\]
\[
=\sup_{\frac{1}{2}\leq\tau\leq1}\sqrt{\frac{1+\sigma}{\sigma}}^{k+1}\left\Vert \nabla_{\tilde{\Sigma}_{\tau}}^{k}A_{\tilde{\Sigma}_{\tau}}\right\Vert _{L^{\infty}\left(B_{\frac{1}{2}}\left(O\right)\right)}\leq C\left(n,\sigma,\kappa,k\right).
\]
\end{proof}
What is also essential in deriving the smooth estimates of the deviation
is an appropriate choice of the local coordinates for the flow, as
seen in the following proposition. Note that due to condition (\ref{ANP: curvature condition}),
the ``scale'' of the local coordinates depends only on time.
\begin{prop}
\label{analyticity: normal parametrization}Given $0<\sigma<1$, let
$\left\{ \Sigma_{t}\right\} _{T\leq t\leq S}$ be a MCF in $\mathbb{R}^{n+1}$,
where $0<T<S$ are constants satisfying $S\geq\left(1+2\sigma\right)T$,
so that 
\begin{equation}
\sup_{T\leq t\leq S}\sqrt{t}\left\Vert A_{\Sigma_{t}}\right\Vert _{L^{\infty}}\leq\kappa\label{ANP: curvature condition}
\end{equation}
for some constant $\kappa>0$. Then given $\delta>0$, there exists
a constant $0<\theta<1$ depending on $n$, $\sigma$, $\kappa$ and
$\delta$ with the following property. 

Let $P$ be an arbitrary point in $\Sigma_{t_{0}}$ with $t_{0}\in\left[\left(1+2\sigma\right)T,S\right]$.
Near the point $P$ and time $t_{0}$, the flow $\left\{ \Sigma_{t}\right\} $
admits a local coordinate chart 
\[
X:B_{\theta\sqrt{t_{0}}}^{n}\left(O\right)\times\left[\left(1-\theta^{2}\right)t_{0},t_{0}\right]\rightarrow\mathbb{R}^{n+1}
\]
with $X\left(O,t_{0}\right)=P$ and $\partial_{t}X\left(x,t\right)=\vec{H}\left(x,t\right)$.
Also, for any $l,m\in\mathbb{N}\cup\left\{ 0\right\} $ we have 
\begin{itemize}
\item The metric $g_{ij}$ satisfies 
\[
\left(1-\delta\right)\delta_{ij}\leq g_{ij}\left(x,t\right)\leq\left(1+\delta\right)\delta_{ij},\quad\sqrt{t_{0}}\left|\partial_{x}g_{ij}\left(x,t\right)\right|\leq C\left(n\right)\delta,
\]
\[
\sqrt{t_{0}}^{2l+m}\left|\partial_{t}^{l}\partial_{x}^{m}g_{ij}\left(x,t\right)\right|\leq C\left(n,\sigma,\kappa,l,m\right).
\]
 
\item The Christoffel symbols $\Gamma_{ij}^{k}$ satisfy 
\[
\sqrt{t_{0}}\left|\Gamma_{ij}^{k}\left(x,t\right)\right|\leq\delta,
\]
\[
\sqrt{t_{0}}^{2l+m+1}\left|\partial_{t}^{l}\partial_{x}^{m}\Gamma_{ij}^{k}\left(x,t\right)\right|\leq C\left(n,\sigma,\kappa,l,m\right).
\]
\item The second fundamental form $A_{ij}$ satisfies 
\[
\sqrt{t_{0}}\left|A_{ij}\left(x,t\right)\right|\leq\frac{1+C\left(n\right)\delta}{\sqrt{1-\theta^{2}}}\kappa,
\]
 
\[
\sqrt{t_{0}}^{2l+m+1}\left|\partial_{t}^{l}\partial_{x}^{m}A_{ij}\left(x,t\right)\right|\leq C\left(n,\sigma,\kappa,l,m\right),
\]
 
\[
\sqrt{t_{0}}^{2l+m+2}\left|\partial_{t}^{l}\partial_{x}^{m}\nabla_{k}A_{ij}\left(x,t\right)\right|\leq C\left(n,\sigma,\kappa,l,m\right),
\]
 
\[
\sqrt{t_{0}}^{2l+m+3}\left|\partial_{t}^{l}\partial_{x}^{m}\nabla_{kq}A_{ij}\left(x,t\right)\right|\leq C\left(n,\sigma,\kappa,l,m\right).
\]
Note that $\nabla_{k}A_{ij}$ and $\nabla_{kq}A_{ij}$ denote the
coordinates of $\nabla_{\Sigma_{t}}A_{\Sigma_{t}}$ and $\nabla_{\Sigma_{t}}^{2}A_{\Sigma_{t}}$,
respectively. 
\end{itemize}
\end{prop}

\begin{proof}
Given $\delta>0$, let $0<\theta\ll1$ be a constant to be determined
(which will depend only on $n,\sigma,\kappa,\delta$). 

Fix $P\in\Sigma_{t_{0}}$ with $t_{0}\geq\left(1+2\sigma\right)T$.
Since $\sqrt{t_{0}}\left\Vert A_{\Sigma_{t_{0}}}\right\Vert _{L^{\infty}}\leq\kappa$,
if $\theta\ll1$ (depending on $n,\kappa$), we can locally parametrize
$\Sigma_{t_{0}}$ as a graph over $T_{P}\Sigma_{t_{0}}$, say 
\[
X\left(\cdot,t_{0}\right):B_{\theta\sqrt{t_{0}}}^{n}\left(O\right)\rightarrow\mathbb{R}^{n+1}
\]
with $X\left(O,t_{0}\right)=P$, so that the gradient of the graph
is sufficiently small and Lemma \ref{equivalence of smooth estimates}
is applicable to the image of $\frac{1}{\sqrt{t_{0}}}\left(X\left(\cdot,t_{0}\right)-P\right)$.
In view of (\ref{ESE: metric}), (\ref{ESE: second fundamental form}),
(\ref{ESE: Christoffel symbols}), and Lemma \ref{smooth decay estimate },
we may also assume that 
\begin{equation}
\left(1-\frac{\delta}{2}\right)\delta_{ij}\leq g_{ij}\left(x,t_{0}\right)\leq\left(1+\frac{\delta}{2}\right)\delta_{ij},\label{ANP: initial metric}
\end{equation}
\begin{equation}
\sqrt{t_{0}}\left|\Gamma_{ij}^{k}\left(x,t_{0}\right)\right|\leq\frac{\delta}{2},\quad\sqrt{t_{0}}^{m+1}\left|\partial_{x}^{m}\Gamma_{ij}^{k}\left(x,t\right)\right|\leq C\left(n,\sigma,\kappa,m\right)\label{ANP: initial Christoffel symbols}
\end{equation}
for all $x\in B_{\theta\sqrt{t_{0}}}^{n}\left(O\right)$ and $m\in\mathbb{N}$,
provided that $\theta\ll1$ (depending on $n,\kappa,\delta$).

Next, let 
\[
X\left(\cdot,t\right):B_{\theta\sqrt{t_{0}}}^{n}\left(O\right)\rightarrow\mathbb{R}^{n+1}
\]
be the trajectory of $X\left(\cdot,t_{0}\right)$ along the MCF $\left\{ \Sigma_{t}\right\} $,
namely 
\[
\partial_{t}X\left(x,t\right)=\vec{H}\left(x,t\right).
\]
Below we would like to show that all the estimates stated in the proposition
hold for $X\left(\cdot,t\right)$ as long as $\left(1-\theta^{2}\right)t_{0}\leq t\leq t_{0}$,
provided that $\theta\ll1$ (depending on $n,\sigma,\kappa,\delta$).

To begin with, let us recall that 
\[
\partial_{t}g_{ij}=-2HA_{ij}
\]
(cf. Chapter 2 in \cite{M}), from which for every nonzero constant
vector $\xi$ in $\mathbb{R}^{n}$ we have
\[
\partial_{t}\left[\ln\left(g_{ij}\xi^{i}\xi^{j}\right)\right]=\frac{\partial_{t}\left(g_{ij}\xi^{i}\xi^{j}\right)}{g_{ij}\xi^{i}\xi^{j}}=-2H\frac{A_{ij}\xi^{i}\xi^{j}}{g_{ij}\xi^{i}\xi^{j}}
\]
\[
\Rightarrow\left|\partial_{t}\left[\ln\left(g_{ij}\xi^{i}\xi^{j}\right)\right]\right|\leq2\left|HA\right|\leq\frac{2\sqrt{n}\kappa^{2}}{t}.
\]
It follows that 
\[
\left(\frac{t_{0}}{t}\right)^{-2\sqrt{n}\kappa^{2}}\leq\frac{g_{ij}\left(x,t\right)\xi^{i}\xi^{j}}{g_{ij}\left(x,t_{0}\right)\xi^{i}\xi^{j}}\leq\left(\frac{t_{0}}{t}\right)^{2\sqrt{n}\kappa^{2}}
\]
for all $\xi\in\mathbb{R}^{n}\setminus\left\{ O\right\} $, $\left(1-\theta^{2}\right)t_{0}\leq t\leq t_{0}$,
which combined with (\ref{ANP: initial metric}) imply 
\[
\left(1-\delta\right)\delta_{ij}\leq g_{ij}\left(x,t\right)\xi^{i}\xi^{j}\leq\left(1+\delta\right)\delta_{ij}
\]
provided that $\theta\ll1$ (depending on $n,\sigma,\kappa,\delta$).
In other words, the metric $g_{ij}\left(x,t\right)$ is equivalent
to $\delta_{ij}$; whence there is no need to distinguish between
the Riemannian norm of a tensor and the $l^{2}$ norm of its coordinates
(just as in Lemma \ref{equivalence of smooth estimates}). As a consequence,
we have 
\[
\left|\partial_{t}g_{ij}\left(x,t\right)\right|=\left|-2HA_{ij}\left(x,t\right)\right|\leq\frac{C\left(n\right)\kappa^{2}}{t_{0}},
\]
\[
\sqrt{t_{0}}\left|A_{ij}\left(x,t\right)\right|\leq\frac{1+C\left(n\right)\delta}{\sqrt{1-\theta^{2}}}\kappa
\]
for all $x\in B_{\theta\sqrt{t_{0}}}^{n}\left(O\right)$, $\left(1-\theta^{2}\right)t_{0}\leq t\leq t_{0}$.

To estimate the Christoffel symbols, recall that 
\begin{equation}
\partial_{t}\Gamma_{ij}^{k}=\frac{1}{2}g^{kl}\left(\nabla_{i}\dot{g}_{jl}+\nabla_{j}\dot{g}_{il}-\nabla_{l}\dot{g}_{ij}\right)\label{ANP: evolving Chirstoffel tensor}
\end{equation}
where 
\begin{equation}
\dot{g}_{ij}=\partial_{t}g_{ij}=-2HA_{ij},\label{ANP: evolving metric tensor}
\end{equation}
\[
\nabla_{k}\dot{g}_{ij}=\partial_{k}\dot{g}_{ij}-\Gamma_{ki}^{l}\,\dot{g}_{lj}-\Gamma_{kj}^{l}\,\dot{g}_{il}
\]
(cf. Chapter 2 in \cite{M}). Then it follows from Eqs. (\ref{ANP: evolving Chirstoffel tensor}),
(\ref{ANP: evolving metric tensor}), and Lemma \ref{smooth decay estimate }
that 
\begin{equation}
\left|\partial_{t}\Gamma_{ij}^{k}\right|=\left|\nabla A\ast A\right|\leq C\left(n\right)\left|\nabla A\right|\left|A\right|\leq\frac{C\left(n,\sigma,\kappa\right)}{\sqrt{t_{0}}^{3}},\label{ANP: time derivative of Christoffel}
\end{equation}
which combined with (\ref{ANP: initial Christoffel symbols}) imply
\begin{equation}
\sqrt{t_{0}}\left|\Gamma_{ij}^{k}\left(x,t\right)\right|\leq\delta\label{ANP: Christoffel symbols}
\end{equation}
for all $x\in B_{\theta\sqrt{t_{0}}}^{n}\left(O\right)$, $\left(1-\theta^{2}\right)t_{0}\leq t\leq t_{0}$,
provided that $\theta\ll1$ (depending on $n,\sigma,\kappa,\delta$).
Moreover, in view of the equation
\[
\partial_{k}g_{ij}=\Gamma_{ki}^{l}\,g_{lj}+\Gamma_{kj}^{l}\,g_{il},
\]
we get 
\[
\sqrt{t_{0}}\left|\partial_{k}g_{ij}\right|\leq C\left(n\right)\delta.
\]
To estimate the spatial derivative of Christoffel symbols, let $\dot{\Gamma}_{ij}^{k}=\partial_{t}\Gamma_{ij}^{k}$
(which satisfies Eq. (\ref{ANP: evolving Chirstoffel tensor})) and
note that 
\[
\partial_{t}\partial_{l}\Gamma_{ij}^{k}=\partial_{l}\partial_{t}\Gamma_{ij}^{k}=\nabla_{l}\dot{\Gamma}_{ij}^{k}+\Gamma_{li}^{m}\dot{\Gamma}_{mj}^{k}+\Gamma_{lj}^{m}\dot{\Gamma}_{im}^{k}-\Gamma_{lm}^{k}\dot{\Gamma}_{ij}^{m}.
\]
Using the fact that 
\[
\left|\nabla_{l}\dot{\Gamma}_{ij}^{k}\right|\leq\left|\nabla^{2}A\ast A\right|+\left|\nabla A\ast\nabla A\right|\leq\frac{C\left(n,\sigma,\kappa\right)}{\sqrt{t_{0}}^{4}}
\]
(see Eqs. (\ref{ANP: evolving Chirstoffel tensor}), (\ref{ANP: evolving metric tensor})
and Lemma \ref{smooth decay estimate }) together with (\ref{ANP: time derivative of Christoffel})
and (\ref{ANP: Christoffel symbols}), we obtain 
\[
\sqrt{t_{0}}^{4}\left|\partial_{t}\partial_{l}\Gamma_{ij}^{k}\right|\leq C\left(n,\sigma,\kappa\right),
\]
which combined with (\ref{ANP: initial Christoffel symbols}) imply
\[
\sqrt{t_{0}}^{2}\left|\partial_{l}\Gamma_{ij}^{k}\right|\leq C\left(n,\sigma,\kappa\right).
\]

To estimate the derivatives of the second fundamental form $A_{ij}$,
note that 
\[
\partial_{k}A_{ij}=\nabla_{k}A_{ij}+\Gamma_{ki}^{l}A_{lj}+\Gamma_{kj}^{l}A_{il}
\]
and 
\[
\partial_{t}A_{ij}=\triangle A_{ij}+\left|A\right|^{2}A_{ij}-2HA_{ij}^{2}
\]
(cf. \cite{M}). Then it follows from Lemma \ref{smooth decay estimate }
and (\ref{ANP: Christoffel symbols}) that
\[
\sqrt{t_{0}}^{2}\left|\partial_{x}A_{ij}\left(x,t\right)\right|+\sqrt{t_{0}}^{3}\left|\partial_{t}A_{ij}\left(x,t\right)\right|\leq C\left(n,\sigma,\kappa\right)
\]
for all $x\in B_{\theta\sqrt{t_{0}}}^{n}\left(O\right)$, $\left(1-\theta^{2}\right)t_{0}\leq t\leq t_{0}$. 

Lastly, in view of
\[
\partial_{t}\nabla A=\triangle\nabla A+A\ast A\ast\nabla A,
\]
\[
\partial_{t}\nabla^{2}A=\triangle\nabla^{2}A+A\ast A\ast\nabla^{2}A+A\ast\nabla A\ast\nabla A,
\]
and so forth (cf. Chapter 2 in \cite{M}), all other estimates for
the higher order derivatives can be deduced in a similar fashion,
so we omit the proof.
\end{proof}
Taking advantage of Eq. (\ref{ED: equation of deviation}), the local
coordinates in Proposition \ref{analyticity: normal parametrization},
and the regularity theory for quasilinear parabolic equations, we
can now derive the smooth estimates of the deviation in the following
proposition.
\begin{prop}
\label{interpolation inequality}Given $0<\sigma<1$, let $\left\{ \Sigma_{t}\right\} _{T\leq t\leq S}$
and $\left\{ \tilde{\Sigma}_{t}\right\} _{T\leq t\leq S}$ be two
MCFs in $\mathbb{R}^{n+1}$, where $0<T<S$ are constants satisfying
$S\geq\left(1+2\sigma\right)T$, with the following properties:
\begin{enumerate}
\item There exists a constant $\kappa>0$ so that 
\[
\sup_{T\leq t\leq S}\sqrt{t}\left\Vert A_{\Sigma_{t}}\right\Vert _{L^{\infty}}\leq\kappa.
\]
\item For every $T\leq t\leq S$, $\tilde{\Sigma}_{t}$ is a normal graph
of $v\left(\cdot,t\right)$ over $\Sigma_{t}$ with 
\[
\sup_{T\leq t\leq S}\left(\left\Vert \nabla_{\Sigma_{t}}v\right\Vert _{L^{\infty}\left(\Sigma_{t}\right)}+\kappa\frac{\left\Vert v\right\Vert _{L^{\infty}\left(\Sigma_{t}\right)}}{\sqrt{t}}\right)\leq\frac{\varsigma}{2},
\]
where $\varsigma$ is the constant in Proposition \ref{evolution of deviation}.
\end{enumerate}
Then for every $k\in\mathbb{N}$ we have 
\[
\sqrt{t}^{k-1}\left\Vert \nabla_{\Sigma_{t}}^{k}v\right\Vert _{L^{\infty}\left(\Sigma_{t}\right)}\leq C\left(n,\sigma,\kappa,k\right)\sup_{\frac{t}{1+2\sigma}\leq\tau\leq t}\frac{\left\Vert v\right\Vert _{L^{\infty}\left(\Sigma_{\tau}\right)}}{\sqrt{\tau}}
\]
for $\left(1+2\sigma\right)T\leq t\leq S$.
\end{prop}

\begin{proof}
Firstly, note that the two hypotheses combined imply
\[
\sup_{T\leq t\leq S}\left(\left\Vert \nabla_{\Sigma_{t}}v\right\Vert _{L^{\infty}}+\left\Vert A_{\Sigma_{t}}v\right\Vert _{L^{\infty}}\right)\leq\frac{\varsigma}{2}.
\]
Recall that the function $v$ satisfies Eq. (\ref{ED: equation of deviation}),
i.e. 
\begin{equation}
\partial_{t}v=a^{ij}\nabla_{ij}v+f,\label{II: equation}
\end{equation}
where 
\begin{equation}
a^{ij}=g^{ij}+2A^{ij}v+Q^{ij}\left(\nabla v,Av\right),\label{II: leading coefficients}
\end{equation}
\begin{equation}
f=\left|A\right|^{2}v+v\nabla A\ast\nabla v+A\ast Q\left(\nabla v,Av\right)+v\nabla A\ast Q\left(\nabla v,Av\right).\label{II: lower order terms}
\end{equation}
Let $0<\delta\ll1$ be a constant to be determined (which will depend
only on $n$). Let $P$ be an arbitrary point in $\Sigma_{t_{0}}$
with $\left(1+2\sigma\right)T\leq t_{0}\leq S$. For this choice of
$\delta$, Proposition \ref{analyticity: normal parametrization}
ensures that near the point $P$ and time $t_{0}$, the flow $\left\{ \Sigma_{t}\right\} $
admits a local coordinate chart 
\[
X:B_{\theta\sqrt{t_{0}}}^{n}\left(O\right)\times\left(\left(1-\theta^{2}\right)t_{0},t_{0}\right]\rightarrow\mathbb{R}^{n+1}
\]
with $X\left(O,t_{0}\right)=P$ and all other properties stated therein,
where $0<\theta<1$ is the constant in Proposition \ref{analyticity: normal parametrization}
(or possibly smaller). Note that $\theta$ depends on $n$, $\sigma$,
$\kappa$, and $\delta$. From now on, let us identify points on the
flow in the neighborhood of $\left(P,t_{0}\right)$ with their local
coordinates. With this identification, it follows from Eqs. (\ref{II: equation}),
(\ref{II: leading coefficients}), and (\ref{II: lower order terms})
that the function $v=v\left(x,t\right)$ satisfies a quasilinear parabolic
equation
\begin{equation}
\partial_{t}v=a^{ij}\partial_{ij}v-a^{ij}\Gamma_{ij}^{k}\left(x,t\right)\partial_{k}v+f\label{II: quasilinear equation}
\end{equation}
with $a^{ij}=a^{ij}\left(x,t,v,\partial_{x}v\right)$ and $f=f\left(x,t,v,\partial_{x}v\right)$.
To be more specific, the structures of the dependence are in the following
forms:

\[
a^{ij}\left(x,t,z,p\right)=\boldsymbol{a}^{ij}\left(g_{kl}\left(x,t\right),g^{pq}\left(x,t\right),A_{rs}\left(x,t\right)z,p\right),
\]
\[
f\left(x,t,z,p\right)=\boldsymbol{f}\left(g_{ij}\left(x,t\right),g^{kl}\left(x,t\right),A_{pq}\left(x,t\right),A_{rs}\left(x,t\right)z,z\nabla_{a}A_{bc}\left(x,t\right),p\right).
\]
If $\delta\ll1$ (depending on $n$) and $\theta\ll1$ (depending
on $n,\sigma,\kappa,\delta$), the second hypothesis in the proposition
together with Proposition \ref{analyticity: normal parametrization}
imply
\begin{equation}
\left|\partial_{x}v\left(x,t\right)\right|+\frac{\kappa}{\sqrt{1-\theta^{2}}}\frac{\left|v\left(x,t\right)\right|}{\sqrt{t_{0}}}\leq\frac{2}{3}\varsigma\label{II: C1 bound}
\end{equation}
for all $x\in B_{\theta\sqrt{t_{0}}}^{n}\left(O\right)$, $t\in\left(\left(1-\theta^{2}\right)t_{0},t_{0}\right]$.
Furthermore, in view of condition (\ref{NG: elliptic}), Eqs. (\ref{II: leading coefficients}),
(\ref{II: lower order terms}), and Proposition \ref{analyticity: normal parametrization},
if $\delta\ll1$ (depending on $n$), we have
\begin{equation}
\frac{1}{2}\delta^{ij}\leq a^{ij}\left(x,t,z,p\right)\leq\frac{3}{2}\delta^{ij},\label{II: parabolicity}
\end{equation}
\begin{equation}
\sqrt{t_{0}}\left|\partial_{x}a^{ij}\left(x,t,z,p\right)\right|+\sqrt{t_{0}}\left|\partial_{z}a^{ij}\left(x,t,z,p\right)\right|+\left|\partial_{p}a^{ij}\left(x,t,z,p\right)\right|\leq C\left(n,\sigma,\kappa\right),\label{II: derivatives of leading coefficients}
\end{equation}
\begin{equation}
\sqrt{t_{0}}\left|\Gamma_{ij}^{k}\left(x,t\right)\right|+\sqrt{t_{0}}\left|f\left(x,t,z,p\right)\right|\leq C\left(n,\sigma,\kappa\right)\label{II: estimates of lower order terms}
\end{equation}
for $x\in B_{\theta\sqrt{t_{0}}}^{n}\left(O\right)$, $t\in\left(\left(1-\theta^{2}\right)t_{0},t_{0}\right],$
$\left(z,p\right)\in\mathbb{R}\times\mathbb{R}^{n}$ satisfying
\[
\left|p\right|+\frac{\kappa}{\sqrt{1-\theta^{2}}}\frac{\left|z\right|}{\sqrt{t_{0}}}\leq\varsigma.
\]
Now consider the following change of variables: 
\begin{equation}
\bar{x}=\frac{x}{\sqrt{t_{0}}},\quad\bar{t}=\frac{t}{t_{0}},\quad\bar{v}\left(\bar{x},\bar{t}\right)=\frac{1}{\sqrt{t_{0}}}v\left(\sqrt{t_{0}}\,\bar{x},t_{0}\,\bar{t}\right).\label{II: change of variables}
\end{equation}
It follows from the H$\ddot{\textrm{o}}$lder gradient estimate for
quasilinear parabolic equations (cf. Chapter 12 in \cite{L}, by applying
the theory to the corresponding equation of $\bar{v}\left(\bar{x},\bar{t}\right)$
obtained from Eq. (\ref{II: quasilinear equation})), and using conditions
(\ref{II: C1 bound}), (\ref{II: parabolicity}), (\ref{II: derivatives of leading coefficients}),
and (\ref{II: estimates of lower order terms}), that there exists
$0<\alpha<1$ (depending on $n,\sigma,\kappa,\delta$) so that 
\begin{equation}
\sqrt{t_{0}}^{\alpha}\,\frac{\left|\partial_{x}v\left(x,t\right)-\partial_{x}v\left(x',t'\right)\right|}{\left(\left|x-x'\right|+\left|t-t'\right|^{\frac{1}{2}}\right)^{\alpha}}\leq C\left(n,\sigma,\kappa,\delta\right)\label{II: Holder gradient estimate}
\end{equation}
for all $x,x'\in B_{\frac{\theta}{2}\sqrt{t_{0}}}^{n}\left(O\right)$
and $t,t'\in\left(\left(1-\frac{\theta^{2}}{4}\right)t_{0},t_{0}\right]$
with $\left(x,t\right)\neq\left(x',t'\right)$. 

To proceed further, by taking into account the condition that $Q$
is at least ``quadratic'' in its argument (see Lemma \ref{normal graph}),
we can decompose the third term in Eq. (\ref{II: lower order terms})
into 
\[
A\ast Q\left(\nabla v,Av\right)=A\ast F_{1}\left(\nabla v,Av\right)\ast\nabla v+A\ast F_{0}\left(\nabla v,Av\right)\ast Av,
\]
where $F_{1}$ and $F_{2}$ are analytic tensors. It follows that
the function $f$ defined in (\ref{II: lower order terms}) can be
decomposed into
\[
f=b^{k}\nabla_{k}v+cv,
\]
where $b$ is a 1-tensor and $c$ is a function defined by
\begin{equation}
b=A\ast F_{1}\left(\nabla v,Av\right),\label{II: first order term}
\end{equation}
\begin{equation}
c=\left|A\right|^{2}+\nabla A\ast\nabla v+A\ast A\ast F_{0}\left(\nabla v,Av\right)+\nabla A\ast Q\left(\nabla v,Av\right).\label{II: zero order term}
\end{equation}
Thus, in terms of the local coordinates, the function $v\left(x,t\right)$
satisfies 
\begin{equation}
\partial_{t}v=a^{ij}\partial_{ij}v+\left[-a^{ij}\Gamma_{ij}^{k}\left(x,t\right)+b^{k}\right]\partial_{k}v+cv\label{II: homogeneous linear equation}
\end{equation}
with $a^{ij}=a^{ij}\left(x,t,v,\partial_{x}v\right)$, $b^{k}=b^{k}\left(x,t,v,\partial_{x}v\right)$
and $c=c\left(x,t,v,\partial_{x}v\right)$. By a similar calculation
as that in deriving (\ref{II: derivatives of leading coefficients}),
it follows from (\ref{II: first order term}), (\ref{II: zero order term}),
and Proposition \ref{analyticity: normal parametrization} that 
\begin{equation}
\sqrt{t_{0}}\left|b^{k}\left(x,t,z,p\right)\right|+\sqrt{t_{0}}^{2}\left|\partial_{x}b^{k}\left(x,t,z,p\right)\right|+\sqrt{t_{0}}^{3}\left|\partial_{t}b^{k}\left(x,t,z,p\right)\right|\leq C\left(n,\sigma,\kappa\right),\label{II: derivatives of lower order terms}
\end{equation}
 
\[
\sqrt{t_{0}}^{2}\left|\partial_{z}b^{k}\left(x,t,z,p\right)\right|+\sqrt{t_{0}}\left|\partial_{p}b^{k}\left(x,t,z,p\right)\right|\leq C\left(n,\sigma,\kappa\right),
\]
 
\[
\sqrt{t_{0}}^{2}\left|c\left(x,t,z,p\right)\right|+\sqrt{t_{0}}^{3}\left|\partial_{x}c\left(x,t,z,p\right)\right|+\sqrt{t_{0}}^{4}\left|\partial_{t}c\left(x,t,z,p\right)\right|\leq C\left(n,\sigma,\kappa\right),
\]
 
\[
\sqrt{t_{0}}^{3}\left|\partial_{z}c\left(x,t,z,p\right)\right|+\sqrt{t_{0}}^{2}\left|\partial_{p}c\left(x,t,z,p\right)\right|\leq C\left(n,\sigma,\kappa\right),
\]
 
\[
\sqrt{t_{0}}\left|\Gamma_{ij}^{k}\left(x,t\right)\right|+\sqrt{t_{0}}^{2}\left|\partial_{x}\Gamma_{ij}^{k}\left(x,t\right)\right|+\sqrt{t_{0}}^{3}\left|\partial_{t}\Gamma_{ij}^{k}\left(x,t\right)\right|\leq C\left(n,\sigma,\kappa\right).
\]
for $x\in B_{\theta\sqrt{t_{0}}}^{n}\left(O\right)$, $t\in\left(\left(1-\theta^{2}\right)t_{0},t_{0}\right],$
$\left(z,p\right)\in\mathbb{R}\times\mathbb{R}^{n}$ satisfying 
\[
\left|p\right|+\frac{\kappa}{\sqrt{1-\theta^{2}}}\frac{\left|z\right|}{\sqrt{t_{0}}}\leq\varsigma.
\]
Now consider the change of variables (\ref{II: change of variables})
once again. Using conditions (\ref{II: C1 bound}), (\ref{II: parabolicity}),
(\ref{II: derivatives of leading coefficients}), (\ref{II: Holder gradient estimate}),
(\ref{II: derivatives of lower order terms}) and Eq. (\ref{II: homogeneous linear equation}),
it follows from Schauder estimates for linear parabolic equations
(cf. Chapter 4 in \cite{L}) that 
\[
\left|\partial_{x}v\left(x,t\right)\right|\,+\,\sqrt{t_{0}}\left|\partial_{x}^{2}v\left(x,t\right)\right|\,\leq\frac{C\left(n,\sigma,\kappa,\delta\right)}{\sqrt{t_{0}}}\sup_{\left(1-\frac{\theta^{2}}{4}\right)t_{0}\leq\tau\leq t_{0}}\left\Vert v\left(\cdot,\tau\right)\right\Vert _{L^{\infty}\left(B_{\frac{\theta}{2}\sqrt{t_{0}}}^{n}\left(O\right)\right)}
\]
for $x\in B_{\frac{\theta}{3}\sqrt{t_{0}}}^{n}\left(O\right)$, $t\in\left(\left(1-\frac{\theta^{2}}{9}\right)t_{0},t_{0}\right]$.
In view of Proposition \ref{analyticity: normal parametrization}
and the fact that 
\[
\nabla_{ij}v=\partial_{ij}v-\Gamma_{ij}^{k}\partial_{k}v,
\]
by setting $x=O$ and $t=t_{0}$ we obtain 
\[
\left|\nabla_{\Sigma_{t_{0}}}v\right|\left(P\right)\,+\,\sqrt{t_{0}}\left|\nabla_{\Sigma_{t_{0}}}^{2}v\right|\left(P\right)\,\leq\frac{C\left(n,\sigma,\kappa,\delta\right)}{\sqrt{t_{0}}}\sup_{\left(1-\frac{\theta^{2}}{4}\right)t_{0}\leq\tau\leq t_{0}}\left\Vert v\left(\cdot,\tau\right)\right\Vert _{L^{\infty}\left(B_{\frac{\theta}{2}\sqrt{t_{0}}}^{n}\left(O\right)\right)}
\]
 
\[
\leq C\left(n,\sigma,\kappa,\delta\right)\sup_{\frac{t_{0}}{1+\frac{3}{2}\sigma}\leq\tau\leq t_{0}}\frac{\left\Vert v\left(\cdot,\tau\right)\right\Vert _{L^{\infty}\left(\Sigma_{\tau}\right)}}{\sqrt{\tau}},
\]
provided that $\theta\ll1$ so that $1-\frac{\theta^{2}}{4}\geq\left(1+\frac{3}{2}\sigma\right)^{-1}$.
More generally, by a bootstrap argument (in the use of Schauder estimates)
and Proposition \ref{analyticity: normal parametrization}, for every
$k\in\mathbb{N}$ we have
\[
\sqrt{t}_{0}^{k-1}\left|\nabla_{\Sigma_{t_{0}}}^{k}v\right|\left(P\right)\leq C\left(n,\sigma,\kappa,\delta,k\right)\sup_{\frac{t_{0}}{1+2\sigma}\leq\tau\leq t_{0}}\frac{\left\Vert v\right\Vert _{L^{\infty}\left(\Sigma_{\tau}\right)}}{\sqrt{\tau}}.
\]
\end{proof}
In the following two theorems we intend to address the question of
the long time stability. By combining the $L^{\infty}$ estimate in
Proposition \ref{evolution of deviation} with the smooth estimate
in Proposition \ref{interpolation inequality}, we show that under
condition (\ref{AS: growth rate of curvature}), the deviation may
grow but quite slowly compared to the rate $\sqrt{t}$ (see (\ref{AS: deviation})
for $k=0$). In this way, the deviation of the corresponding NMCFs
will converge to zero (see (\ref{AS: rescaled deviation})).
\begin{thm}
\label{asymptotic stability}Given $0<\sigma<1$, let $\left\{ \Sigma_{t}\right\} _{t_{0}\leq t<\infty}$
be a MCF in $\mathbb{R}^{n+1}$, where $t_{0}\geq0$ is a constant,
so that $\Sigma_{t_{0}}$ is asymptotic to a regular cone $\mathcal{C}$
at infinity and there holds 
\begin{equation}
\sup_{t\geq T}\sqrt{t}\left\Vert A_{\Sigma_{t}}\right\Vert _{L^{\infty}}\leq\kappa\label{AS: growth rate of curvature}
\end{equation}
for some constants $0<\kappa<\frac{1}{\sqrt{2}}$ and $T>0$ satisfying
$T\geq t_{0}$. Let
\begin{equation}
K=\sup_{t_{0}\leq t\leq\left(1+3\sigma\right)T}\left(\left\Vert A_{\Sigma_{t}}\right\Vert _{L^{\infty}}+\left\Vert \nabla_{\Sigma_{t}}A_{\Sigma_{t}}\right\Vert _{L^{\infty}}+\left\Vert \nabla_{\Sigma_{t}}^{2}A_{\Sigma_{t}}\right\Vert _{L^{\infty}}\right),\label{AS: curvature bound}
\end{equation}
which is finite by Proposition \ref{spatial curvature estimate} and
Corollary \ref{spatial smooth estimates}. Then there exists a constant
$\delta>0$ depending on $n,$ $\sigma$, $\kappa$, $K$, and $T$
with the following property.

If $\left\{ \tilde{\Sigma}_{t}\right\} _{t_{0}\leq t\leq\left(1+3\sigma\right)T}$
is a MCF in $\mathbb{R}^{n+1}$ so that $\tilde{\Sigma}_{t_{0}}$
is asymptotic to $\mathcal{C}$ at infinity and $\tilde{\Sigma}_{t_{0}}$
is a normal graph of $v_{t_{0}}$ over $\Sigma_{t_{0}}$ with
\[
\left\Vert \nabla_{\Sigma_{t_{0}}}v_{t_{0}}\right\Vert _{L^{\infty}}+\left\Vert v_{t_{0}}\right\Vert _{L^{\infty}}\leq\delta.
\]
Then the MCF $\left\{ \tilde{\Sigma}_{t}\right\} _{t_{0}\leq t\leq\left(1+3\sigma\right)T}$
can be extended to time infinity; moreover, for every $t>t_{0}$,
$\tilde{\Sigma}_{t}$ is a normal graph of $v_{t}$ over $\Sigma_{t}$
with 
\begin{equation}
\sqrt{t}^{k-1}\left\Vert \nabla_{\Sigma_{t}}^{k}v\right\Vert _{L^{\infty}}\leq C\left(n,\sigma,k\right)\left(\frac{t}{\left(1+2\sigma\right)\left(1+\sigma\right)T}\right)^{-\frac{1}{2}\left(\frac{1}{2}-\kappa^{2}\right)}\label{AS: deviation}
\end{equation}
for $k\in\mathbb{N}\cup\left\{ 0\right\} $ and $t\geq\left(1+2\sigma\right)T$. 

Consequently, for every $t>t_{0}$, $\frac{1}{\sqrt{t}}\tilde{\Sigma}_{t}$
is a normal graph of 
\[
w_{t}\left(Y\right)=\frac{1}{\sqrt{t}}\,v\left(\sqrt{t}\,Y\right)
\]
over $\frac{1}{\sqrt{t}}\Sigma_{t}$ with 
\begin{equation}
\left\Vert \nabla_{\frac{1}{\sqrt{t}}\Sigma_{t}}^{k}w_{t}\right\Vert _{L^{\infty}}\leq C\left(n,\sigma,k\right)\left(\frac{t}{\left(1+2\sigma\right)\left(1+\sigma\right)T}\right)^{-\frac{1}{2}\left(\frac{1}{2}-\kappa^{2}\right)}\label{AS: rescaled deviation}
\end{equation}
for $k\in\mathbb{N}\cup\left\{ 0\right\} $ and $t\geq\left(1+2\sigma\right)T$.
In particular, $w_{t}\overset{C^{\infty}}{\longrightarrow}0$ as $t\rightarrow\infty$.
\end{thm}

\begin{proof}
Let $\left\{ \tilde{\Sigma}_{t}\right\} $ be as stated in the theorem
with $0<\delta\ll1$ to be specified. Let $0<\varepsilon\ll1$ be
a constant to be determined (which will depend only on $n,\sigma,\kappa,T$).
By Theorem \ref{stability}, if $\delta\ll1$ (depending on $n$,
$K$, $T$, $\varepsilon$), then for every $T\leq t\leq\left(1+3\sigma\right)T$,
$\tilde{\Sigma}_{t}$ is a normal graph of $v_{t}$ over $\Sigma_{t}$
with
\[
\sup_{T\leq t\leq\left(1+3\sigma\right)T}\left(\left\Vert \nabla_{\Sigma_{t}}v_{t}\right\Vert _{L^{\infty}}+\kappa\,\frac{\left\Vert v_{t}\right\Vert _{L^{\infty}}}{\sqrt{t}}\right)\leq\frac{\varsigma}{3},
\]
\begin{equation}
\sup_{T\leq t\leq\left(1+3\sigma\right)T}\frac{\left\Vert v_{t}\right\Vert _{L^{\infty}}}{\sqrt{t}}\leq\varepsilon,\label{AS: initially small growth of deviation}
\end{equation}
where $\varsigma$ is the same constant as in Proposition \ref{evolution of deviation}.
Now let $S\in\left[\left(1+3\sigma\right)T,\infty\right]$ be the
supremum of times before which the MCF $\left\{ \tilde{\Sigma}_{t}\right\} $
can be extended and every time-slice $\tilde{\Sigma}_{t}$ is a normal
graph of $v_{t}$ over $\Sigma_{t}$ with
\begin{equation}
\sup_{T\leq t\leq S}\left(\left\Vert \nabla_{\Sigma_{t}}v_{t}\right\Vert _{L^{\infty}}+\kappa\,\frac{\left\Vert v_{t}\right\Vert _{L^{\infty}}}{\sqrt{t}}\right)\leq\frac{\varsigma}{2},\label{AS: normal graphical condition}
\end{equation}
\begin{equation}
\sup_{T\leq t\leq S}\frac{\left\Vert v_{t}\right\Vert _{L^{\infty}}}{\sqrt{t}}\leq\sqrt{2T}\,\varepsilon.\label{AS: small growth of deviation}
\end{equation}
By Proposition \ref{evolution of deviation}, the function $v_{\textrm{max}}^{2}\left(t\right)=\left\Vert v_{t}\right\Vert _{L^{\infty}}^{2}$
satisfies Eq. (\ref{ED: ODE inequality}), i.e.
\[
D^{-}v_{\textrm{max}}^{2}\left(t\right)=\limsup_{h\searrow0}\frac{v_{\textrm{max}}^{2}\left(t\right)-v_{\textrm{max}}^{2}\left(t-h\right)}{h}
\]
\[
\leq\left\{ 2\left\Vert A_{\Sigma_{t}}\right\Vert _{L^{\infty}}^{2}+C\left(n\right)\left\Vert A_{\Sigma_{t}}v\right\Vert _{L^{\infty}}\left[\left\Vert A_{\Sigma_{t}}\right\Vert _{L^{\infty}}^{2}+\left\Vert \nabla_{\Sigma_{t}}A_{\Sigma_{t}}\right\Vert _{L^{\infty}}\right]\right\} v_{\textrm{max}}^{2}\left(t\right)
\]
for $T<t<S$. Note that by condition (\ref{AS: growth rate of curvature})
and Lemma \ref{smooth decay estimate } we have
\begin{equation}
\left\Vert A_{\Sigma_{t}}\right\Vert _{L^{\infty}}^{2}+\left\Vert \nabla_{\Sigma_{t}}A_{\Sigma_{t}}\right\Vert _{L^{\infty}}\leq\frac{C\left(n\right)}{t}\label{AS: residual}
\end{equation}
for $\left(1+\sigma\right)T\leq t<S$. Thus, using conditions (\ref{AS: growth rate of curvature}),
(\ref{AS: small growth of deviation}), and (\ref{AS: residual}),
we can deduce from Eq. (\ref{ED: ODE inequality}) that
\[
D^{-}v_{\textrm{max}}^{2}\left(t\right)\leq\left\{ 2\left\Vert A_{\Sigma_{t}}\right\Vert _{L^{\infty}}^{2}+\left\Vert A_{\Sigma_{t}}v\right\Vert _{L^{\infty}}\frac{C\left(n\right)}{t}\right\} v_{\textrm{max}}^{2}\left(t\right)
\]
\[
\leq\left(\frac{2\kappa^{2}+C\left(n,T\right)\varepsilon}{t}\right)v_{\textrm{max}}^{2}\left(t\right)
\]
for $\left(1+\sigma\right)T\leq t<S$. Using the comparison principle
for ODE (cf. Chapter 2 in \cite{Wa}) and condition (\ref{AS: initially small growth of deviation}),
we obtain 
\[
v_{\textrm{max}}^{2}\left(t\right)\leq v_{\textrm{max}}^{2}\left(\left(1+\sigma\right)T\right)\,\left(\frac{t}{\left(1+\sigma\right)T}\right)^{2\kappa^{2}+C\left(n,T\right)\varepsilon}\leq\left(1+\sigma\right)T\,\varepsilon^{2}\left(\frac{t}{\left(1+\sigma\right)T}\right)^{2\kappa^{2}+C\left(n,T\right)\varepsilon}
\]
for $\left(1+\sigma\right)T\leq t<S$, which implies 
\begin{equation}
\frac{\left\Vert v_{t}\right\Vert _{L^{\infty}}}{\sqrt{t}}\leq\sqrt{\left(1+\sigma\right)T}\,\varepsilon\,\left(\frac{t}{\left(1+\sigma\right)T}\right)^{\kappa^{2}+\frac{1}{2}C\left(n,T\right)\varepsilon-\frac{1}{2}}\label{AS: smaller growth of deviation}
\end{equation}
\[
\leq\sqrt{\left(1+\sigma\right)T}\,\varepsilon\,\left(\frac{t}{\left(1+\sigma\right)T}\right)^{-\frac{1}{2}\left(\frac{1}{2}-\kappa^{2}\right)}
\]
for $\left(1+\sigma\right)T\leq t<S$, provided that $\varepsilon\ll1$
so that $C\left(n,T\right)\varepsilon\leq\frac{1}{2}-\kappa^{2}$.
Then it follows from Proposition \ref{interpolation inequality} that
\begin{equation}
\sqrt{t}^{k-1}\left\Vert \nabla_{\Sigma_{t}}^{k}v\right\Vert _{L^{\infty}}\leq C\left(n,\sigma,k\right)\sup_{\frac{t}{1+2\sigma}\leq\tau\leq t}\frac{\left\Vert v_{\tau}\right\Vert _{L^{\infty}}}{\sqrt{\tau}}\label{AS: smaller derivatives}
\end{equation}
\[
\leq C\left(n,\sigma,k\right)\,\sqrt{\left(1+\sigma\right)T}\,\varepsilon\,\left(\frac{t}{\left(1+2\sigma\right)\left(1+\sigma\right)T}\right)^{-\frac{1}{2}\left(\frac{1}{2}-\kappa^{2}\right)}
\]
for $\left(1+2\sigma\right)T\leq t<S$. In view of (\ref{AS: smaller growth of deviation})
and (\ref{AS: smaller derivatives}), if $\varepsilon\ll1$ (depending
on $n,\sigma,\kappa,T$) we have
\[
\left\Vert \nabla_{\Sigma_{t}}v_{t}\right\Vert _{L^{\infty}}+\kappa\frac{\left\Vert v_{t}\right\Vert _{L^{\infty}}}{\sqrt{t}}\leq\frac{\varsigma}{3},\quad\frac{\left\Vert v_{t}\right\Vert _{L^{\infty}}}{\sqrt{t}}\leq\sqrt{\left(1+\sigma\right)T}\,\varepsilon,
\]
 
\begin{equation}
\sqrt{t}^{k-1}\left\Vert \nabla_{\Sigma_{t}}^{k}v\right\Vert _{L^{\infty}}\leq C\left(n,\sigma,k\right)\sup_{\frac{t}{1+2\sigma}\leq\tau\leq t}\frac{\left\Vert v\right\Vert _{L^{\infty}}}{\sqrt{\tau}}\label{AS: decay rate of derivatives}
\end{equation}
\[
\leq C\left(n,\sigma,k\right)\left(\frac{t}{\left(1+2\sigma\right)\left(1+\sigma\right)T}\right)^{-\frac{1}{2}\left(\frac{1}{2}-\kappa^{2}\right)}
\]
for all $k\in\mathbb{N}\cup\left\{ 0\right\} $, $\left(1+2\sigma\right)T\leq t<S$. 

We claim that $S=\infty$; from this and (\ref{AS: decay rate of derivatives})
the theorem follows. For if not, then by the short time existence
theorem for MCF (cf. Section 4 in \cite{EH2}) and the argument in
proving Theorem \ref{stability}, we would be able to find a constant
$S'>S$ so that the MCF $\left\{ \tilde{\Sigma}_{t}\right\} $ can
be extended to time $S'$ and every time-slice $\tilde{\Sigma}_{t}$
is a normal graph of $v_{t}$ over $\Sigma_{t}$ with (\ref{AS: normal graphical condition})
and (\ref{AS: small growth of deviation}) holding up to time $S'$.
This would contradict the maximality of $S$. Therefore, $S=\infty$.
\end{proof}
The following theorem is a direct consequence of Theorem \ref{asymptotic stability}.
One of its applications is to study the stability of an expander $\Gamma$
as an equilibrium solution of NMCF by taking $\Sigma_{t}=\sqrt{t}\,\Gamma$,
in which $\Gamma$ is assumed to be asymptotic to a cone and satisfies
$\left\Vert A_{\Gamma}\right\Vert _{L^{\infty}}<\frac{1}{\sqrt{2}}$.
\begin{thm}
\label{approaching of NMCF}Let $\left\{ \Sigma_{t}\right\} _{1\leq t<\infty}$
be a MCF in $\mathbb{R}^{n+1}$ so that $\Sigma_{1}$ is asymptotic
to a regular cone $\mathcal{C}$ at infinity and that
\[
\sup_{t\geq1}\sqrt{t}\left\Vert A_{\Sigma_{t}}\right\Vert _{L^{\infty}}\leq\kappa
\]
for some constant $0<\kappa<\frac{1}{\sqrt{2}}$. Let $\tilde{\Sigma}_{1}$
be a smooth, properly embedded, and oriented hypersurface in $\mathbb{R}^{n+1}$
that is asymptotic to $\mathcal{C}$ at infinity. Note that 
\begin{equation}
\sup_{1\leq t\leq2}\left(\left\Vert A_{\Sigma_{t}}\right\Vert _{L^{\infty}}+\left\Vert \nabla_{\Sigma_{t}}A_{\Sigma_{t}}\right\Vert _{L^{\infty}}+\left\Vert \nabla_{\Sigma_{t}}^{2}A_{\Sigma_{t}}\right\Vert _{L^{\infty}}\right)\leq K,\quad\left\Vert A_{\tilde{\Sigma}_{1}}\right\Vert _{L^{\infty}}\leq\tilde{K}\label{AN: curvature bound}
\end{equation}
for some constants $K,\tilde{K}>0$ by Proposition \ref{spatial curvature estimate},
Corollary \ref{spatial smooth estimates} and the asymptotic condition.

Then there exists a constant $\delta>0$ depending on $n$, $\kappa$,
$K$, and $\tilde{K}$ so that if $\tilde{\Sigma}_{1}$ is a normal
graph of $v$ over $\Sigma_{1}$ with
\[
\left\Vert \nabla_{\Sigma_{1}}v\right\Vert _{L^{\infty}}+\left\Vert v\right\Vert _{L^{\infty}}\leq\delta,
\]
the corresponding MCF $\left\{ \tilde{\Sigma}_{t}\right\} _{t\geq1}$
has long time existence. Moreover, for every $t>1$, $\frac{1}{\sqrt{t}}\tilde{\Sigma}_{t}$
is a normal graph of $w_{t}$ over $\frac{1}{\sqrt{t}}\Sigma_{t}$
with $w_{t}\overset{C^{\infty}}{\longrightarrow}0$ as $t\rightarrow\infty$.
In other words, the two NMCFs starting out from $\tilde{\Sigma}_{1}$
and $\Sigma_{1}$, respectively, will approach as $t\rightarrow\infty$.
\end{thm}

\begin{proof}
First of all, by the short time existence theorem for MCF (cf. Section
4 in \cite{EH2}) and Corollary \ref{uniqueness}, there exist a constant
$0<\sigma\leq\frac{1}{3}$ (depending on $n$, $\tilde{K}$) and a
unique MCF $\left\{ \tilde{\Sigma}_{t}\right\} _{1\leq t\leq1+3\sigma}$
that starts out from $\tilde{\Sigma}_{1}$. 

By Theorem \ref{asymptotic stability} (substituting $t_{0}=T=1$),
if $\tilde{\Sigma}_{1}$ is a normal graph of $v$ over $\Sigma_{1}$
with
\[
\left\Vert \nabla_{\Sigma_{1}}v\right\Vert _{L^{\infty}}+\left\Vert v\right\Vert _{L^{\infty}}\leq\delta,
\]
where $0<\delta\ll1$ (depending on $n,$ $\sigma$, $\kappa$, $K$),
then the MCF $\left\{ \tilde{\Sigma}_{t}\right\} $ can be extended
to time infinity. Furthermore, for every $t>1$, $\frac{1}{\sqrt{t}}\tilde{\Sigma}_{t}$
is a normal graph of $w_{t}$ over $\frac{1}{\sqrt{t}}\Sigma_{t}$
with $w_{t}\overset{C^{\infty}}{\longrightarrow}0$ as $t\rightarrow\infty$. 
\end{proof}

\section{Asymptotic Self-Similarity\label{Asymptotic self-similarity}}

Our goal in this section is to study the asymptotic behavior of an
immortal MCF with certain properties. More precisely, if the flow
has a conical end and satisfies the curvature condition (\ref{BDM: curvature condition}),
it will become asymptotically self-expanding (see Theorem \ref{blow-down of MCF}).
On the other hand, as stated by Theorem \ref{decay rate of curvature},
condition (\ref{BDM: curvature condition}) holds provided the entropy
of the tangent cone is small. Thereby we obtain Theorem \ref{asymptotic behavior of MCF}
in virtue of Theorems \ref{decay rate of curvature} and \ref{blow-down of MCF}.
Likewise, Theorems \ref{decay rate of curvature} and \ref{approaching of NMCF}
combined bring us Theorem \ref{stability of expanders}.

As a first step in studying the asymptotic behavior of MCF at time
infinity, we zoom out the flow by the parabolic rescaling. The following
proposition guarantees that such a procedure always yields smooth
limits. 
\begin{prop}
\label{sequential blow-down}Let $\left\{ \Sigma_{t}\right\} _{0\leq t<\infty}$
be a MCF in $\mathbb{R}^{n+1}$ so that $\Sigma_{0}$ is asymptotic
to a regular cone $\mathcal{C}$ at infinity with $E\left[\mathcal{C}\right]<\infty$,
and that 
\[
\sup_{t\geq T}\sqrt{t}\left\Vert A_{\Sigma_{t}}\right\Vert _{L^{\infty}}\leq\kappa
\]
for some constants $\kappa>0$ and $T>0$. Then given a sequence of
numbers $\left\{ R_{i}\nearrow\infty\right\} $, the corresponding
zooming out sequence of MCF
\[
\left\{ \Sigma_{\tau}^{R_{i}}=\frac{1}{R_{i}}\Sigma_{R_{i}^{2}\tau}\right\} _{0\leq\tau<\infty}
\]
has a subsequence that converges locally smoothly to a MCF $\left\{ \Gamma_{\tau}\right\} _{0\leq\tau<\infty}$
in the space-time $\left(\mathbb{R}^{n+1}\times\left[0,\infty\right)\right)\setminus\left\{ \left(O,0\right)\right\} $.
In addition, we have 
\[
\sup_{\tau>0}\sqrt{\tau}\left\Vert A_{\Gamma_{\tau}}\right\Vert _{L^{\infty}}\leq\kappa.
\]
\end{prop}

\begin{proof}
Without loss of generality, we may assume that $T\gg1$ (depending
on $n$, $\mathcal{C}$, $\Sigma_{0}$). By Proposition \ref{spatial curvature estimate},
Corollary \ref{spatial smooth estimates}, and Lemma \ref{Gaussian on the large scale},
for every $k\in\mathbb{N}\cup\left\{ 0\right\} $ we have
\[
\sup_{t\geq T}F_{O,t}\left(\Sigma_{0}\right)\leq2E\left[\mathcal{C}\right],\quad\sup_{t\geq T}\sqrt{t}\left\Vert A_{\Sigma_{t}}\right\Vert _{L^{\infty}}\leq\kappa,
\]
\[
\sup_{t\geq0}\left\Vert \,\left|X\right|^{k+1}\nabla_{\Sigma_{t}}^{k}A_{\Sigma_{t}}\right\Vert _{L^{\infty}\left(\mathbb{R}^{n+1}\setminus B_{\max\left\{ \Lambda_{k}\sqrt{T},\,\Lambda_{k}\sqrt{t}\right\} }\left(O\right)\right)}\leq C\left(n,k,\mathcal{C},\Sigma_{0}\right),
\]
where $\Lambda_{k}\geq1$ is a constant (depending on $n$, $\mathcal{C}$,
$\Sigma_{0}$). Note that by Huisken's monotonicity formula for MCF
(cf. \cite{H}) we have
\[
F_{O,1}\left(\Sigma_{\tau}^{R_{i}}\right)\leq F_{O,1+\tau}\left(\Sigma_{0}^{R_{i}}\right)=F_{O,1+\tau}\left(\frac{1}{R_{i}}\Sigma_{0}\right)=F_{O,R_{i}^{2}\left(1+\tau\right)}\left(\Sigma_{0}\right).
\]
Accordingly, as long as $R_{i}\geq\sqrt{T}$ , for every $k\in\mathbb{N}\cup\left\{ 0\right\} $
we have 

\begin{equation}
\sup_{\tau\geq0}F_{O,1}\left(\Sigma_{\tau}^{R_{i}}\right)\leq2E\left[\mathcal{C}\right],\quad\sup_{\tau\geq T_{i}}\sqrt{\tau}\left\Vert A_{\Sigma_{\tau}^{R_{i}}}\right\Vert _{L^{\infty}}\leq\kappa,\label{SBD: decay of curvature}
\end{equation}
\[
\sup_{\tau\geq0}\left\Vert \,\left|X\right|^{k+1}\nabla_{\Sigma_{\tau}^{R_{i}}}^{k}A_{\Sigma_{\tau}^{R_{i}}}\right\Vert _{L^{\infty}\left(\mathbb{R}^{n+1}\setminus B_{\max\left\{ \Lambda_{k}\sqrt{T_{i}},\,\Lambda_{k}\sqrt{\tau}\right\} }\left(O\right)\right)}\leq C\left(n,k,\mathcal{C},\Sigma_{0}\right),
\]
where $T_{i}=\frac{T}{R_{i}^{2}}\searrow0$. Then the proposition
follows from the smooth compactness theorem for MCF and passing (\ref{SBD: decay of curvature})
to the limit.
\end{proof}
A crucial step in proving Theorem \ref{blow-down of MCF} is the following
proposition, in which we characterize the limiting flow arising from
the preceding proposition under an extra condition that the constant
$\kappa$ therein is no more than $\frac{1}{\sqrt{2}}.$ The proof
is based primarily on Theorem \ref{asymptotic stability}. 
\begin{prop}
\label{asymptotically self-expanding}Let $\left\{ \Sigma_{t}\right\} _{0\leq t<\infty}$
be a MCF in $\mathbb{R}^{n+1}$ so that $\Sigma_{0}$ is asymptotic
to a regular cone $\mathcal{C}$ at infinity, and that 
\[
\sup_{t\geq T}\sqrt{t}\left\Vert A_{\Sigma_{t}}\right\Vert _{L^{\infty}}\leq\kappa
\]
for some constant $0<\kappa<\frac{1}{\sqrt{2}}$ and $T>0$. Suppose
that there is a sequence of numbers $\left\{ R_{i}\nearrow\infty\right\} $
so that the corresponding sequence of MCF
\[
\left\{ \Sigma_{\tau}^{R_{i}}=\frac{1}{R_{i}}\Sigma_{R_{i}^{2}\tau}\right\} _{0\leq\tau<\infty}
\]
converges locally smoothly to a MCF $\left\{ \Gamma_{\tau}\right\} _{0\leq\tau<\infty}$
in the space-time $\left(\mathbb{R}^{n+1}\times\left[0,\infty\right)\right)\setminus\left\{ \left(O,0\right)\right\} $.
Then the limiting flow $\left\{ \Gamma_{\tau}\right\} $ is self-expanding,
i.e. $\Gamma_{\tau}=\sqrt{\tau}\,\Gamma_{1}$ for $\tau>0$. Furthermore,
$\Gamma_{0}=\mathcal{C}$ and $\Gamma_{\tau}$ is asymptotic to $\mathcal{C}$
at infinity for every $\tau>0$.
\end{prop}

\begin{proof}
First of all, given $\iota>0$, let us define
\[
\Sigma_{t}^{\iota}=\frac{1}{\iota}\,\Sigma_{\iota^{2}t},\quad\Gamma_{t}^{\iota}=\frac{1}{\iota}\,\Gamma_{\iota^{2}t}
\]
 for $t\geq0$. Note that
\begin{equation}
\frac{1}{R_{i}}\Sigma_{R_{i}^{2}\tau}^{\iota}=\frac{1}{\iota R_{i}}\Sigma_{\iota^{2}R_{i}^{2}\tau}=\frac{1}{\iota}\Sigma_{\iota^{2}\tau}^{R_{i}}\,\overset{C_{loc}^{\infty}}{\longrightarrow}\,\frac{1}{\iota}\,\Gamma_{\iota^{2}\tau}=\Gamma_{\tau}^{\iota}\quad\textrm{in}\,\,\,\mathbb{R}^{n+1}\;\,\textrm{as}\,\,i\rightarrow\infty\label{ASE: convergence}
\end{equation}
for every $\tau>0$. Note also that $\Sigma_{0}^{\iota}=\frac{1}{\iota}\,\Sigma_{0}$
is asymptotic to $\mathcal{C}$ at infinity.

Let $\delta>0$ be the constant in Theorem \ref{asymptotic stability},
with the choices $\sigma=\frac{1}{3}$ and $t_{0}=0$. Note that the
constant $\delta$ depends on $n$, $\kappa$, $T$, and $K$ (as
defined in (\ref{AS: curvature bound})). In view of the second condition
in Definition \ref{asymptotically conical}, we can find $0<\varepsilon<1$
(depending on $n$, $\delta$, $\mathcal{C}$, $\Sigma_{0}$) so that
for every $\left(1+\varepsilon\right)^{-1}\leq\iota\leq1+\varepsilon$,
$\Sigma_{0}^{\iota}=\frac{1}{\iota}\,\Sigma_{0}$ is a normal graph
of $v^{\iota}$ over $\Sigma_{0}$ with 
\begin{equation}
\left\Vert \nabla_{\Sigma_{0}}v^{\iota}\right\Vert _{L^{\infty}}+\left\Vert v^{\iota}\right\Vert _{L^{\infty}}\leq\delta.\label{ASE: initial closeness}
\end{equation}
Given $\tau>0$, set $t_{i}=R_{i}^{2}\tau$. By condition (\ref{ASE: convergence}),
for every $\left(1+\varepsilon\right)^{-1}\leq\iota\leq1+\varepsilon$
we have 
\[
\frac{\sqrt{\tau}}{\sqrt{t_{i}}}\Sigma_{t_{i}}^{\iota}=\frac{1}{R_{i}}\Sigma_{R_{i}^{2}\tau}^{\iota}\,\overset{C_{loc}^{\infty}}{\longrightarrow}\,\Gamma_{\tau}^{\iota}\quad\textrm{in}\,\,\,\mathbb{R}^{n+1}\;\,\textrm{as}\,\,i\rightarrow\infty,
\]
which implies 
\begin{equation}
\frac{1}{\sqrt{t_{i}}}\Sigma_{t_{i}}^{\iota}\overset{C_{loc}^{\infty}}{\longrightarrow}\frac{1}{\sqrt{\tau}}\Gamma_{\tau}^{\iota}\quad\textrm{in}\,\,\,\mathbb{R}^{n+1}\;\,\textrm{as}\,\,i\rightarrow\infty.\label{ASE: biased normalization}
\end{equation}
In particular, substituting $\iota=1$ in (\ref{ASE: biased normalization})
gives
\begin{equation}
\frac{1}{\sqrt{t_{i}}}\Sigma_{t_{i}}\,\overset{C_{loc}^{\infty}}{\longrightarrow}\,\frac{1}{\sqrt{\tau}}\Gamma_{\tau}\quad\textrm{in}\,\,\,\mathbb{R}^{n+1}\;\,\textrm{as}\,\,i\rightarrow\infty.\label{ASE: normalization}
\end{equation}
On the other hand, by Theorem \ref{asymptotic stability} and condition
(\ref{ASE: initial closeness}), for every $\left(1+\varepsilon\right)^{-1}\leq\iota\leq1+\varepsilon$,
$\frac{1}{\sqrt{t}}\Sigma_{t}^{\iota}$ is a normal graph of a function
$w_{t}^{\iota}$ over $\frac{1}{\sqrt{t}}\Sigma_{t}$ with $w_{t}^{\iota}\overset{C^{\infty}}{\longrightarrow}0$
as $t\rightarrow\infty$. In view of conditions (\ref{ASE: biased normalization})
and (\ref{ASE: normalization}), we then conclude that $\frac{1}{\sqrt{\tau}}\Gamma_{\tau}^{\iota}=\frac{1}{\sqrt{\tau}}\Gamma_{\tau}$,
or equivalently, $\Gamma_{\tau}^{\iota}=\Gamma_{\tau}$. That is to
say,
\begin{equation}
\frac{1}{\iota}\,\Gamma_{\iota^{2}\tau}=\Gamma_{\tau}\label{ASE: iteration}
\end{equation}
for every $\left(1+\varepsilon\right)^{-1}\leq\iota\leq1+\varepsilon$
and $\tau>0$.

To prove the self-similarity of $\left\{ \Gamma_{\tau}\right\} $,
we will iterate (\ref{ASE: iteration}) for various choices of $\tau>0$.
Let us begin with substituting $\tau=1$ in (\ref{ASE: iteration})
and get 
\[
\frac{1}{\sqrt{\tau}}\,\Gamma_{\tau}=\Gamma_{1}\quad\,\,\forall\,\,\left(1+\varepsilon\right)^{-2}\leq\tau\leq\left(1+\varepsilon\right)^{2}.
\]
Then proceed by substituting $\tau=\left(1+\varepsilon\right)^{2}$
and $\tau=\left(1+\varepsilon\right)^{-2}$, respectively, in (\ref{ASE: iteration})
to get
\[
\frac{1}{\sqrt{\tau}}\,\Gamma_{\tau}=\frac{1}{1+\varepsilon}\Gamma_{\left(1+\varepsilon\right)^{2}}=\Gamma_{1}\quad\,\,\forall\,\,1\leq\tau\leq\left(1+\varepsilon\right)^{4}
\]
and
\[
\frac{1}{\sqrt{\tau}}\,\Gamma_{\tau}=\frac{1}{\left(1+\varepsilon\right)^{-1}}\Gamma_{\left(1+\varepsilon\right)^{-2}}=\Gamma_{1}\quad\,\,\forall\,\,\left(1+\varepsilon\right)^{-4}\leq\tau\leq1.
\]
We continue this procedure, eventually obtaining
\[
\frac{1}{\sqrt{\tau}}\,\Gamma_{\tau}=\Gamma_{1}\quad\,\,\forall\,\,\left(1+\varepsilon\right)^{-2k}\leq\tau\leq\left(1+\varepsilon\right)^{2k},
\]
for every $k\in\mathbb{N}$, from which we conclude that $\left\{ \Gamma_{\tau}\right\} $
is self-expanding.

Lastly, note that by the first condition in Definition \ref{asymptotically conical}
we have
\[
\Sigma_{0}^{R_{i}}=\frac{1}{R_{i}}\Sigma_{0}\,\overset{C_{loc}^{\infty}}{\longrightarrow}\,\mathcal{C}\quad\textrm{in}\,\,\,\mathbb{R}^{n+1}\setminus\left\{ O\right\} \;\,\textrm{as}\,\,i\rightarrow\infty.
\]
Thus we have $\Gamma_{0}=\mathcal{C}$. It follows from Corollary
\ref{self-expanding MCF out of cone} that $\Gamma_{\tau}$ is asymptotic
to $\mathcal{C}$ at infinity for every $\tau>0$.
\end{proof}
In Propositions \ref{sequential blow-down} and \ref{asymptotically self-expanding},
we have shown that any zooming out sequence will have a subsequence
that converges to a self-expanding MCF (which may depend on the choice
of sequence). In the following theorem we prove that actually every
zooming out sequence will converge to the same limit. 
\begin{thm}
\label{blow-down of MCF}Let $\left\{ \Sigma_{t}\right\} _{0\leq t<\infty}$
be a MCF in $\mathbb{R}^{n+1}$ so that $\Sigma_{0}$ is asymptotic
to a regular cone $\mathcal{C}$ at infinity with $E\left[\mathcal{C}\right]<\infty$,
and that
\begin{equation}
\sup_{t\geq T}\sqrt{t}\left\Vert A_{\Sigma_{t}}\right\Vert _{L^{\infty}}\leq\kappa\label{BDM: curvature condition}
\end{equation}
for some constants $0<\kappa<\frac{1}{\sqrt{2}}$ and $T>0$. Then
we have
\[
\left\{ \Sigma_{\tau}^{R}=\frac{1}{R}\Sigma_{R^{2}\tau}\right\} _{0\leq\tau<\infty}\overset{C_{loc}^{\infty}}{\longrightarrow}\,\left\{ \Gamma_{\tau}\right\} _{0\leq\tau<\infty}\quad\textrm{as}\,\,\,R\rightarrow\infty
\]
in the space-time $\left(\mathbb{R}^{n+1}\times\left[0,\infty\right)\right)\setminus\left\{ \left(O,0\right)\right\} $.
The limiting flow $\left\{ \Gamma_{\tau}\right\} $ is a self-expanding
MCF satisfying $\Gamma_{0}=\mathcal{C}$ and $\Gamma_{\tau}$ is asymptotic
to $\mathcal{C}$ at infinity for every $\tau>0$. Indeed, we have
\[
\frac{1}{\sqrt{t}}\Sigma_{t}\,\overset{C^{\infty}}{\longrightarrow}\,\Gamma_{1}\quad\textrm{in}\,\,\,\mathbb{R}^{n+1}\;\,\,\textrm{as}\,\,\,t\rightarrow\infty
\]
in the sense that there exists $\mathfrak{T}>0$ so that for every
$t>\mathfrak{T}$, $\frac{1}{\sqrt{t}}\Sigma_{t}$ is a normal graph
of $w_{t}$ over $\Gamma_{1}$ with $w_{t}\overset{C^{\infty}}{\longrightarrow}0$
as $t\rightarrow\infty$.
\end{thm}

\begin{proof}
Let $\left\{ R_{i}\nearrow\infty\right\} $ be an arbitrary sequence
of numbers. By Proposition \ref{sequential blow-down}, we may assume
that, after passing to a subsequence, 
\begin{equation}
\left\{ \Sigma_{\tau}^{R_{i}}\right\} _{0\leq\tau<\infty}\overset{C_{loc}^{\infty}}{\longrightarrow}\,\left\{ \Gamma_{\tau}\right\} _{0\leq\tau<\infty}\quad\textrm{in}\,\,\left(\mathbb{R}^{n+1}\times\left[0,\infty\right)\right)\setminus\left\{ \left(O,0\right)\right\} \;\,\,\textrm{as}\,\,\,i\rightarrow\infty\label{BDM: sequential convergence}
\end{equation}
with 
\begin{equation}
\sup_{\tau>0}\sqrt{\tau}\left\Vert A_{\Gamma_{\tau}}\right\Vert _{L^{\infty}}\leq\kappa.\label{BDM: decay of curvature}
\end{equation}
It follows from Proposition \ref{asymptotically self-expanding} that
the limiting flow $\left\{ \Gamma_{\tau}\right\} $ is self-expanding
with $\Gamma_{0}=\mathcal{C}$, and that $\Gamma_{\tau}$ is asymptotic
to $\mathcal{C}$ at infinity for every $\tau>0$. 

Let $0<\delta\ll1$ (depending on $n$, $\kappa$, $\Gamma_{1}$)
and $M\gg1$ (depending on $n$, $\mathcal{C}$, $\Sigma_{0}$, $\Gamma_{1}$,
$\delta$) be constants to be specified. By Proposition \ref{spatial curvature estimate},
for every $t\geq1$ we have
\begin{equation}
\left\Vert A_{\frac{1}{\sqrt{t}}\Sigma_{t}}\right\Vert _{L^{\infty}\left(\mathbb{R}^{n+1}\setminus B_{M}\left(O\right)\right)}=\sqrt{t}\left\Vert A_{\Sigma_{t}}\right\Vert _{L^{\infty}\left(\mathbb{R}^{n+1}\setminus B_{M\sqrt{t}}\left(O\right)\right)}\leq\delta,\label{BDM: curvature outside}
\end{equation}
provided that $M\gg1$ (depending on $n$, $\mathcal{C}$, $\Sigma_{0}$,
$\delta$). On the other hand, since both $\Sigma_{0}$ and $\Gamma_{1}$
are asymptotic to $\mathcal{C}$ at infinity, Remark \ref{ACM: graphic along MCF}
implies that for every $t\geq1$, $\frac{1}{\sqrt{t}}\Sigma_{t}\setminus B_{M}\left(O\right)$
is a normal graph of $\omega_{t}$ over $\Gamma_{1}$ with 
\begin{equation}
\left\Vert \nabla_{\Gamma_{1}}\omega_{t}\right\Vert _{L^{\infty}}+\left\Vert \omega_{t}\right\Vert _{L^{\infty}}\leq\delta,\label{BDM: small Lipschitz}
\end{equation}
provided that $\delta\ll1$ (depending on $n$) and $M\gg1$ (depending
on $n$, $\mathcal{C}$, $\Sigma_{0}$, $\Gamma_{1}$, $\delta$).
Furthermore, by (\ref{BDM: sequential convergence}) we have 
\begin{equation}
\frac{1}{R_{i}}\Sigma_{R_{i}^{2}}=\Sigma_{1}^{R_{i}}\,\overset{C^{\infty}}{\longrightarrow}\,\Gamma_{1}\quad\textrm{in}\,\,B_{2M}\left(O\right)\;\,\,\textrm{as}\,\,\,i\rightarrow\infty.\label{BDM: local convergence}
\end{equation}
Hence, by conditions (\ref{BDM: decay of curvature}), (\ref{BDM: curvature outside}),
(\ref{BDM: small Lipschitz}), and (\ref{BDM: local convergence}),
we can find $i_{0}\gg1$ so that $\Sigma_{1}^{R_{i_{0}}}=\frac{1}{R_{i_{0}}}\Sigma_{R_{i_{0}}^{2}}$
is a normal graph of $v$ over $\Gamma_{1}$ with 
\begin{equation}
\left\Vert A_{\Sigma_{1}^{R_{i_{0}}}}\right\Vert _{L^{\infty}}\leq1,\quad\left\Vert \nabla_{\Gamma_{1}}v\right\Vert _{L^{\infty}}+\left\Vert v\right\Vert _{L^{\infty}}\leq\delta.\label{BDM: small initial deviation}
\end{equation}
Note that the function $v$ agrees with $\omega_{R_{i_{0}}^{2}}$
in $\Gamma_{1}$ outside a compact subset. It follows from Theorem
\ref{approaching of NMCF} (using $\left\{ \Gamma_{\tau}\right\} $
for $\left\{ \Sigma_{t}\right\} $ and $\tilde{K}=1$) and conditions
(\ref{BDM: decay of curvature}) and (\ref{BDM: small initial deviation})
that, if $\delta\ll1$ (depending on $n$, $\kappa$, $\Gamma_{1}$),
then for every $\tau>1$, 
\[
\frac{1}{\sqrt{\tau}}\Sigma_{\tau}^{R_{i_{0}}}=\frac{1}{\sqrt{R_{i_{0}}^{2}\tau}}\Sigma_{R_{i_{0}}^{2}\tau}
\]
is a normal graph of $w_{\tau}$ over $\frac{1}{\sqrt{\tau}}\Gamma_{\tau}=\Gamma_{1}$
with $w_{\tau}\overset{C^{\infty}}{\longrightarrow}0$ as $\tau\rightarrow\infty$.
Therefore, we obtain
\begin{equation}
\frac{1}{\sqrt{t}}\Sigma_{t}\overset{C^{\infty}}{\longrightarrow}\Gamma_{1}\quad\textrm{in}\,\,\mathbb{R}^{n+1}\;\,\,\textrm{as}\,\,t\rightarrow\infty.\label{BDM: asymptotically self-expanding}
\end{equation}

In fact, condition (\ref{BDM: asymptotically self-expanding}) also
implies the uniqueness of expanders arising from the zooming out procedure.
To see that, let $\left\{ \tilde{R}_{i}\nearrow\infty\right\} $ be
any other sequence of numbers for which 
\[
\left\{ \Sigma_{\tau}^{\tilde{R}_{i}}\right\} _{0\leq\tau<\infty}\overset{C_{loc}^{\infty}}{\longrightarrow}\,\left\{ \tilde{\Gamma}_{\tau}\right\} _{0\leq\tau<\infty}\quad\textrm{in}\,\,\left(\mathbb{R}^{n+1}\times\left[0,\infty\right)\right)\setminus\left\{ \left(O,0\right)\right\} \;\,\,\textrm{as}\,\,\,i\rightarrow\infty.
\]
Note that $\left\{ \tilde{\Gamma}_{\tau}\right\} $ is self-expanding
by Proposition \ref{asymptotically self-expanding}. Given that 
\[
\frac{1}{\tilde{R}_{i}}\Sigma_{\tilde{R}_{i}^{2}}=\Sigma_{1}^{\tilde{R}_{i}}\,\overset{C_{loc}^{\infty}}{\longrightarrow}\,\tilde{\Gamma}_{1}\quad\textrm{in}\,\,\mathbb{R}^{n+1}\;\,\,\textrm{as}\,\,i\rightarrow\infty
\]
and condition (\ref{BDM: asymptotically self-expanding}), we infer
that $\tilde{\Gamma}_{1}=\Gamma_{1}$. The theorem then follows from
Propositions \ref{sequential blow-down} and \ref{asymptotically self-expanding},
the uniqueness of zooming out limit and (\ref{BDM: asymptotically self-expanding}).
\end{proof}
Finally, we can prove Theorems \ref{asymptotic behavior of MCF} and
\ref{stability of expanders} on the basis of Theorems \ref{decay rate of curvature},
\ref{approaching of NMCF}, and \ref{blow-down of MCF} as announced
in the beginning of the section.
\begin{proof}
\textit{of Theorem \ref{asymptotic behavior of MCF}}

Given $0<\kappa<\frac{1}{\sqrt{2}}$, let $\epsilon$ be the same
constant as in Theorem \ref{decay rate of curvature}. The theorem
follows immediately from Theorems \ref{decay rate of curvature} and
\ref{blow-down of MCF}.
\end{proof}
\begin{proof}
\textit{of Theorem \ref{stability of expanders}}

Let $\Gamma$ be as stated in the theorem and let $\left\{ \Gamma_{t}=\sqrt{t}\,\Gamma\right\} _{t>0}$
be the MCF induced by $\Gamma$. Through Theorem \ref{decay rate of curvature}
(using $\left\{ \Gamma_{t+1}\right\} $ for $\left\{ \Sigma_{t}\right\} $)
we obtain 
\[
\sup_{t\geq T}\,\sqrt{\frac{t}{t+1}}\left\Vert A_{\Gamma}\right\Vert _{L^{\infty}}=\sup_{t\geq T}\sqrt{t}\left\Vert A_{\Gamma_{t+1}}\right\Vert _{L^{\infty}}\leq\kappa
\]
for some $T>0$ (depending on $n$, $\kappa$, $\mathcal{C}$, $\Gamma$),
from which we infer that $\left\Vert A_{\Gamma}\right\Vert _{L^{\infty}}\leq\kappa$
and hence
\begin{equation}
\sqrt{t}\left\Vert A_{\Gamma_{t}}\right\Vert _{L^{\infty}}=\left\Vert A_{\Gamma}\right\Vert _{L^{\infty}}\leq\kappa\label{SE: curvature condition}
\end{equation}
for every $t>0$. It then follows from Proposition \ref{smooth estimates}
that 
\begin{equation}
\sup_{1\leq t\leq2}\left(\left\Vert A_{\Gamma_{t}}\right\Vert _{L^{\infty}}+\left\Vert \nabla_{\Gamma_{t}}A_{\Gamma_{t}}\right\Vert _{L^{\infty}}+\left\Vert \nabla_{\Gamma_{t}}^{2}A_{\Gamma_{t}}\right\Vert _{L^{\infty}}\right)\leq C\left(n\right)\coloneqq K.\label{SE: curvature and its derivatives}
\end{equation}

Given $\Lambda>0$, let $\tilde{\Sigma}$ be as stated in the theorem,
where $0<\delta\ll1$ is a constant to be determined (which will depend
only on $n$, $\kappa$, $\Lambda$). Note that $\tilde{\Sigma}$
is obviously asymptotic to $\mathcal{C}$ at infinity. Also, by Lemma
\ref{normal graph}, especially (\ref{NG: curvature estimate}), and
condition (\ref{SE: curvature and its derivatives}), we have
\begin{equation}
\left\Vert A_{\tilde{\Sigma}}\right\Vert _{L^{\infty}}\leq C\left(n,\Lambda\right)\coloneqq\tilde{K},\label{SE: curvature}
\end{equation}
provided that $\delta\ll1$ (depending on $n$). Then it follows from
Theorem \ref{approaching of NMCF} and conditions (\ref{SE: curvature condition}),
(\ref{SE: curvature and its derivatives}), and (\ref{SE: curvature})
that, if $\delta\ll1$ (depending on $n$, $\kappa$, $\Lambda$),
the MCF $\left\{ \tilde{\Sigma}_{t}\right\} _{t\geq1}$ that starts
out from $\tilde{\Sigma}$ can be extended to the time infinity; moreover,
for every $t>1$, $\frac{1}{\sqrt{t}}\tilde{\Sigma}_{t}$ is a normal
graph of $w_{t}$ over $\frac{1}{\sqrt{t}}\Sigma_{t}$ with $w_{t}\overset{C^{\infty}}{\longrightarrow}0$
as $t\rightarrow\infty$. 
\end{proof}

\vspace{0.25in}
\email{
\noindent Department of Mathematics, Indiana University - Rawles Hall\\
831 East 3rd St.
Bloomington, IN 47405\\\\
E-mail address: \textsf{siaoguo@iu.edu}
}

\end{document}